\documentclass{amsart}

\usepackage{amssymb}
\usepackage{amscd}
\usepackage{epsfig}
\usepackage{color}

\usepackage[setpagesize=false]{hyperref}

\usepackage{ascmac}
\usepackage[all]{xy}
\usepackage{time}

\usepackage[OT2,T1]{fontenc}
\DeclareSymbolFont{cyrletters}{OT2}{wncyr}{m}{n}
\DeclareMathSymbol{\Sha}{\mathalpha}{cyrletters}{"58}

\usepackage{stackengine}
\stackMath
\newcommand\undertilde[2][1]{%
 \def\useanchorwidth{T}%
  \ifnum#1>1%
    \stackunder[0pt]{\tenq[\numexpr#1-1\relax]{#2}}{\scriptscriptstyle\sim}%
  \else%
    \stackunder[1pt]{#2}{\scriptscriptstyle\sim}%
  \fi%
}

\DeclareMathAlphabet{\mathpzc}{OT1}{pzc}{m}{it}

\usepackage{yhmath}

\usepackage{mathrsfs}

\usepackage{ulem}

\title[Kashiwara-Vergne and dihedral Lie algebras]{Kashiwara-Vergne and dihedral\\
bigraded Lie algebras in mould theory}

\author{Hidekazu Furusho}
\author{Nao Komiyama}

\address{Graduate School of Mathematics, Nagoya University, 
Furo-cho, Chikusa-ku, Nagoya, 464-8602, Japan}
\email{furusho@math.nagoya-u.ac.jp}
\email{m15027u@math.nagoya-u.ac.jp}

\subjclass[2020]{17B05 (Primary), 11M32, 16S30, 17B40 (Secondary)}

\keywords{Kashiwara-Verge Lie algebra, dihedral Lie algebra, mould theory}

\date{February 21, 2022} 
\newtheorem{thm}{Theorem}[section]
\newtheorem{lem}[thm]{Lemma}
\newtheorem{cor}[thm]{Corollary}
\newtheorem{prop}[thm]{Proposition}

{\theoremstyle{definition} \newtheorem{rem}[thm]{Remark}}
{\theoremstyle{definition} \newtheorem{defn}[thm]{Definition}}

\theoremstyle{definition}

\newtheorem{nota}[thm]{Notation}
\newtheorem{eg}[thm]{Example}
{\theoremstyle{remark} }

\numberwithin{equation}{section}
\newcommand{\C}{\mathbb{C}}

\newcommand{\Q}{\mathbb{Q}}
\newcommand{\Z}{\mathbb{Z}}
\newcommand{\N}{\mathbb{N}}

\newcommand{\shuffle}{\scalebox{.8}{$\Sha$}}
\newcommand{\Sh}[3]{{\rm Sh}\binom{#1;#2}{#3}}

\newcommand{\vecx}{{\bf x}}
\newcommand{\vecy}{{\bf y}}
\newcommand{\varia}[2]{\left({}^{#1}_{#2}\right)}

\newcommand{\q}{\mathbf{q}}

\newcommand{\krv}{\mathfrak{krv}}
\newcommand{\lkrv}{\mathfrak{lkrv}}
\newcommand{\lkrvd}{\mathfrak{lkrvd}}

\newcommand{\tder}{\mathfrak{tder}}
\newcommand{\sder}{\mathfrak{sder}}
\newcommand{\mt}{\mathfrak{mt}}
\newcommand{\tr}{\mathrm{tr}}
\newcommand{\anti}{\mathrm{anti}}
\newcommand{\Fil}{\mathrm{Fil}}
\newcommand{\cp}{\mathtt{cy}}
\newcommand{\res}{\mathrm{res}}

\newcommand{\ma}{\mathrm{ma}}
\newcommand{\vimo}{\mathrm{vimo}}
\newcommand{\teru}{\mathrm{teru}}
\newcommand{\push}{\mathrm{push}}
\newcommand{\pus}{\mathrm{pus}}
\newcommand{\pusnu}{\mathrm{pusnu}}
\newcommand{\nega}{\mathrm{neg}}
\newcommand{\mantar}{\mathrm{mantar}}
\newcommand{\swap}{\mathrm{swap}}
\newcommand{\pspush}{\mathrm{sena}}

\newcommand{\ARI}{\mathrm{ARI}}
\newcommand{\ARID}{\mathrm{ARID}}

\newcommand{\ulflex}[2]{{#1}\rceil_{\scalebox{.5}{$#2$}}}
\newcommand{\urflex}[2]{{}_{\scalebox{.5}{$#1$}}\lceil{#2}}
\newcommand{\llflex}[2]{{#1}\rfloor_{\scalebox{.5}{$#2$}}}
\newcommand{\lrflex}[2]{{}_{\scalebox{.5}{$#1$}}\lfloor{#2}}
\newcommand{\arit}{\mathrm{arit}}
\newcommand{\ari}{\mathrm{ari}}

\newcommand{\aritu}{\mathrm{arit}_u}
\newcommand{\ariu}{\mathrm{ari}_u}
\newcommand{\preariu}{\mathrm{preari}_u}

\newcommand{\pol}{\mathrm{pol}}
\newcommand{\al}{\mathrm{al}}
\newcommand{\il}{\mathrm{il}}

\newcommand{\fin}{\mathrm{fin}}
\newcommand{\gr}{\mathrm{gr}}
\newcommand{\id}{\mathrm{id}}
\newcommand{\D}{\mathcal{D}}
\newcommand{\f}{\mathsf{f}}
\newcommand{\h}{\undertilde{t}}
\newcommand{\tswap}{\undertilde{u}}

\newcommand{\Li}{\mathrm{Li}}
\newcommand{\KZ}{\mathrm{KZ}}
\newcommand{\corr}{\mathrm{corr}}

\newcommand{\Zag}{\mathrm{Zag}}
\newcommand{\Zig}{\mathrm{Zig}}
\newcommand{\Mini}{\mathrm{Mini}}
\newcommand{\Mono}{\mathrm{Mono}}
\newcommand{\reg}{\mathrm{reg}}

\newcommand{\GARI}{\mathrm{GARI}}
\newcommand{\as}{\mathrm{as}}
\newcommand{\is}{\mathrm{is}}

\begin{document}
\bibliographystyle{amsalpha+}
\maketitle
\begin{abstract}
We introduce the Kashiwara-Vergne bigraded Lie algebra associated with a finite abelian group and give its mould theoretic reformulation.
By using the mould theory, we show that 
it includes Goncharov's dihedral Lie algebra,
which generalizes the result of Raphael and Schneps. 
\end{abstract}
\tableofcontents
\setcounter{section}{-1}
\section{Introduction}\label{introduction}

The dihedral Lie algebra  $D(\Gamma)_{\bullet\bullet}$ is the bigraded Lie algebra introduced in \cite{G}
which is associated with a finite abelian group $\Gamma$.
It reflects the double shuffle and distribution relations among multiple polylogarithms evaluated at roots of unity.
Its relation with a certain bigraded variant of  motivic Lie algebra is discussed in ~loc.~cit.

Kashiwara-Vergne Lie algebra $\krv_\bullet$ is the filtered graded Lie algebra 
introduced in  \cite{AT} and \cite{AET}. It acts on the set of solutions of 
\lq a formal version' of Kashiwara-Vergne conjecture.
Related to conjectures on mixed Tate motives,  it is 
expected to be isomorphic to the motivic Lie algebra (cf. \cite{F}).


A bigraded variant $\lkrv_{\bullet\bullet}$ of $\krv_\bullet$ 
is introduced in \cite{RS}
where they give its interpretation  in terms of Ecalle's mould theory
(\cite{E81, E-ARIGARI, E-flex}).
The results in \cite{M, RS} give  an inclusion of bigraded Lie algebras  
\begin{equation}\label{D to lkrv}
D(\{e\})_{\bullet\bullet}\hookrightarrow\lkrv_{\bullet\bullet}.
\end{equation}
Our objective of  this paper is to 
extend it 
to any $\Gamma$
by exploiting Ecalle's mould theory 
with self-contained proofs. 
Our results are exhibited as follows:

\begin{enumerate}
\renewcommand{\labelenumi}{(\roman{enumi})}
\item
In Definition \ref{defn: KV condition},
we introduce the filtered graded $\mathbb Q$-linear space $\krv(\Gamma)_\bullet$
which generalizes $\krv_{\bullet}$.
In Theorem \ref{thm:reform:krv}, we show that 
$\krv(\Gamma)_\bullet$ is identified with the $\mathbb Q$-linear space of finite polynomial-valued alternal moulds 
satisfying Ecalle's senary relation \eqref{senary relation}
and whose $\swap$'s \eqref{eq:swap} are $\pus$-neutral \eqref{pus-neutral},
that is, there is an isomorphism of $\mathbb Q$-linear spaces
$$
\krv(\Gamma)_\bullet\simeq
\ARI(\Gamma)_{\pspush/\pusnu}\cap\ARI(\Gamma)_\al^{\fin,\pol}.$$
\item
In Definition \ref{defn:lrkv},
we introduce a bigraded $\mathbb Q$-linear space $\lkrv(\Gamma)_{\bullet\bullet}$
which is defined by the \lq leading terms' of the defining equations of 
$\krv(\Gamma)_{\bullet}$.
We also consider its subspace $\lkrvd(\Gamma)_{\bullet\bullet}$ by imposing the distribution relation in Definition \ref{def:LKRVD}.
Both of them  recover  $\lkrv_{\bullet\bullet}$ when $\Gamma=\{e\}$.
It is shown  that they form Lie algebras
in  Theorem \ref{thm:lkrv Lie alg}  and Corollary \ref{cor:reform:lkrvd}.
An inclusion of bigraded  $\mathbb Q$-linear spaces
\begin{equation*}
\gr_{\D}\krv(\Gamma)_\bullet \hookrightarrow
\lkrv(\Gamma)_{\bullet\bullet}
\end{equation*}
is presented in \eqref{eq:inclusion grKRV},
where the first term means  
the associated bigraded  
of the filtered graded linear space $\krv(\Gamma)_\bullet$. 
\item
In Theorem \ref{thm:reform:krv:bigrade}, we show that
$\lkrv(\Gamma)_{\bullet\bullet}$
is identified with
the Lie algebra (cf. Theorem \ref{thm: ARIpushpusnu Lie}) of finite polynomial-valued alternal moulds 
which are $\push$-invariant \eqref{push-invariant}
and whose $\swap$'s are $\pus$-neutral \eqref{pus-neutral},
that is, there is an isomorphism of Lie algebras
$$
\lkrv(\Gamma)_{\bullet\bullet}\simeq\ARI(\Gamma)_{\push/ \pusnu}\cap \ARI(\Gamma)_\al^{\fin,\pol}.$$
%
%
\item  In \S \ref{sec:dihedral Lie algebra},
we consider Goncharov's dihedral bigraded Lie algebra $D(\Gamma)_{\bullet\bullet}$
and its related Lie algebra ${\mathbb D}(\Gamma)_{\bullet\bullet}$ with the dihedral symmetry  which contains $D(\Gamma)_{\bullet\bullet}$.
It is explained in Theorem \ref{thm:reform:dihedral} that
its depth>1-part of ${\mathbb D}(\Gamma)_{\bullet\bullet}$ 
coincides with the depth>1-part of the Lie algebra of finite polynomial-valued part of
the set of moulds $\ARI(\Gamma)_{\underline{\al}/\underline{\al}}$ (cf. Definition \ref{def:ARIalal}), 
namely
$$
\Fil^2_\D {\mathbb D}(\Gamma)_{\bullet\bullet}\simeq
\Fil^2_\D \ARI(\Gamma)_{\underline{\al}/ \underline{\al}}^{\fin,\pol}.
$$
In Theorem \ref{thm:embedding}, we show  that 
there is an inclusion of graded Lie algebras
$$
\Fil_{\D}^{2}\ARI(\Gamma)_{\underline{\al}/ \underline{\al}}\hookrightarrow
\ARI(\Gamma)_{\push/ \pusnu}.
$$
In Corollary \ref{cor:embedding}, by taking an intersection with $ \ARI(\Gamma)_\al^{\fin,\pol}$
we obtain the inclusion of bigraded Lie algebras from
the depth >1 part $\Fil^2_{D}{\mathbb D}(\Gamma)_{\bullet\bullet}$ of ${\mathbb D}(\Gamma)_{\bullet\bullet}$ 
to $\lkrv(\Gamma)_{\bullet\bullet}$
$$
\Fil_{\D}^{2}{\mathbb D}(\Gamma)_{\bullet\bullet}\hookrightarrow\lkrv(\Gamma)_{\bullet\bullet},
$$
which 
extends \eqref{D to lkrv}.
By imposing the distribution relation there
the inclusion 
$\Fil_{\D}^{2}{D}(\Gamma)_{\bullet\bullet}\hookrightarrow\lkrvd(\Gamma)_{\bullet\bullet}$
is similarly obtained in Corollary \ref{cor:embedding2}.
\end{enumerate}

In Appendix \ref{sec:appendix},
we give self-contained proofs of  several
fundamental properties on the $\ari$-bracket of  the Lie algebra $\ARI(\Gamma)$ 
of moulds associated with $\Gamma$.
In Appendix \ref{sec:MPL at roots},
we discuss moulds arising from the  multiple polylogarithms 
evaluated at roots of unity. 


\subsubsection*{Acknowledgements}
We thank for L. Schneps who gave comments on the first version of the paper
and informing us \cite{RS}.
We are grateful to the referee whose comments helped us to improve the paper in a better form.
H.F. and N.K. have been supported by grants JSPS KAKENHI  JP18H01110
and  JP18J14774 respectively.

\section{Preparation on mould theory}\label{sec:preparation on mould theory}
We prepare several techniques of moulds which will be employed in our later sections.
The notion of moulds, the alternality, flexions and the ari-bracket
associated with a finite abelian group $\Gamma$ 
are explained 
in \S\ref{subsec:moulds} and \S\ref{subsec:ari-bracket}.
In \S\ref{subsec:swap}, we explain that the set 
$\ARI(\Gamma)_{\underline\al/\underline\al}$ of 
bialternal moulds forms a Lie algebra under the ari-bracket
(whose self-contained proof is given in Appendix \ref{sec:appendix}).
In \S\ref{subsec:push-invariance and pus-neutrality},
we introduce the  set $\ARI_{\push/\pusnu}(\Gamma)$ of push-invariant and pus-neutral
moulds and show that it forms a Lie algebra under the $\ari$-bracket
in Theorem \ref{thm: ARIpushpusnu Lie}.
\subsection{Moulds and alternality}\label{subsec:moulds} 
We introduce and discuss moulds associated with a finite abelian group $\Gamma$.

The notion of moulds was invented by Ecalle (cf. \cite[Tome I. pp.12-13]{E81}).
For our convenience we employ the following formulation
influenced by  \cite{S-ARIGARI}
which is different from  the one employed in \cite[D\'{e}finition 1 or D\'{e}finition II.1]{Cre} and \cite[\S 4.1]{Sau}.

Let $\Gamma$ be a finite abelian group.
We set 
$\mathcal{F}:=\bigcup_{m\geqslant 1}\mathbb Q
(x_1,\dots,x_m)$.

\begin{defn}
\label{def:mould}
A {\it mould} on $\Z_{\geqslant0}$
with values in  $\mathcal F$ and
indexed by $\Gamma$ in a lower layer is a collection 
\begin{equation*}
	M=\left( M^m\varia{x_1, \dots, x_m}{\sigma_1, \dots, \sigma_m} \right)_{m\in\Z_{\geqslant0}, \sigma_i\in\Gamma}
\end{equation*}
with 
$M^m\varia{x_1, \dots, x_m}{\sigma_1, \dots, \sigma_m}\in \Q(x_1,\dots,x_m)$ for $m\geqslant 0$.
\footnote{
In this case, we have $M^0(\emptyset)\in\Q$ for $m=0$.
}
We denote the set of all such moulds with values
in $\mathcal{F}$ by $\mathcal{M}(\mathcal{F};\Gamma)$. 
The set $\mathcal M(\mathcal F;\Gamma)$ forms a $\mathbb Q$-linear space by
\begin{align*}
	A+ B
	&:= \left( A^m\varia{x_1, \dots, x_m}{\sigma_1, \dots, \sigma_m}+ B^m\varia{x_1, \dots, x_m}{\sigma_1, \dots, \sigma_m} \right)_{m\in\Z_{\geqslant0}, \sigma_i\in\Gamma}, \\
	c A
	&:= \left( c A^m\varia{x_1, \dots, x_m}{\sigma_1, \dots, \sigma_m} \right)_{m\in\Z_{\geqslant0}, \sigma_i\in\Gamma},
\end{align*}
for $A, B\in\mathcal M(\mathcal F;\Gamma)$ and $c\in \mathbb Q$,
namely the addition and the scalar are taken componentwise.
We define a product on $\mathcal M(\mathcal F;\Gamma)$ by
\begin{equation*}
	 (A\times B)^m\varia{x_1, \dots, x_m}{\sigma_1, \dots, \sigma_m}
	:=\sum_{i=0}^mA^i\varia{x_1, \dots, x_i}{\sigma_1, \dots, \sigma_i}
	B^{m-i}\varia{x_{i+1}, \dots, x_m}{\sigma_{i+1}, \dots, \sigma_m},
\end{equation*}
for $A,B\in\mathcal M(\mathcal F;\Gamma)$ and for $m\geqslant0$ and for $(\sigma_1,\dots,\sigma_m)\in\Gamma^{\oplus m}$. 
Then the pair $(\mathcal M(\mathcal F;\Gamma),\times)$ is a non-commutative, associative, unital $\mathbb Q$-algebra. Here, the unit $ I\in\mathcal M(\mathcal F;\Gamma)$ is given by 
$ I:=(1,0,0,\dots)$.
\end{defn}

By the regular action of $\Gamma$ on $\mathbb Q[\Gamma]$, 
$\mathcal M(\mathcal F;\Gamma)$ admits the action of $\Gamma$.
It is described by
$$(\gamma M)^m\varia{x_1, \dots, x_m}{\sigma_1, \dots, \sigma_m}
=M^m\varia{x_1, \dots, x_m}{\gamma^{-1}\sigma_1, \dots, \gamma^{-1}\sigma_m}
$$
for $\gamma\in\Gamma$.

The set  $\mathcal M(\mathcal F;\Gamma)$ is encoded with the depth filtration
$\{\Fil^m_{\D}\mathcal M(\mathcal F;\Gamma)\}_{m\geqslant 0}$
where
$\Fil^m_{\D}\mathcal M(\mathcal F;\Gamma)$ is the collection  of moulds
with $ M^r\varia{x_1, \dots, x_r}{\sigma_1, \dots, \sigma_r}=0$ for $r<m$.
It is clear that the algebra structure of $\mathcal M(\mathcal F;\Gamma)$ is compatible with the depth filtration.
Put
 $$\ARI(\Gamma):=\{ M\in\mathcal M(\mathcal F;\Gamma)\ |\  M^0(\emptyset)=0\}.$$ 
It is a filtered (non-unital) subalgebra. 


\begin{defn}\label{defn:fin,pol}
A mould $ M\in \mathcal M(\mathcal F;\Gamma)$ is called {\it finite} when
$ M^m\varia{x_1, \dots, x_m}{\sigma_1, \dots, \sigma_m}=0$ except for finitely many $m$.
It is called {\it polynomial-valued}
when $ M^m\varia{x_1, \dots, x_m}{\sigma_1, \dots, \sigma_m}\in\mathbb Q[x_1,\dots,x_m]$ for all $(\sigma_1, \dots, \sigma_m)\in\Gamma^{\oplus m}$ and $m$.
We denote $\mathcal M(\mathcal F;\Gamma)^{\fin,\pol}$
(resp. $\ARI(\Gamma)^{\fin,\pol}$)
to be 
the subset of all finite polynomial-valued  moulds in $\mathcal M(\mathcal F;\Gamma)$ (resp. $\ARI(\Gamma)$).
\end{defn}
 
We prepare the following algebraic formulation which is useful to present  the notion of the alternality of mould:
Put $X:=\left\{\binom{x_i}{\sigma}\right\}_{i\in\N,\sigma\in\Gamma}$. Let $X_{\Z}$ be the set such that
$$X_{\Z}:=\left\{\varia{u}{\sigma}\ \middle|\ u=a_1x_1+\cdots +a_kx_k,\ k\in\N,\ a_j\in\Z,\ \sigma\in\Gamma\right\},$$
and let 
$X_{\Z}^\bullet$
be the non-commutative free monoid generated by all elements of $X_{\Z}$
with the empty word $\emptyset$ as the unit. 
Occasionally we denote each element $\omega=u_1\cdots u_m\in X_{\Z}^\bullet$
with $u_1,\dots,u_m\in X_{\Z}$ by
$\omega=(u_1,\dots,u_m)$
as a sequence.
The {\it length} of $\omega=u_1\cdots u_m$ is defined to be $l(\omega):=m$.

	For our simplicity we  occasionally denote $M\in\mathcal M(\mathcal F;\Gamma)$ by
	\begin{equation*}
		 M=( M^m(\vecx_m))_{m\in\Z_{\geqslant0}} \quad \text{ or }\quad
		 M=( M(\vecx_m))_{m\in\Z_{\geqslant0}},
	\end{equation*}
	where $\vecx_0:=\emptyset$ and $\vecx_m:=\varia{x_1,\ \dots,\ x_m}{\sigma_1,\ \dots,\ \sigma_m}$ for $m\geqslant1$. Under the notations,  the product of $ A, B\in \ARI(\Gamma)$ is expressed as
	\begin{equation*}
		 A\times  B
	=\left(\sum_{\substack{
		\vecx_m=\alpha\beta \\
		}}A^{l(\alpha)}(\alpha)B^{l(\beta)}(\beta)\right)_{m\in\Z_{\geqslant0}}
	\end{equation*}
	where $\alpha$ and $\beta$ run over  $X_{\Z}^\bullet$.

We set 
$\mathcal A_X:=\mathbb Q \langle X_{\Z} \rangle$
to be the non-commutative polynomial algebra generated by 
$X_{\Z}$
(i.e. $\mathcal A_X$ is the $\mathbb Q$-linear space generated by $X_{\Z}^{\bullet}$).
\footnote{
We should beware of the inequality 
$\varia{x_1+x_2}{\quad\sigma}\neq\varia{x_1}{\sigma}+\varia{x_2}{\sigma}$ and
$\varia{0}{\sigma}\neq0$.
} 
We equip $\mathcal A_X$ with a product $\shuffle:\mathcal A_X^{\otimes2} \rightarrow \mathcal A_X$ (called the {\it shuffle product}) 
which is linearly defined by $\emptyset\ \shuffle\ \omega:=\omega\ \shuffle\ \emptyset:=w$ and
\begin{equation}
	u\omega\ \shuffle\ v\eta:=u(\omega\ \shuffle\ v\eta)+v(u\omega\ \shuffle\ \eta),
\end{equation}
for $u,v\in X_{\Z}$
and $\omega,\eta\in X_{\Z}^\bullet$.
Then the pair $(\mathcal A_X,\shuffle)$ forms a commutative, associative, unital $\mathbb Q$-algebra.

We define the family $\left\{\Sh{\omega}{\eta}{\alpha}\right\}_{\omega,\eta,\alpha\in X_{\Z}^\bullet}$ in $\Z$ by
\begin{equation*}
	\omega\ \shuffle\ \eta
	=\sum_{\alpha\in X_{\Z}^\bullet}
	\Sh{\omega}{\eta}{\alpha}\alpha.
\end{equation*}
Particularly for $p,q\in\N$ and $u_1,\dots,u_{p+q}\in X_{\Z}$, we rewrite the shuffle product by 
\begin{equation*}
	(u_1,\dots,u_p)\ \shuffle\ (u_{p+1},\dots,u_{p+q})=\sum_{\sigma\in\Sha_{p,q}}(u_{\sigma(1)},\dots,u_{\sigma(p)},u_{\sigma(p+1)},\dots,u_{\sigma(p+q)}).
\end{equation*}
Here the set $\Sha_{p,q}$ is defined by 
\begin{equation}\label{shuffle permutation}
	\{\sigma\in S_{p+q}
	\ |\ \sigma^{-1}(1)<\cdots<\sigma^{-1}(p),\ \sigma^{-1}(p+1)<\cdots<\sigma^{-1}(p+q)\},
\end{equation}
where $S_{p+q}$ is the symmetry group with degree $p+q$.

\begin{defn}\label{def:ARIal}
A mould $M \in\ARI(\Gamma)$ is called {\it alternal} (cf. \cite[I--p.118]{E81}) if
\begin{equation}\label{eq:alternal}
\sum_{\alpha\in X_{\Z}^\bullet}\Sh{\varia{x_1,\ \dots,\ x_p}{\sigma_1,\ \dots,\ \sigma_p}}{\varia{x_{p+1},\ \dots,\ x_{p+q}}{\sigma_{p+1},\ \dots,\ \sigma_{p+q}}}{\alpha} M^{p+q}(\alpha)=0,
\end{equation}
for all $p,q\geqslant1$. 
The $\mathbb Q$-linear space $\ARI(\Gamma)_\al$ is defined to be the subset of moulds $ M\in\ARI(\Gamma)$ which are alternal
(cf. \cite{E-ARIGARI}, \cite{E-flex}).
\end{defn}

We encode it 
 with the induced depth filtrations.
We exhibit a couple of examples of alternal moulds for $\Gamma=\{e\}$ below.

\begin{eg}\label{ex:alternal mould}
(a). For $f(x)\in \mathbb Q(x)$, 
we define $M_f\in\ARI(\Gamma)$ by
	\begin{equation*}
		M_f(\vecx_m):=\left\{\begin{array}{ll}
		f(x_1) & (m=1),\\
		0 & (m\neq1),
		\end{array}\right.
	\end{equation*}
	that is,
	\begin{equation*}
		M_f=\left(0,\ f(x_1),\ 0,\ 0,\dots\right).
	\end{equation*}
(b). We define $A\in\ARI(\Gamma)$ by
\footnote{This mould is presented in \cite[p. 124]{E81} and \cite[(II.64)]{Cre}.}
	 \begin{equation*}
		A(\vecx_m)
		:=\left\{\begin{array}{ll}
			0 & (m=0,1), \\
			\frac{1}{x_2-x_1}\cdots\frac{1}{x_m-x_{m-1}} & (m\geqslant2),
		\end{array}\right.
	\end{equation*}
	that is,
	\begin{equation*}
		A=\left(0,\ 0,\ \frac{1}{x_2-x_1},\ \frac{1}{(x_2-x_1)(x_3-x_2)},\ \frac{1}{(x_2-x_1)(x_3-x_2)(x_4-x_3)},\dots\right).
	\end{equation*}
\end{eg}

\noindent
{\it Proof of their alternalities}:
(a).\ By definition, we have $M_f(\emptyset)=0$. Put $\omega,\eta\in X^\bullet$ with $l(\omega),l(\eta)\geqslant 1$. Note that for $\alpha\in X^\bullet$ with $\Sh{\omega}{\eta}{\alpha}\neq 0$, we have $l(\alpha)=l(\omega)+l(\eta)\geqslant 2$ and we get
\begin{equation*}
	M_f(\alpha)=0.
\end{equation*}
Therefore, we obtain
\begin{equation*}
	\sum_{\alpha\in X^\bullet}\Sh{\omega}{\eta}{\alpha}M_f(\alpha)
	=\sum_{\substack{
		\alpha\in X^\bullet \\
		\Sh{\omega}{\eta}{\alpha}\neq0}}
	\Sh{\omega}{\eta}{\alpha}M_f(\alpha)
	=0.
\end{equation*}
Hence, $M_f$ is an alternal mould.

\noindent
(b).\ The proof is given in \cite[Lemme II.5]{Cre}, but we prove the alternality of $A$ by using the following lemma:

\begin{lem}\label{lem:ex of alternal}
For $r\geqslant 2$ and $s\geqslant 0$, we have
\begin{equation*}\label{eqn:2.2.4}
	\sum_{\alpha\in X_\Z^\bullet}\Sh{(\omega_1,\dots,\omega_{r-1})}{(\omega_{r+1},\dots,\omega_{r+s})}{\alpha}A(\alpha,\omega_r)
	=(-1)^sA(\omega_1,\dots,\omega_{r-1},\omega_r,\omega_{r+s},\dots,\omega_{r+1}),
\end{equation*}
where $\omega_i:=\varia{x_i}{\, e}$ for $i\in\N$.
\end{lem}

\begin{proof}
When $s=0$, the claim is clear.
Assume $s\geqslant 1$.
We prove by induction on $r+s$.
By the definition of the mould $A$, we easily see the case of $(r,s)=(2,1)$.
For $r\geqslant 2$ and $s\geqslant 1$, we have
\begin{align*}
	\sum_{\alpha\in X_\Z^\bullet}&\Sh{(\omega_1,\dots,\omega_{r-1})}{(\omega_{r+1},\dots,\omega_{r+s})}{\alpha}A(\alpha,\omega_r) \\
	=&\sum_{\alpha\in X_\Z^\bullet}\left\{\Sh{(\omega_1,\dots,\omega_{r-2})}{(\omega_{r+1},\dots,\omega_{r+s})}{\alpha}A(\alpha,\omega_{r-1},\omega_r)\right. \\
	&\left.+\Sh{(\omega_1,\dots,\omega_{r-1})}{(\omega_{r+1},\dots,\omega_{r+s-1})}{\alpha}A(\alpha,\omega_{r+s},\omega_r)\right\}. \\
\intertext{Because $A(\omega_1,\dots,\omega_{r-1},\omega_r)=\frac{1}{x_r-x_{r-1}}A(\omega_1,\dots,\omega_{r-1})$ for $r\geqslant 3$, we get}
	=&\frac{1}{x_r-x_{r-1}}\sum_{\alpha\in X_\Z^\bullet}\Sh{(\omega_1,\dots,\omega_{r-2})}{(\omega_{r+1},\dots,\omega_{r+s})}{\alpha}A(\alpha,\omega_{r-1}) \\
	&+\frac{1}{x_r-x_{r+s}}\sum_{\alpha\in X_\Z^\bullet}\Sh{(\omega_1,\dots,\omega_{r-1})}{(\omega_{r+1},\dots,\omega_{r+s-1})}{\alpha}A(\alpha,\omega_{r+s}). 
\intertext{By our induction hypothesis, we calculate}
	=&(-1)^{s}\frac{1}{x_r-x_{r-1}}A(\omega_1,\dots,\omega_{r-2},\omega_{r-1},\omega_{r+s},\dots,\omega_{r+1}) \\
	&+(-1)^{s-1}\frac{1}{x_r-x_{r+s}}A(\omega_1,\dots,\omega_{r-1},\omega_{r+s},\omega_{r+s-1},\dots,\omega_{r+1}) \\
	=&(-1)^{s}\left\{\frac{1}{x_r-x_{r-1}}-\frac{1}{x_r-x_{r+s}}\right\}A(\omega_1,\dots,\omega_{r-2},\omega_{r-1},\omega_{r+s},\dots,\omega_{r+1}) \\
	=&(-1)^{s}\frac{x_{r-1}-x_{r+s}}{(x_r-x_{r-1})(x_r-x_{r+s})}
	A(\omega_1,\dots,\omega_{r-2},\omega_{r-1},\omega_{r+s},\dots,\omega_{r+1}) \\
	=&(-1)^{s}\frac{x_{r-1}-x_{r+s}}{(x_r-x_{r-1})(x_r-x_{r+s})}A(\omega_1,\dots,\omega_{r-1})
	\frac{1}{x_{r+s}-x_{r-1}}\frac{1}{x_{r+s-1}-x_{r+s}}\cdots\frac{1}{x_{r+1}-x_{r+2}} \\
	=&(-1)^{s}A(\omega_1,\dots,\omega_{r-1})\frac{1}{x_r-x_{r-1}}\frac{1}{x_{r+s}-x_r}
	\frac{1}{x_{r+s-1}-x_{r+s}}\cdots\frac{1}{x_{r+1}-x_{r+2}} \\
	=&(-1)^sA(\omega_1,\dots,\omega_{r-1},\omega_r,\omega_{r+s},\dots,\omega_{r+1}).
\end{align*}
Hence, we obtain the claim.
\end{proof}
By using the above lemma, we prove the alternality of $A$.
By the definition, we have $A(\emptyset)=0$, so it is sufficient to prove
\begin{equation}\label{eqn:alternality of ex mould}
	\sum_{\alpha\in X_\Z^\bullet}\Sh{(\omega_1,\dots,\omega_r)}{(\omega_{r+1},\dots,\omega_{r+s})}{\alpha}A(\alpha)=0,
\end{equation}
for $r,s\geqslant 1$.
Assume $r\geqslant s$ without loss of generality.
When $r=s=1$, the left hand side of \eqref{eqn:alternality of ex mould} is equal to
\begin{equation*}
	A(\omega_1,\omega_2)+A(\omega_2,\omega_1)=\frac{1}{x_2-x_1}+\frac{1}{x_1-x_2}=0.
\end{equation*}
When $r\geqslant 2$ and $s\geqslant 1$, the left hand side of \eqref{eqn:alternality of ex mould} is equal to
\begin{align*}
&\sum_{\alpha\in X_\Z^\bullet}\Sh{(\omega_1,\dots,\omega_{r-1})}{(\omega_{r+1},\dots,\omega_{r+s})}{\alpha}A(\alpha,\omega_r) \\
	&\quad+\sum_{\alpha\in X_\Z^\bullet}\Sh{(\omega_1,\dots,\omega_r)}{(\omega_{r+1},\dots,\omega_{r+s-1})}{\alpha}A(\alpha,\omega_{r+s}).
\intertext{By using Lemma \ref{lem:ex of alternal}, we get}
	&=(-1)^sA(\omega_1,\dots,\omega_{r-1},\omega_r,\omega_{r+s},\dots,\omega_{r+1})+(-1)^{s-1}A(\omega_1,\dots,\omega_r,\omega_{r+s},\omega_{r+s-1},\dots,\omega_{r+1}) \\
	&=0.
\end{align*}
Hence, the mould $A$ is alternal.
\hfill $\Box$

\begin{rem}\label{rem:component of mould by embedding}
Assume that $u_1,\dots,u_m\in \mathcal F$ are algebraically independent over $\mathbb Q$.  
For $ M\in\ARI(\Gamma)$ we denote
$		 M^m\varia{u_1,\ \dots,\ u_m}{\sigma_1,\ \dots,\ \sigma_m}$
to be the image of $M^m\varia{x_1,\ \dots,\ x_m}{\sigma_1,\ \dots,\ \sigma_m}$ under the field embedding
$\mathbb Q(x_1,\dots,x_m)\hookrightarrow \mathcal F$
sending $x_i\mapsto u_i$.
\end{rem}
For our later use, we  prepare more notations:

\begin{nota}[{\cite[\S 2.1]{E-flex}}]\label{nota:mould operations}
For any mould $M=\left( M^m\varia{u_1,\ \dots,\ u_m}{\sigma_1,\ \dots,\ \sigma_m}\right)_m\in\ARI(\Gamma)$,  we define 
\begin{align*}
&\mantar( M)^m
{\scriptsize\left(\begin{array}{rrr}
	u_1,& \dots,& u_m \\
	\sigma_1,& \dots,& \sigma_m
\end{array}\right)}
	=(-1)^{m-1} M^m
	{\scriptsize\left(\begin{array}{rrr}
		u_m,& \dots,& u_1 \\
		\sigma_m,& \dots,& \sigma_1
	\end{array}\right)},\\
&\push( M)^m
{\scriptsize\left(\begin{array}{rrr}
	u_1,& \dots,& u_m \\
	\sigma_1,& \dots,& \sigma_m
\end{array}\right)}
	= M^m
	{\scriptsize\left(\begin{array}{rrrr}
		-u_1-\cdots-u_m,& u_1,& \dots,& u_{m-1} \\
		\sigma_m^{-1},& \sigma_1\sigma_m^{-1},& \dots,& \sigma_{m-1}\sigma_m^{-1}
	\end{array}\right)},\\
&\nega( M)^m
{\scriptsize\left(\begin{array}{rrr}
	u_1,& \dots,& u_m \\
	\sigma_1,& \dots,& \sigma_m
\end{array}\right)}
	= M^m
	{\scriptsize\left(\begin{array}{rrr}
		-u_1,& \dots,& -u_m \\
		\sigma_1^{-1},& \dots,& \sigma_m^{-1}
\end{array}\right)},\\
&\teru( M)^m
{\scriptsize\left(\begin{array}{rrr}
	u_1,& \dots,& u_m \\
	\sigma_1,& \dots,& \sigma_m
\end{array}\right)}
	= M^m
	{\scriptsize\left(\begin{array}{rrr}
		u_1,& \dots,& u_m \\
		\sigma_1,& \dots,& \sigma_m
\end{array}\right)} \\
&\qquad+\frac{1}{u_m}\left\{ M^{m-1}
	{\scriptsize\left(\begin{array}{rrrl}
		u_1,& \dots,& u_{m-2},& u_{m-1}+u_m \\
		\sigma_1,& \dots,& \sigma_{m-2},& \sigma_{m-1}
	\end{array}\right)}
	- M^{m-1}
	{\scriptsize\left(\begin{array}{rrr}
		u_1,& \dots,& u_{m-1} \\
		\sigma_1,& \dots,& \sigma_{m-1}
	\end{array}\right)}\right\}.
\end{align*}
\end{nota}
Note that they are all $\mathbb Q$-linear endomorphisms on $\ARI(\Gamma)$.
We remark that $\nega\circ\nega=\id$ and
 $\mantar\circ\mantar=\id$.

\begin{defn}
We call a mould $M\in\ARI(\Gamma)$ {\it push-invariant} 
when we have
\begin{equation}\label{push-invariant}
	\push( M)= M.
\end{equation}
We define $\ARI(\Gamma)_{\push}$ (\cite[\S 2.5]{E-flex})
to be the set of moulds $ M$ in $\ARI(\Gamma)$
which is push-invariant \eqref{push-invariant}.
\end{defn}

\begin{eg}
By direct computations we observe that 
the mould $P\in \ARI(\{e\})$ defined by
\begin{equation*}
	P(\vecx_m)
	:=\left\{\begin{array}{ll}
		0 & (m=0,1), \\
		\frac{1}{x_1} + \cdots + \frac{1}{x_m} - \frac{1}{x_1 + \cdots + x_m} & (m\geqslant2)
	\end{array}\right.
\end{equation*}
is $\push$-invariant, but is not alternal.
\end{eg}

\subsection{Flexions and ari-bracket}\label{subsec:ari-bracket}
We explain an ari-bracket $\ari_u$ on $\ARI(\Gamma)$
by using flexions. 


The notion of flexions  is introduced by Ecalle in \cite[\S 2.1]{E-flex}
for bimoulds (cf. \cite[\S 2.2]{S-ARIGARI}). 
Here we consider those for moulds in $\ARI(\Gamma)$.
\begin{defn}\label{def:flexion}
The {\it flexions} are the 
 four binary operators $\urflex{*}{*},\ \ulflex{*}{*},\ \lrflex{*}{*},\ \llflex{*}{*}:X_{\Z}^\bullet\times X_{\Z}^\bullet\rightarrow X_{\Z}^\bullet$
 which are defined by
\begin{align*}
	\urflex{\beta}{\alpha}
	&:= {\scriptsize\left(\begin{array}{rrrr}
		b_1+\cdots+b_n+a_1,& a_2,& \dots,& a_m \\
		\sigma_1,& \sigma_2,& \dots,& \sigma_m 
	\end{array}\right)} , \\
	\ulflex{\alpha}{\beta} 
	&:= {\scriptsize\left(\begin{array}{rrrl}
		a_1,& \dots,& a_{m-1},& a_m+b_1+\cdots+b_n \\
		\sigma_1,& \dots,& \sigma_{m-1},& \sigma_m
	\end{array}\right)} , \\
	\lrflex{\beta}{\alpha} 
	&:={\scriptsize\left(\begin{array}{rrr}
		a_1,& \dots,& a_m \\
		\tau_n^{-1}\sigma_1,& \dots,& \tau_n^{-1}\sigma_m
	\end{array}\right)} , \\
	\llflex{\alpha}{\beta} 
	&:={\scriptsize\left(\begin{array}{rrr}
		a_1,& \dots,& a_m \\
		\sigma_1\tau_1^{-1},& \dots,& \sigma_m\tau_1^{-1}
	\end{array}\right)}  , \\
	\urflex{\emptyset}{\gamma}&:=\ulflex{\gamma}{\emptyset}:=\lrflex{\emptyset}{\gamma}:=\llflex{\gamma}{\emptyset}:=\gamma , \\
	\urflex{\gamma}{\emptyset}&:=\ulflex{\emptyset}{\gamma}:=\lrflex{\gamma}{\emptyset}:=\llflex{\emptyset}{\gamma}:=\emptyset ,
\end{align*}
for $\alpha=\varia{a_1,\dots,a_m}{\sigma_1,\dots,\sigma_m},\beta=\varia{b_1,\dots,b_n}{\tau_1,\dots,\tau_n}\in X_{\Z}^\bullet$ ($m,n\geqslant1$) and $\gamma\in X_{\Z}^\bullet$.
\end{defn}
	Note that we have
		$
		l(\urflex{\beta}{\alpha})=l(\ulflex{\alpha}{\beta})=l(\lrflex{\beta}{\alpha})=l(\llflex{\alpha}{\beta})=l(\alpha) \text{ and } 
		l(\alpha,\beta)=l(\alpha)+l(\beta)
		$
	for $\alpha,\beta\in X_{\Z}^\bullet$.
%

The derivation $\arit$ and bracket $\ari$  are introduced for bimoulds in terms of flexions in \cite[\S 2.2]{E-flex} (cf. \cite[\S 2.2]{S-ARIGARI}) and
here we consider those for $\ARI(\Gamma)$ as follows.

\begin{defn}\label{def:aritu}
      Let $ B\in \mathcal \ARI(\Gamma)$. The linear map
      \footnote{The lower suffix $u$ is reflected by the notion of $u$-moulds in \cite{S-ARIGARI}.}
     $\arit_u( B):\ARI(\Gamma) \rightarrow \ARI(\Gamma)$ is defined by
     \footnote{\label{footnote:arit embedding} For $m\geqslant2$ and $\alpha,\beta,\gamma\in X_\Z^\bullet$ with $\vecx_m=\alpha\beta\gamma$, all letters appearing in the word $\alpha\urflex{\beta}{\gamma}$ (resp. $\lrflex{\beta}{\gamma}$) are algebraically independent over $\mathbb Q$. So by Remark \ref{rem:component of mould by embedding}, the component $ A^{l(\alpha,\gamma)}(\alpha\urflex{\beta}{\gamma})$ (resp. $B^{l(\beta)}(\llflex{\beta}{\gamma})$) is well-defined. Similarly, $ A^{l(\alpha,\gamma)}(\ulflex{\alpha}{\beta}\gamma)$ and $B^{l(\beta)}(\lrflex{\alpha}{\beta})$ are also well-defined.}
	\begin{align*}
		&(\arit_u( B)( A))^m(\vecx_m)
		=(\arit_u( B)( A))^m\varia{x_1,\ \dots,\ x_m}{\sigma_1,\ \dots,\ \sigma_m} \\
		&:=\left\{\begin{array}{ll}
			\displaystyle \sum_{\substack{\vecx_m=\alpha\beta\gamma \\
				\beta,\gamma\neq\emptyset}}
			 A^{l(\alpha,\gamma)}(\alpha\urflex{\beta}{\gamma}) B^{l(\beta)}(\llflex{\beta}{\gamma})
			-\sum_{\substack{\vecx_m=\alpha\beta\gamma \\
				\alpha,\beta\neq\emptyset}}
			 A^{l(\alpha,\gamma)}(\ulflex{\alpha}{\beta}\gamma) B^{l(\beta)}(\lrflex{\alpha}{\beta}) & (m\geqslant2),\\
			0 & (m=0,1),
		\end{array}\right.
	\end{align*}
	for $A\in\ARI(\Gamma)$.
\end{defn}


It is shown in \cite[Appendix A.1]{S-ARIGARI}  that 
$\arit$ forms a derivation on the set of bimoulds.
The same holds for  $\ARI(\Gamma)$  as follows,
where we present an alternative  proof  to hers.

\begin{lem}\label{lem:arit derivation}
For any $A\in\ARI(\Gamma)$, $\aritu(A)$ forms a derivation of $\ARI(\Gamma)$
with respect to the product $\times$,
 that is, for any $B,C\in\ARI(\Gamma)$, we have
\begin{equation*}
	\aritu(A)(B\times C)=\aritu(A)(B)\times C+B\times \aritu(A)(C).
\end{equation*}
\end{lem}

\begin{proof}
Let $m\geqslant0$. We have
{\footnotesize
\begin{align*}
	(\arit(A) & (B \times C))(\vecx_m)
	= \sum_{\substack{
		\vecx_m=abc \\
		b,c\neq\emptyset}}
	(B\times C)(a \urflex{b}{c})A(\llflex{b}{c})
	-\sum_{\substack{
		\vecx_m=abc \\
		a,b\neq\emptyset}}
	(B\times C)(\ulflex{a}{b} c)A(\lrflex{a}{b}) \\
	=& \sum_{\substack{
		\vecx_m=abc \\
		b,c\neq\emptyset}}
	\left\{
		\sum_{a \urflex{b}{c}=de}
		B(d)C(e)
	\right\}A(\llflex{b}{c})
	-\sum_{\substack{
		\vecx_m=abc \\
		a,b\neq\emptyset}}
	\left\{
		\sum_{\ulflex{a}{b}c=ef}
		B(e)C(f)
	\right\}A(\lrflex{a}{b})
	\intertext{{\normalsize
By applying Lemma \ref{lem:tri-factorization}.(1) for $c\neq\emptyset$, $f=\emptyset$ to the first term and by applying Lemma \ref{lem:tri-factorization}.(2) for $a\neq\emptyset$, $d=\emptyset$ to the second term, we have}}
	=& \sum_{\substack{
		\vecx_m=abc \\
		b,c\neq\emptyset}}
	\left\{
	\sum_{a=a_1a_2}
	B(a_1)C(a_2 \urflex{b}{c})
	+\sum_{\substack{
		c=c_1c_2 \\
		c_1\neq\emptyset}}
	B(a \urflex{b}{c_1})C(c_2)
	\right\}A(\llflex{b}{c}) \\
	&-\sum_{\substack{
		\vecx_m=abc \\
		a,b\neq\emptyset}}
	\left\{
	\sum_{\substack{
		a=a_1a_2 \\
		a_2\neq\emptyset}}
	B(a_1)C(\ulflex{a_2}{b} c)
	+\sum_{c=c_1c_2}
	B(\ulflex{a}{b} c_1)C(c_2)
	\right\}A(\lrflex{a}{b}) \\
	=& \sum_{\substack{
		\vecx_m=a_1a_2bc \\
		b,c\neq\emptyset}}
	B(a_1)C(a_2 \urflex{b}{c})A(\llflex{b}{c})
	+\sum_{\substack{
		\vecx_m=abc_1c_2 \\
		b,c_1\neq\emptyset}}
	B(a \urflex{b}{c_1})C(c_2)A(\llflex{b}{c_1c_2}) \\
	&-\sum_{\substack{
		\vecx_m=a_1a_2bc \\
		a_2,b\neq\emptyset}}
	B(a_1)C(\ulflex{a_2}{b}c)A(\lrflex{a_1a_2}{b})
	-\sum_{\substack{
		\vecx_m=abc_1c_2 \\
		a,b\neq\emptyset}}
	B(\ulflex{a}{b} c_1)C(c_2)A(\lrflex{a}{b}).
	\intertext{{\normalsize
By applying Lemma \ref{lem:flexion's action} to the second term, we get $A(\llflex{b}{c_1c_2})=A(\llflex{b}{c_1})$. Similarly, for the third term, we get $A(\lrflex{a_1a_2}{b})=A(\lrflex{a_2}{b})$. Therefore, we calculate}}
	=& \sum_{\vecx_m=dc_2}
	\left\{
	\sum_{\substack{
		d=abc_1 \\
		b,c_1\neq\emptyset}}
	B(a \urflex{b}{c_1})A(\llflex{b}{c_1})
	-\sum_{\substack{
		d=abc_1 \\
		a,b\neq\emptyset}}
	B(\ulflex{a}{b} c_1)A(\lrflex{a}{b})
	\right\}C(c_2) \\
	&+\sum_{\vecx_m=a_1d}
	B(a_1)
	\left\{
	\sum_{\substack{
		d=a_2bc \\
		b,c\neq\emptyset}}
	C(a_2 \urflex{b}{c})A(\llflex{b}{c})
	-\sum_{\substack{
		d=a_2bc \\
		a_2,b\neq\emptyset}}
	C(\ulflex{a_2}{b} c)A(\lrflex{a_2}{b})
	\right\} \\
	=& \sum_{\vecx_m=dc_2}
	(\arit(A)(B))(d)C(c_2) 
	+\sum_{\bf w_n=a_1d}
	B(a_1)(\arit(A)(C))(d) \\
	=& (\arit(A)(B)\times C)(\vecx_m) 
	+(B\times\arit(A)(C))(\vecx_m).	
\end{align*}
}
Hence, we obtain the claim.
\end{proof}

We define the following bracket as with the bracket $\ari$ introduced in \cite[(2.40)]{E-flex}.
\begin{defn}
The {\it $\ari_u$-bracket} means 
the bilinear map $\ari_u:\ARI(\Gamma)^{\otimes2}\rightarrow \ARI(\Gamma)$ which is defined by
\begin{equation}\label{ari-bracket}
	\ari_u( A, B)
	:=\arit_u( B)( A)-\arit_u( A)( B)
	+[ A,  B]
\end{equation}
for $ A, B\in\ARI(\Gamma)$.
Here
\footnote{In the papers \cite{E-flex}, \cite{S-ARIGARI} and \cite{RS}, the product $A\times B$ (resp. the bracket $[A,B]$) is denoted by $mu(A,B)$ (resp. $lu(A,B)$).}
, we have
$ [A, B]:= A\times  B- B\times  A$.
\end{defn}
		
We note that the bracket $\ari_u(A,  B)$  in the case when $\Gamma=\{e\}$
also appears in \cite[(A.3)]{R-PhD}
and is denoted by  $[ A, B]_{\ari}$.

The following is also stated for moulds and bimoulds
in \cite[Proposition 2.2.2]{S-ARIGARI}
where her key formula (2.2.10) looks unproven and containing a signature error. 

\begin{prop}\label{ARI Lie algebra}
	The $\mathbb Q$-linear space $\ARI(\Gamma)$ forms a filtered Lie algebra under the $\ari_u$-bracket.
\end{prop}
\begin{proof}
	We give a self-contained proof in Appendix \ref{sec:A.1}.
\end{proof}

The following proposition for $\Gamma=\{e\}$ is shown in \cite[Appendix A]{SS}.
\begin{prop}
\label{ARIal Lie algebra}
The $\mathbb Q$-linear space $\ARI(\Gamma)_\al$ forms a filtered Lie subalgebra of $\ARI(\Gamma)$ under the $\ari_u$-bracket.
\end{prop}
\begin{proof}
	We prove this in Appendix \ref{sec:A.2}.
\end{proof}



\subsection{Swap and bialternality}\label{subsec:swap}
We encode  $\ARI(\Gamma)$ with another Lie algebra structure introduced by $\ari_v$.
We prepare $\overline{\ARI}(\Gamma)$, a copy of $\ARI(\Gamma)$.
We denote
$M^m_{\sigma_1, \dots, \sigma_m}(x_1,\ \dots,\ x_m)$ by $M^m
\varia{\sigma_1,\ \dots,\ \sigma_m}{x_1,\ \dots,\ x_m}$
\footnote{
We note that the top and bottom rows are switched.
}
for each element $M$ in $\overline{\ARI}(\Gamma)$
to distinguish it from an element in  $\ARI(\Gamma)$.

Similarly to our previous sections, we work over the following algebraic formulation:
Put $Y:=\left\{\binom{\sigma}{y_i}\right\}_{i\in\N,\sigma\in\Gamma}$. Let $Y_{\Z}$ be the set such that
$$Y_{\Z}:=\left\{\varia{\sigma}{v}\ \middle|\ v=a_1y_1+\cdots +a_ky_k,\ k\in\N,\ a_j\in\Z,\ \sigma\in\Gamma\right\},$$
and let 
$Y_{\Z}^\bullet$
be the non-commutative free monoid generated by all elements of $Y_{\Z}$
with the empty word $\emptyset$ as the unit. 
We set 
$\mathcal A_Y:=\mathbb Q \langle Y_{\Z} \rangle$
to be the non-commutative polynomial algebra generated by 
$Y_{\Z}$.
In the same way to $\mathcal A_X$, 
it is equipped with a structure of
a commutative, associative, unital $\mathbb Q$-algebra
with
the shuffle product $\shuffle:\mathcal A_Y^{\otimes2} \rightarrow \mathcal A_Y$ .
The flexions are also introduced in this setting.
\begin{defn}
The {\it flexions} on  $Y_{\Z}^\bullet$ are the 
 four binary operators $\urflex{*}{*},\ \ulflex{*}{*},\ \lrflex{*}{*},\ \llflex{*}{*}:Y_{\Z}^\bullet\times Y_{\Z}^\bullet\rightarrow Y_{\Z}^\bullet$
 which are defined by
\begin{align*}
	\urflex{\beta}{\alpha}
	&:= {\scriptsize\left(\begin{array}{crrr}
		\sigma_1\tau_1\cdots\tau_n,& \sigma_2,& \dots,& \sigma_m \\
		a_1,& a_2,& \dots,& a_m 
	\end{array}\right)} , \\
	\ulflex{\alpha}{\beta} 
	&:= {\scriptsize\left(\begin{array}{rrrc}
		\sigma_1,& \dots,& \sigma_{m-1},& \sigma_m\tau_1\cdots\tau_n \\
		a_1,& \dots,& a_{m-1},& a_m
	\end{array}\right)} , \\
	\lrflex{\beta}{\alpha} 
	&:={\scriptsize\left(\begin{array}{crc}
		\sigma_1,& \dots,& \sigma_m \\
		a_1-b_n,& \dots,& a_m-b_n
	\end{array}\right)} , \\
	\llflex{\alpha}{\beta} 
	&:={\scriptsize\left(\begin{array}{crc}
		\sigma_1,& \dots,& \sigma_m \\
		a_1-b_1,& \dots,& a_m-b_1
	\end{array}\right)}  , \\
	\urflex{\emptyset}{\gamma}&:=\ulflex{\gamma}{\emptyset}:=\lrflex{\emptyset}{\gamma}:=\llflex{\gamma}{\emptyset}:=\gamma , \\
	\urflex{\gamma}{\emptyset}&:=\ulflex{\emptyset}{\gamma}:=\lrflex{\gamma}{\emptyset}:=\llflex{\emptyset}{\gamma}:=\emptyset ,
\end{align*}
for $\alpha=\varia{\sigma_1,\dots,\sigma_m}{a_1,\dots,a_m},\beta=\varia{\tau_1,\dots,\tau_n}{b_1,\dots,b_n}\in Y_{\Z}^\bullet$ ($m,n\geqslant1$) and $\gamma\in Y_{\Z}^\bullet$.
\end{defn}

We denote $\vecy_0:=\emptyset$ and $\vecy_m:=\varia{\sigma_1,\ \dots,\ \sigma_m}{y_1,\ \dots,\ y_m}$ for $m\geqslant1$.
\begin{defn}
      Let $ B\in \overline{\ARI}(\Gamma)$. The linear map $\arit_v( B):\overline{\ARI}(\Gamma) \rightarrow \overline{\ARI}(\Gamma)$ is defined by
            \footnote{The lower suffix $v$ is reflected by the notion of $v$-moulds in \cite{S-ARIGARI}.}
	\begin{align*}
		&(\arit_v( B)( A))^m(\vecy_m)
		=(\arit_v( B)( A))^m\varia{\sigma_1,\ \dots,\ \sigma_m}{y_1,\ \dots,\ y_m} \\
		&:=\left\{\begin{array}{ll}
			\displaystyle \sum_{\substack{\vecy_m=\alpha\beta\gamma \\
				\beta,\gamma\neq\emptyset}}
			 A^{l(\alpha,\gamma)}(\alpha\urflex{\beta}{\gamma}) B^{l(\beta)}(\llflex{\beta}{\gamma})
			-\sum_{\substack{\vecy_m=\alpha\beta\gamma \\
				\alpha,\beta\neq\emptyset}}
			 A^{l(\alpha,\gamma)}(\ulflex{\alpha}{\beta}\gamma) B^{l(\beta)}(\lrflex{\alpha}{\beta}) & (m\geqslant2),\\
			0 & (m=0,1),
		\end{array}\right.
	\end{align*}
	for $A\in\overline{\ARI}(\Gamma)$.
\end{defn}
Similarly to Lemma \ref{lem:arit derivation}, the following holds.
\begin{lem}\label{lem:aritv derivation}
For any $A\in\overline{\ARI}(\Gamma)$, $\arit_v(A)$ forms a derivation of $\overline{\ARI}(\Gamma)$
with respect to the product $\times$.
\end{lem}
\begin{proof}
The $\arit_v$-bracket can be expressed by the exactly same formula as the $\aritu$-bracket in terms of flexions. 
Therefore it can be proved in the same way to the one of
Lemma \ref{lem:arit derivation}.
\end{proof}

\begin{defn}
The {\it $\ari_v$-bracket} means 
the bilinear map $\ari_v:\overline{\ARI}(\Gamma)^{\otimes2}\rightarrow \overline{\ARI}(\Gamma)$ which is defined by
\footnote{In \cite{RS}, the brackets $\arit_v( A, B)$ and $\ari_v( A, B)$ are denoted by $\overline{\arit}$ and $\overline{\ari}$ respectively.}
\begin{equation}\label{ari v-bracket}
	\ari_v( A, B)
	:=\arit_v( B)( A)-\arit_v( A)( B)
	+[ A,  B]
\end{equation}
for $ A, B\in\overline{\ARI}(\Gamma)$.
Here 
$ [A, B]:= A\times  B- B\times  A$.
\end{defn}
Similarly to Proposition \ref{ARI Lie algebra}, the following holds.
\begin{prop}\label{barARI Lie algebra}
	The $\mathbb Q$-linear space $\overline{\ARI}(\Gamma)$ forms a filtered Lie algebra under the $\ari_v$-bracket.
\end{prop}
\begin{proof}
It can be also proved in the same way to the one of
Proposition \ref{ARI Lie algebra}.
\end{proof}

In the same way to Definition \ref{def:ARIal},
alternal moulds in $\overline{\ARI}(\Gamma)$ can be introduced and we denote  $\overline{\ARI}(\Gamma)_\al$ to be
its subset consisting of  alternal moulds. 
Similarly to Proposition \ref{ARIal Lie algebra}, the following holds.

\begin{prop}\label{barARIal Lie algebra}
$\overline{\ARI}(\Gamma)_\al$ forms a Lie algebra under the $\ari_v$-bracket.
\end{prop}
\begin{proof}
It can be proved in the same way to the one of
Proposition \ref{ARIal Lie algebra}.
\end{proof}

We define the $\mathbb Q$-linear map $\swap:\ARI(\Gamma) \rightarrow \overline{\ARI}(\Gamma)$ by
\begin{equation}\label{eq:swap}
	\swap( M)^m
{\scriptsize\left(\begin{array}{rrr}
	\sigma_1,& \dots,& \sigma_m \\
	v_1,& \dots,& v_m
\end{array}\right)}
= M^m
{\scriptsize\left(\begin{array}{rrrrr}
	v_m,& v_{m-1}-v_m,& \dots,& v_2-v_3,& v_1-v_2 \\
	\sigma_1\cdots\sigma_m,& \sigma_1\cdots\sigma_{m-1},& \dots,& \sigma_1\sigma_2,& \sigma_1
\end{array}\right)}
\end{equation}
for any mould $M=\left(M^m\varia{u_1,\ \dots,\ u_m}{\sigma_1,\ \dots,\ \sigma_m}\right)\in\ARI(\Gamma)$.


\begin{defn}[cf. \cite{E-ARIGARI}, \cite{E-flex}]\label{def:ARIalal}
The subset 
$\ARI(\Gamma)_{\underline\al/\underline\al}$
of bialternal moulds
is defined to be
\begin{align*}
	\ARI(\Gamma)_{\underline\al/\underline\al}:=&\{M\in\ARI(\Gamma)_\al \ |\
	\swap(M)\in\overline{\ARI}(\Gamma)_\al, \  M^1\varia{x_1}{\sigma_1}=M^1\varia{-x_1}{\sigma_1^{-1}}\}.
\end{align*}
\end{defn}

Here are observations on Example \ref{ex:alternal mould}.
\begin{eg}
(a).
When $f(x)\in\mathbb Q(x)$ is an even function, 
the mould $M_f$ in Example \ref{ex:alternal mould}.(a) 
is in $\ARI(\Gamma)_{\underline\al/\underline\al}$.

(b).
However, the mould $A$ in Example \ref{ex:alternal mould}.(b) is not in $\ARI(\Gamma)_{\underline\al/\underline\al}$.
In fact, we have
$$
\swap(A)^2
{\scriptsize\left(\begin{array}{cc}
	e,& e \\
	v_1,& v_2
\end{array}\right)}
=A^2
{\scriptsize\left(\begin{array}{cc}
	v_2,& v_1-v_2 \\
	e,& e
\end{array}\right)}
=\frac{1}{v_1-2v_2},
$$
and we get
\begin{align*}
\sum_{\alpha\in Y_\Z^\bullet}
\Sh{\varia{e}{v_1}}{\varia{e}{v_2}}{\alpha}\swap(A)^2&(\alpha)
=\swap(A)^2
{\scriptsize\left(\begin{array}{cc}
	e,& e \\
	v_1,& v_2
\end{array}\right)}
+\swap(A)^2
{\scriptsize\left(\begin{array}{cc}
	e,& e \\
	v_2,& v_1
\end{array}\right)} \\
&=\frac{1}{v_1-2v_2}+\frac{1}{v_2-2v_1}
=-\frac{v_1+v_2}{(v_1-2v_2)(v_2-2v_1)}
\neq 0.
\end{align*}
Hence, $\swap(A)$ is not alternal, which says 
$A\not\in\ARI(\Gamma)_{\underline\al/\underline\al}$.
\end{eg}

\begin{prop}
\label{ARIalal Lie algebra}
The $\mathbb Q$-linear space $\ARI(\Gamma)_{\underline\al/\underline\al}$ 
 forms a filtered Lie subalgebra of $\ARI(\Gamma)_\al$ under the $\ari_u$-bracket.
\end{prop}
\begin{proof}
	We prove this in Appendix \ref{sec:A.3}.
\end{proof}

\subsection{Push-invariance and pus-neutrality}\label{subsec:push-invariance and pus-neutrality}
We introduce the set  $\ARI(\Gamma)_{\push/\pusnu}$ of moulds which are push-invariant
and whose $\swap$ are pus-neutral and
show that it forms a Lie algebra under the ari-bracket.

By abuse of notation, we introduce the following notations for $\overline{\ARI}(\Gamma)$
similarly to Notation \ref{nota:mould operations}.
\begin{nota}[{\cite[\S 2.1]{E-flex}}]
For any mould $M=\left( M^m\varia{\sigma_1,\ \dots,\ \sigma_m}{v_1,\ \dots,\ v_m}\right)_m\in\overline{\ARI}(\Gamma)$,  we define 
\begin{align*}
&\mantar( M)^m
{\scriptsize\left(\begin{array}{rrr}
	\sigma_1,& \dots,& \sigma_m \\
	v_1,& \dots,& v_m
\end{array}\right)}
	=(-1)^{m-1} M^m
	{\scriptsize\left(\begin{array}{rrr}
		\sigma_m,& \dots,& \sigma_1 \\
		v_m,& \dots,& v_1
	\end{array}\right)},\\
&\pus( M)^m
{\scriptsize\left(\begin{array}{rrr}
	\sigma_1,& \dots,& \sigma_m \\
	v_1,& \dots,& v_m
\end{array}\right)}
	= M^m
	{\scriptsize\left(\begin{array}{rrrr}
		\sigma_m,& \sigma_1,& \dots,& \sigma_{m-1} \\
		v_m,& v_1,& \dots,& v_{m-1}
\end{array}\right)},\\
&\nega( M)^m
{\scriptsize\left(\begin{array}{rrr}
	\sigma_1,& \dots,& \sigma_m \\
	v_1,& \dots,& v_m
\end{array}\right)}
	= M^m
	{\scriptsize\left(\begin{array}{rrr}
		\sigma_1^{-1},& \dots,& \sigma_m^{-1} \\
		-u_1,& \dots,& -u_m
\end{array}\right)}.
\end{align*}
\end{nota}

\begin{defn}
\label{def:ARIpushpusnu}
We call a mould $N\in\overline{\ARI}(\Gamma)$ {\it pus-neutral} (\cite[(2.73)]{E-flex})
\footnote{
It is not push-neutral but pus-neutral.
In \cite[Definition 5']{RS},  it is called {\it circ-neutral} when \eqref{pus-neutral} holds for $m>1$ when $\Gamma=\{e\}$.
}
when we have
\begin{equation}\label{pus-neutral}
	\sum_{i=1}^{m}\pus^i(N)^m\varia{\sigma_1,\ \dots,\ \sigma_m}{v_1,\ \dots,\ v_m}
	\left(=\sum_{i\in\Z/m\Z} N^m\varia{\sigma_{i+1},\ \dots,\ \sigma_{i+m}}{v_{i+1},\ \dots,\ v_{i+m}}\right)=0
\end{equation}
for all $m\geqslant1$ and $\sigma_1,\dots,\sigma_m\in\Gamma$.
We define  $\ARI(\Gamma)_{\push/ \pusnu}$
\footnote{
We expect that our $\ARI(\Gamma)_{\push/ \pusnu}$
might be related to Ecalle's $\ARI^{\ast/\overline{\pusnu}}$
(\cite[\S 2.5]{E-flex}),
whose precise definition looks missing.
}
to be the set  of    
moulds $M\in \ARI(\Gamma)$  which is push-invariant \eqref{push-invariant} and
whose $\swap (M)$ is pus-neutral  \eqref{pus-neutral}.
\end{defn}

\begin{eg}
We consider the mould $Q\in \overline{\ARI}(\{e\})$ defined by
\begin{equation*}
	Q{\scriptsize\left(\begin{array}{ccc}
	e,& \dots,& e \\
	v_1,& \dots,& v_m
\end{array}\right)}
	:=\left\{\begin{array}{ll}
		0 & (m=0,1), \\
		\frac{1}{v_2-v_1} + \cdots + \frac{1}{v_m-v_1} & (m\geqslant2).
	\end{array}\right.
\end{equation*}
Then this mould $Q$ is $\pus$-neutral.
In fact, for $1\leqslant i\leqslant m$, we have
$$
Q{\scriptsize\left(\begin{array}{cccccc}
	e,& \dots,& e, &e,& \dots,& e \\
	v_i,& \dots,& v_m,&v_1,& \dots,& v_{i-1}
\end{array}\right)}
=\sum_{j=i+1}^m\frac{1}{v_j-v_i} + \sum_{j=1}^{i-1}\frac{1}{v_j-v_i},
$$
and so we calculate
\begin{align*}
\sum_{i=1}^m\pus^i(Q){\scriptsize\left(\begin{array}{ccc}
	e,& \dots,& e \\
	v_1,& \dots,& v_m
\end{array}\right)}
&=\sum_{i=1}^mQ{\scriptsize\left(\begin{array}{cccccc}
	e,& \dots,& e, &e,& \dots,& e \\
	v_i,& \dots,& v_m,&v_1,& \dots,& v_{i-1}
\end{array}\right)} \\
&=\sum_{i=1}^m\left\{ \sum_{j=i+1}^m\frac{1}{v_j-v_i} + \sum_{j=1}^{i-1}\frac{1}{v_j-v_i} \right\} \\
&=\sum_{\substack{
	i,j\in \{1,\dots,m\} \\
	i<j}}
\frac{1}{v_j-v_i}
+ \sum_{\substack{
	i,j\in \{1,\dots,m\} \\
	i>j}}
\frac{1}{v_j-v_i}
=0.
\end{align*}
Hence, the mould $Q$ is $\pus$-neutral.
\end{eg}

Firstly we show that the set $\ARI_{\push}$ forms a Lie algebra under the $\ari_u$-bracket
which was stated in \cite[\S 2.5]{E-flex} without a detailed proof.

\begin{prop}\label{prop:ARIpush Lie}
	If $ A, B\in\ARI(\Gamma)$ are $\push$-invariant, then $\ari_u( A, B)$ is $\push$-invariant.
\end{prop}
\begin{proof}
	Let $m\geqslant1$. We put $\omega=\varia{u_1,\ \dots,\ u_m}{\sigma_1,\ \dots,\ \sigma_m}$. We have
	\begin{align*}
		\push&( A\times B)(\omega) \\
		=&( A\times B){\scriptsize\left(\begin{array}{rrrr}
			-u_1-\cdots-u_m,& u_1,& \dots,& u_{m-1} \\
			\sigma_m^{-1},&\sigma_1\sigma_m^{-1},& \dots,& \sigma_{m-1}\sigma_m^{-1}
		\end{array}\right)} \\
		=& A(\emptyset) B{\scriptsize\left(\begin{array}{rrrr}
			-u_1-\cdots-u_m,& u_1,& \dots,& u_{m-1} \\
			\sigma_m^{-1},&\sigma_1\sigma_m^{-1},& \dots,& \sigma_{m-1}\sigma_m^{-1}
		\end{array}\right)} \\
		&+\sum_{i=1}^m A{\scriptsize\left(\begin{array}{rrrr}
			-u_1-\cdots-u_m,& u_1,& \dots,& u_{i-1} \\
			\sigma_m^{-1},&\sigma_1\sigma_m^{-1},& \dots,& \sigma_{i-1}\sigma_m^{-1}
		\end{array}\right)}
		 B{\scriptsize\left(\begin{array}{rrr}
			u_i,& \dots,& u_{m-1} \\
			\sigma_i\sigma_m^{-1},& \dots,& \sigma_{m-1}\sigma_m^{-1}
		\end{array}\right)}.
	\end{align*}
	Because $ A,  B$ are $\push$-invariant, we get
	\begin{align}\label{push on product}
		\push&( A\times B)(\omega) \\
		=& A(\emptyset) B(\omega)
		+ A(\omega) B(\emptyset) \nonumber \\
		&+\sum_{i=1}^{m-1} 
		A{\scriptsize\left(\begin{array}{rrrr}
			u_1,& \dots,& u_{i-1},& u_i+\cdots+u_m \\
			\sigma_1,& \dots,& \sigma_{i-1},& \sigma_m
		\end{array}\right)}
		B{\scriptsize\left(\begin{array}{rrr}
			u_i,& \dots,& u_{m-1} \\
			\sigma_i\sigma_m^{-1},& \dots,& \sigma_{m-1}\sigma_m^{-1}
		\end{array}\right)}. \nonumber
	\end{align}
	On the other hand, for $\eta=\varia{u_0,\ u_1,\ \dots,\ u_{m-1}}{\tau_0,\ \tau_1,\ \dots,\ \tau_{m-1}}$, we have
	\begin{align*}
		&\arit_u( B)( A)(\eta) \\
		=&\sum_{\substack{\eta=\alpha\beta\gamma \\
			\beta,\gamma\neq\emptyset}}
		 A(\alpha\urflex{\beta}{\gamma}) B(\llflex{\beta}{\gamma})
		-\sum_{\substack{\eta=\alpha\beta\gamma \\
			\alpha,\beta\neq\emptyset}}
		 A(\ulflex{\alpha}{\beta}\gamma) B(\lrflex{\alpha}{\beta}) \\
		=&\sum_{0\leqslant i\leqslant j<m-1}
			 A{\scriptsize\left(\begin{array}{rrrcrrr}
			 	u_0,& \dots,& u_{i-1},& u_i+\cdots+u_{j+1},& u_{j+2},& \dots,& u_{m-1} \\
			 	\tau_0,& \dots,& \tau_{i-1},& \tau_{j+1},& \tau_{j+2},& \dots,& \tau_{m-1}
			\end{array}\right)}
			B{\scriptsize\left(\begin{array}{rrr}
				u_i,& \dots,& u_j \\
				\tau_i\tau_{j+1}^{-1},& \dots,& \tau_j\tau_{j+1}^{-1} \\
			\end{array}\right)} \\
		&-\sum_{1\leqslant i\leqslant j\leqslant m-1}
			 A{\scriptsize\left(\begin{array}{rrrcrrr}
			 	u_0,& \dots,& u_{i-2},& u_{i-1}+\cdots+u_{j},& u_{j+1},& \dots,& u_{m-1} \\
			 	\tau_0,& \dots,& \tau_{i-2},& \tau_{i-1},& \tau_{j+1},& \dots,& \tau_{m-1}
			\end{array}\right)}
			 B{\scriptsize\left(\begin{array}{rrr}
				u_i,& \dots,& u_j \\
				\tau_i\tau_{i-1}^{-1},& \dots,& \tau_j\tau_{i-1}^{-1} \\
			\end{array}\right)} \\
		=&\sum_{0\leqslant j<m-1}
			 A{\scriptsize\left(\begin{array}{rrrr}
			 	u_0+\cdots+u_{j+1},& u_{j+2},& \dots,& u_{m-1} \\
			 	\tau_{j+1},& \tau_{j+2},& \dots,& \tau_{m-1}
			\end{array}\right)}
			B{\scriptsize\left(\begin{array}{rrr}
				u_0,& \dots,& u_j \\
				\tau_0\tau_{j+1}^{-1},& \dots,& \tau_j\tau_{j+1}^{-1} \\
			\end{array}\right)} \\
		&-\sum_{1\leqslant j\leqslant m-1}
			 A{\scriptsize\left(\begin{array}{rrrr}
			 	u_0+\cdots+u_{j},& u_{j+1},& \dots,& u_{m-1} \\
			 	\tau_{0},& \tau_{j+1},& \dots,& \tau_{m-1}
			\end{array}\right)}
			 B{\scriptsize\left(\begin{array}{rrr}
				u_1,& \dots,& u_j \\
				\tau_1\tau_{0}^{-1},& \dots,& \tau_j\tau_{0}^{-1} \\
			\end{array}\right)} \\
		&+\sum_{0< i\leqslant j<m-1}
			 A{\scriptsize\left(\begin{array}{rrrcrrr}
			 	u_0,& \dots,& u_{i-1},& u_i+\cdots+u_{j+1},& u_{j+2},& \dots,& u_{m-1} \\
			 	\tau_0,& \dots,& \tau_{i-1},& \tau_{j+1},& \tau_{j+2},& \dots,& \tau_{m-1}
			\end{array}\right)}
			B{\scriptsize\left(\begin{array}{rrr}
				u_i,& \dots,& u_j \\
				\tau_i\tau_{j+1}^{-1},& \dots,& \tau_j\tau_{j+1}^{-1} \\
			\end{array}\right)} \\
		&-\sum_{1< i\leqslant j\leqslant m-1}
			 A{\scriptsize\left(\begin{array}{rrrcrrr}
			 	u_0,& \dots,& u_{i-2},& u_{i-1}+\cdots+u_{j},& u_{j+1},& \dots,& u_{m-1} \\
			 	\tau_0,& \dots,& \tau_{i-2},& \tau_{i-1},& \tau_{j+1},& \dots,& \tau_{m-1}
			\end{array}\right)}
			 B{\scriptsize\left(\begin{array}{rrr}
				u_i,& \dots,& u_j \\
				\tau_i\tau_{i-1}^{-1},& \dots,& \tau_j\tau_{i-1}^{-1} \\
			\end{array}\right)}.
	\end{align*}
	So by putting $u_0=-u_1-\cdots-u_m$ and $\tau_0=\sigma_m^{-1}$ and $\tau_i=\sigma_i\sigma_m^{-1}$ for $1\leqslant i\leqslant m-1$, we calculate
	\begin{align*}
		&\push(\arit_u( B)( A))(\omega) \\
		=&(\arit_u( B)( A)){\scriptsize\left(\begin{array}{rrrr}
			-u_1-\cdots-u_m,& u_1,& \dots,& u_{m-1} \\
			\sigma_m^{-1},&\sigma_1\sigma_m^{-1},& \dots,& \sigma_{m-1}\sigma_m^{-1}
		\end{array}\right)} \\
		=&\sum_{0\leqslant j<m-1}
			 A{\scriptsize\left(\begin{array}{rrrr}
			 	-u_{j+2}-\cdots-u_m,& u_{j+2},& \dots,& u_{m-1} \\
			 	\sigma_{j+1}\sigma_m^{-1},& \sigma_{j+2}\sigma_m^{-1},& \dots,& \sigma_{m-1}\sigma_m^{-1}
			\end{array}\right)} \\
			&\hspace{5cm}\cdot B{\scriptsize\left(\begin{array}{rrrr}
				-u_1-\cdots-u_m,& u_1,& \dots,& u_j \\
				\sigma_{j+1}^{-1},& \sigma_1\sigma_{j+1}^{-1},& \dots,& \sigma_j\sigma_{j+1}^{-1} \\
			\end{array}\right)} \\
		&-\sum_{1\leqslant j\leqslant m-1}
			 A{\scriptsize\left(\begin{array}{rrrr}
			 	-u_{j+1}-\cdots-u_m,& u_{j+1},& \dots,& u_{m-1} \\
			 	\sigma_m^{-1},& \sigma_{j+1}\sigma_m^{-1},& \dots,& \sigma_{m-1}\sigma_m^{-1}
			\end{array}\right)} \\
			&\hspace{5cm}\cdot B{\scriptsize\left(\begin{array}{rrr}
				u_1,& \dots,& u_j \\
				\sigma_1,& \dots,& \sigma_j \\
			\end{array}\right)} \\
		&+\sum_{0< i\leqslant j<m-1}
			 A{\scriptsize\left(\begin{array}{rrrrrrrr}
			 	-u_1-\cdots-u_m,& u_1,& \dots,& u_{i-1},& u_i+\cdots+u_{j+1},& u_{j+2},& \dots,& u_{m-1} \\
			 	\sigma_m^{-1},& \sigma_1\sigma_m^{-1},& \dots,& \sigma_{i-1}\sigma_m^{-1},& \sigma_{j+1}\sigma_m^{-1},& \sigma_{j+2}\sigma_m^{-1},& \dots,& \sigma_{m-1}\sigma_m^{-1}
			\end{array}\right)} \\
			&\hspace{5cm}\cdot B{\scriptsize\left(\begin{array}{rrr}
				u_i,& \dots,& u_j \\
				\sigma_i\sigma_{j+1}^{-1},& \dots,& \sigma_j\sigma_{j+1}^{-1} \\
			\end{array}\right)} \\
		&-\sum_{1< i\leqslant j\leqslant m-1}
			 A{\scriptsize\left(\begin{array}{rrrrrrrr}
			 	-u_1-\cdots-u_m,& u_1,& \dots,& u_{i-2},& u_{i-1}+\cdots+u_{j},& u_{j+1},& \dots,& u_{m-1} \\
			 	\sigma_m^{-1},& \sigma_1\sigma_m^{-1},& \dots,& \sigma_{i-2}\sigma_m^{-1},& \sigma_{i-1}\sigma_m^{-1},& \sigma_{j+1}\sigma_m^{-1},& \dots,& \sigma_{m-1}\sigma_m^{-1}
			\end{array}\right)} \\
			&\hspace{5cm}\cdot B{\scriptsize\left(\begin{array}{rrr}
				u_i,& \dots,& u_j \\
				\sigma_i\sigma_{i-1}^{-1},& \dots,& \sigma_j\sigma_{i-1}^{-1} \\
			\end{array}\right)}.
		\intertext{Because $ A, B$ are $\push$-invariant, we get}
		=&\sum_{0\leqslant j<m-1}
			 A{\scriptsize\left(\begin{array}{rrr}
			 	u_{j+2},& \dots,& u_m \\
			 	\sigma_{j+2}\sigma_{j+1}^{-1},& \dots,& \sigma_{m}\sigma_{j+1}^{-1}
			 \end{array}\right)}
			 B{\scriptsize\left(\begin{array}{rrrr}
			 	u_1,& \dots,& u_j,& u_{j+1}+\cdots+u_m \\
			 	\sigma_1,& \dots,& \sigma_j,& \sigma_{j+1}
			 \end{array}\right)} \\
		&-\sum_{1\leqslant j\leqslant m-1}
			 A{\scriptsize\left(\begin{array}{rrr}
			 	u_{j+1},& \dots,& u_m \\
			 	\sigma_{j+1},& \dots,& \sigma_{m}
			 \end{array}\right)}
			 B{\scriptsize\left(\begin{array}{rrr}
				u_1,& \dots,& u_j \\
				\sigma_1,& \dots,& \sigma_j \\
			\end{array}\right)} \\
		&+\sum_{0< i\leqslant j<m-1}
			A{\scriptsize\left(\begin{array}{rrrcrrr}
				u_1,& \dots,& u_{i-1},& u_i+\cdots+u_{j+1},& u_{j+2},& \dots,& u_m \\
				\sigma_1,& \dots,& \sigma_{i-1},& \sigma_{j+1},& \sigma_{j+2},& \dots,& \sigma_m
			\end{array}\right)}
			B{\scriptsize\left(\begin{array}{rrr}
				u_i,& \dots,& u_j \\
				\sigma_i\sigma_{j+1}^{-1},& \dots,& \sigma_j\sigma_{j+1}^{-1}
			\end{array}\right)} \\
		&-\sum_{1< i\leqslant j\leqslant m-1}
			 A{\scriptsize\left(\begin{array}{rrrcrrr}
				u_1,& \dots,& u_{i-2},& u_{i-1}+\cdots+u_j,& u_{j+1},& \dots,& u_m \\
				\sigma_1,& \dots,& \sigma_{i-2},& \sigma_{j-1},& \sigma_{j+1},& \dots,& \sigma_m
			\end{array}\right)}
			 B{\scriptsize\left(\begin{array}{rrr}
				u_i,& \dots,& u_j \\
				\sigma_i\sigma_{i-1}^{-1},& \dots,& \sigma_j\sigma_{i-1}^{-1} \\
			\end{array}\right)} \\
		=&\sum_{\substack{\omega=\alpha\beta \\
			\alpha,\beta\neq\emptyset}}
			 A(\lrflex{\alpha}{\beta})
			 B(\ulflex{\alpha}{\beta}) 
		-( B\times A)\varia{u_1,\ \dots,\ u_m}{\sigma_1,\ \dots,\ \sigma_m} \\
		&+\sum_{\substack{\omega=\alpha\beta\gamma \\
			\beta,\gamma\neq\emptyset}}
			 A(\alpha\urflex{\beta}{\gamma})
			 B(\llflex{\beta}{\gamma}) 
		-\sum_{0<i\leqslant m-1} A{\scriptsize\left(\begin{array}{rrrr}
			u_1,& \dots,& u_{i-1},& u_i+\cdots+u_m \\
			\sigma_1,& \dots,& \sigma_{i-1},& \sigma_m
			\end{array}\right)} \\
			&\hspace{8.5cm}\cdot B{\scriptsize\left(\begin{array}{rrr}
				u_i,& \dots,& u_{m-1} \\
				\sigma_i\sigma_{m}^{-1},& \dots,& \sigma_j\sigma_{m}^{-1}
			\end{array}\right)} \\
		&-\left\{\sum_{\substack{\omega=\alpha\beta\gamma \\
			\alpha,\beta\neq\emptyset}}
			 A(\ulflex{\alpha}{\beta}\gamma)
			 B(\lrflex{\alpha}{\beta})\right.
		-\left.\sum_{\substack{\omega=\alpha\beta \\
			\alpha,\beta\neq\emptyset}}
			 A(\ulflex{\alpha}{\beta})
			 B(\lrflex{\alpha}{\beta})\right\}.
	\end{align*}
	So we get
	\begin{align}\label{push on arit}
	 \push&(\arit_u( B)( A))(\omega) \\
		=&(\arit_u( B)( A))(\omega) \nonumber
		-( B\times A)(\omega) \\
		&-\sum_{i=1}^{m-1} A{\scriptsize\left(\begin{array}{rrrr}
				u_1,& \dots,& u_{i-1},& u_i+\cdots+u_m \\
				\sigma_1,& \dots,& \sigma_{i-1},& \sigma_m
			\end{array}\right)}
			B{\scriptsize\left(\begin{array}{rrr}
				u_i,& \dots,& u_{m-1} \\
				\sigma_i\sigma_{m}^{-1},& \dots,& \sigma_j\sigma_{m}^{-1}
			\end{array}\right)} \nonumber\\
		&+\sum_{\substack{\omega=\alpha\beta \nonumber\\
			\alpha,\beta\neq\emptyset}}
			\bigl\{ A(\ulflex{\alpha}{\beta})
			 B(\lrflex{\alpha}{\beta})
		+ A(\lrflex{\alpha}{\beta})
		 B(\ulflex{\alpha}{\beta})\bigr\}. \nonumber
	\end{align}
	Therefore, by using \eqref{push on product} and \eqref{push on arit}, we have
	\begin{align*}
		\push(\ari_u( A, B))(\omega)
		=&\push(\arit_u( B)( A))(\omega)-\push(\arit_u( A)( B))(\omega) \\
		&\quad +\push( A\times  B)(\omega)-\push( B\times  A)(\omega) \\
		=&(\arit_u( B)( A))(\omega)-(\arit_u( A)( B))(\omega) \\
		&\quad +( A\times  B)(\omega)-( B\times  A)(\omega) \\
		=&\ari_u( A, B)(\omega).
	\end{align*}
	Thus $\ari_u( A, B)$ is $\push$-invariant.
\end{proof}

\begin{rem}
This proposition improves the results of \cite[\S 4.1.3]{RS},
which shows that the intersection $\ARI_{\push}(\Gamma)\cap\ARI_\al^\pol(\Gamma)$
forms a Lie algebra under the $\ari_u$-bracket when $\Gamma=\{e\}$.
\end{rem}

Secondly we show that the set of pus-neutral moulds is closed under the $\ari_v$-bracket.
The following proposition is a generalization of \cite[Lemma 21]{RS} which treats the case when $\Gamma =\{e\}$.

\begin{prop}\label{pus-neutral under ari}
	If $ A, B\in\overline{\ARI}(\Gamma)$ are $\pus$-neutral, then $\ari_v( A, B)$ is $\pus$-neutral.
\end{prop}

\begin{proof}
The proof goes in the same way to that of \cite{RS}. 
Let $m\geqslant1$. 
Since the algebra $\overline{\ARI}(\Gamma)$ is graded by depth,
it is enough to prove  $A$ and $B\in \ARI(\Gamma)$ with
depth $k$ and $l\in\N$ with $m=k+l$.
	
	Firstly, we prove that $[ A, B]$ is $\pus$-neutral. We calculate
	\begin{align*}
		&\sum_{i=1}^{m}\pus^i([ A, B])^m\varia{\sigma_1,\ \dots,\ \sigma_m}{v_1,\ \dots,\ v_m}
		=\sum_{i\in\Z/m\Z}( A\times  B- B\times  A)^m\varia{\sigma_{1+i},\ \dots,\ \sigma_{m+i}}{v_{1+i},\ \dots,\ v_{m+i}} \\
		&=\sum_{i\in\Z/m\Z}
		\left\{
			A^k\varia{\sigma_{1+i},\ \dots,\ \sigma_{k+i}}{v_{1+i},\ \dots,\ v_{k+i}} 
			B^l\varia{\sigma_{k+1+i},\ \dots,\ \sigma_{m+i}}{v_{k+1+i},\ \dots,\ v_{m+i}} 
			-B^l\varia{\sigma_{1+i},\ \dots,\ \sigma_{l+i}}{v_{1+i},\ \dots,\ v_{l+i}} 
			A^k\varia{\sigma_{l+1+i},\ \dots,\ \sigma_{m+i}}{v_{l+1+i},\ \dots,\ v_{m+i}}
		\right\} \\
		&=\sum_{i\in\Z/m\Z}
		\left\{
			A^k\varia{\sigma_{1+i},\ \dots,\ \sigma_{k+i}}{v_{1+i},\ \dots,\ v_{k+i}} 
			B^l\varia{\sigma_{k+1+i},\ \dots,\ \sigma_{m+i}}{v_{k+1+i},\ \dots,\ v_{m+i}} 
			-B^l\varia{\sigma_{1+i-l},\ \dots,\ \sigma_{i}}{v_{1+i-l},\ \dots,\ v_{i}} 
			A^k\varia{\sigma_{1+i},\ \dots,\ \sigma_{k+i}}{v_{1+i},\ \dots,\ v_{k+i}}
		\right\} \\
		&=\sum_{i\in\Z/m\Z}
		\left\{
			A^k\varia{\sigma_{1+i},\ \dots,\ \sigma_{k+i}}{v_{1+i},\ \dots,\ v_{k+i}} 
			B^l\varia{\sigma_{k+1+i},\ \dots,\ \sigma_{m+i}}{v_{k+1+i},\ \dots,\ v_{m+i}} 
			-B^l\varia{\sigma_{k+1+i},\ \dots,\ \sigma_{m+i}}{v_{k+1+i},\ \dots,\ v_{m+i}} 
			A^k\varia{\sigma_{1+i},\ \dots,\ \sigma_{k+i}}{v_{1+i},\ \dots,\ v_{k+i}}
		\right\}
	\end{align*}
	The first and the second terms are equal, so $[ A, B]$ is $\pus$-neutral.
	
	Secondly, we prove that $\arit_v( B)( A)$ is $\pus$-neutral. Then we have
	\begin{align*}
		&(\arit_v( B)( A))^m\varia{\sigma_1,\ \dots,\ \sigma_m}{v_1,\ \dots,\ v_m}
		=\sum_{\substack{\omega=\alpha\beta\gamma \\
		\beta,\gamma\neq\emptyset}}
		 A^{k}(\alpha\urflex{\beta}{\gamma}) B^{l}(\llflex{\beta}{\gamma})
		-\sum_{\substack{\omega=\alpha\beta\gamma \\
		\alpha,\beta\neq\emptyset}}
		 A^{k}(\ulflex{\alpha}{\beta}\gamma) B^{l}(\lrflex{\alpha}{\beta}).
	\end{align*}
	Because $l\geqslant1$ and $B$ is with depth $l$, we may put
	\footnote{For the first (resp. second) term, the word $\alpha$ (resp. $\gamma$) can be $\emptyset$ and the word $\gamma$ (resp. $\alpha$) is not $\emptyset$. So the index $j$ must run from $0$ (resp. $1$) to $k-1=m-l-1$ (resp. $k=m-l$).}
	$\alpha=\varia{\sigma_1,\ \dots,\ \sigma_{j}}{v_1,\ \dots,\ v_{j}}$,
	$\beta=\varia{\sigma_{j+1},\ \dots,\ \sigma_{j+l}}{v_{j+1},\ \dots,\ v_{j+l}}$ and
	$\gamma=\varia{\sigma_{j+l+1},\ \dots,\ \sigma_{m}}{v_{j+l+1},\ \dots,\ v_{m}}$ and we get
	\begin{align*}
		&(\arit_v( B)( A))^m\varia{\sigma_1,\ \dots,\ \sigma_m}{v_1,\ \dots,\ v_m} \\
		&=\sum_{j=0}^{k-1}
		A^{k}{\scriptsize\left(\begin{array}{rrrcrrr}
			\sigma_1,& \dots,& \sigma_{j},& \sigma_{j+1}\cdots\sigma_{j+l+1},& \sigma_{j+l+2},& \dots,& \sigma_{m} \\
			v_1,& \dots,& v_{j},& v_{j+l+1},& v_{j+l+2},& \dots,& v_{m}
		\end{array}\right)}
		B^{l}{\scriptsize\left(\begin{array}{rrr}
			\sigma_{j+1},& \dots,& \sigma_{j+l} \\
			v_{j+1}-v_{j+l+1},& \dots,& v_{j+l}-v_{j+l+1} \\
		\end{array}\right)} \\
		&\quad-\sum_{j=1}^{k}
		A^{k}{\scriptsize\left(\begin{array}{rrrcrrr}
		 	\sigma_1,& \dots,& \sigma_{j-1},& \sigma_{j}\cdots\sigma_{j+l},& \sigma_{j+l+1},& \dots,& \sigma_{m} \\
		 	v_1,& \dots,& v_{j-1},& v_{j},& v_{j+l+1},& \dots,& v_{m}
		\end{array}\right)}
		B^{l}{\scriptsize\left(\begin{array}{rrr}
			\sigma_{j+1},& \dots,& \sigma_{j+l} \\
			v_{j+1}-v_{j},& \dots,& v_{j+l}-v_{j} \\
		\end{array}\right)}.
	\end{align*}
	Here, for $0\leqslant j\leqslant k-1$, we have
	\begin{align*}
		&\sum_{i\in\Z/m\Z}\sum_{j=0}^{k-1}
		A^{k}{\scriptsize\left(\begin{array}{rrrcrrr}
			\sigma_{i+1},& \dots,& \sigma_{i+j},& \sigma_{i+j+1}\cdots\sigma_{i+j+l+1},& \sigma_{i+j+l+2},& \dots,& \sigma_{i+m} \\
			v_{i+1},& \dots,& v_{i+j},& v_{i+j+l+1},& v_{i+j+l+2},& \dots,& v_{i+m}
		\end{array}\right)} \\
		&\hspace{5cm}\cdot
		B^{l}{\scriptsize\left(\begin{array}{rrr}
			\sigma_{i+j+1},& \dots,& \sigma_{i+j+l} \\
			v_{i+j+1}-v_{i+j+l+1},& \dots,& v_{i+j+l}-v_{i+j+l+1} \\
		\end{array}\right)} \\
		&=\sum_{i\in\Z/m\Z}\sum_{j=0}^{k-1}
		A^{k}{\scriptsize\left(\begin{array}{rrrcrrr}
			\sigma_{i+1-j},& \dots,& \sigma_{i},& \sigma_{i+1}\cdots\sigma_{i+l+1},& \sigma_{i+l+2},& \dots,& \sigma_{i+m-j} \\
			v_{i+1-j},& \dots,& v_{i},& v_{i+l+1},& v_{i+l+2},& \dots,& v_{i+m-j}
		\end{array}\right)} \\
		&\hspace{5cm}\cdot
		B^{l}{\scriptsize\left(\begin{array}{rrr}
			\sigma_{i+1},& \dots,& \sigma_{i+l} \\
			v_{i+1}-v_{i+l+1},& \dots,& v_{i+l}-v_{i+l+1} \\
		\end{array}\right)}. \\
		&=\sum_{i\in\Z/m\Z}\sum_{j=0}^{k-1}
		(\pus^j( A))^{k}{\scriptsize\left(\begin{array}{crrr}
			\sigma_{i+1}\cdots\sigma_{i+l+1},& \sigma_{i+l+2},& \dots,& \sigma_{i+m} \\
			v_{i+l+1},& v_{i+l+2},& \dots,& v_{i+m}
		\end{array}\right)} \\
		&\hspace{5cm}\cdot
		B^l{\scriptsize\left(\begin{array}{rrr}
			\sigma_{i+1},& \dots,& \sigma_{i+l} \\
			v_{i+1}-v_{i+l+1},& \dots,& v_{i+l}-v_{i+l+1} \\
		\end{array}\right)}.
	\end{align*}
	Since $ A$ is $\pus$-neutral, this is 0.
Similarly by using the $\pus$-neutrality of $ A$, we obtain 
	\begin{align*}
		&\sum_{i\in\Z/m\Z}\sum_{j=1}^{k}
		A^{k}{\scriptsize\left(\begin{array}{rrrcrrr}
		 	\sigma_{i+1},& \dots,& \sigma_{i+j-1},& \sigma_{i+j}\cdots\sigma_{i+j+l},& \sigma_{i+j+l+1},& \dots,& \sigma_{i+m} \\
		 	v_{i+1},& \dots,& v_{i+j-1},& v_{i+j},& v_{i+j+l+1},& \dots,& v_{i+m}
		\end{array}\right)} \\
		&\hspace{5cm}\cdot
		B^{l}{\scriptsize\left(\begin{array}{rrr}
			\sigma_{i+j+1},& \dots,& \sigma_{i+j+l} \\
			v_{i+j+1}-v_{i+j},& \dots,& v_{i+j+l}-v_{i+j} \\
		\end{array}\right)}=0
	\end{align*}
	for all $1\leqslant j\leqslant k$.
	Therefore, we get
	\begin{align*}
       \sum_{i=1}^{m}\pus^i(\arit_v( B)( A))^m\varia{\sigma_1,\ \dots,\ \sigma_m}{v_1,\ \dots,\ v_m}
       =\sum_{i\in\Z/m\Z}(\arit_v( B)( A))^m\varia{\sigma_{1+i},\ \dots,\ \sigma_{m+i}}{v_{1+i},\ \dots,\ v_{m+i}}=0.
	\end{align*}
Whence  $\arit_v( B)( A)$ is $\pus$-neutral.

	Thirdly,  in the same way, we can show that $\arit_v( A)( B)$ is $\pus$-neutral by using the $\pus$-neutrality of $ B$. 
	
	Thus we complete the proof because  $\ari_v( A, B)=[ A, B]+\arit_v( B)( A)-\arit_v( A)( B)$.
\end{proof}

We need the following lemma for the proof of Theorem \ref{thm: ARIpushpusnu Lie}.

\begin{lem}
\label{swap-ari commutation}
	If $ A, B\in\ARI(\Gamma)$ are $\push$-invariant, then we have
	$$\swap(\ari_u( A, B))=\ari_v(\swap( A),\swap( B)).$$
\end{lem}

\begin{proof}
This  follows from the proof of  \cite[Lemma 2.4.1]{S-ARIGARI},
which actually works for $\mathrm{BARI}$.
\end{proof}

\begin{thm}\label{thm: ARIpushpusnu Lie}
	The set $\ARI(\Gamma)_{\push/ \pusnu}$ forms a Lie subalgebra of $\ARI(\Gamma)_{\push}$ under the $\ari_u$-bracket.
\end{thm}

\begin{proof}
    Let $ A, B \in \ARI(\Gamma)_{\push/\pusnu}$. Then by Proposition \ref{prop:ARIpush Lie},
    $\ari_u( A, B)\in \ARI(\Gamma)_\push$.
	Let $m\geqslant1$. Because $ A$ and $ B$ are $\push$-invariant, by Lemma \ref{swap-ari commutation}, we have
	\begin{align*}
		\sum_{i\in\Z/m\Z}\pus^i&\circ\swap(\ari_u( A, B))^m
			\varia{\sigma_1,\ \dots,\ \sigma_m}{v_1,\ \dots,\ v_m} \\
		=&\sum_{i\in\Z/m\Z}\pus^i\bigl(\ari_v(\swap( A),\swap( B))\bigr)^m
			\varia{\sigma_1,\ \dots,\ \sigma_m}{v_1,\ \dots,\ v_m}.
	\end{align*}
	By Proposition \ref{pus-neutral under ari}, the right hand side is equal to 0. Hence $\swap(\ari_u( A, B))$ is pus-neutral.
	Thus we obtain $\ari_u(A, B)\in \ARI(\Gamma)_{\push/\pusnu}$.
\end{proof}

\begin{rem}
The above theorem improves \cite[Corollary 22]{RS} which shows that
$\ARI(\Gamma)_{\push/ \pusnu}\cap \ARI_\al^\pol(\Gamma)$
forms a Lie algebra  when $\Gamma=\{e\}$.
\end{rem}

Our Lie  algebra $\ARI(\Gamma)_{\push/ \pusnu}$ 
will play an important role in the following sections.

\begin{defn}\label{def:ARID}
For $N\geqslant 1$, put $\Gamma^N:=\{g^N\in\Gamma\ |\  g\in\Gamma\}$.
We consider the map $i_N:\ARI(\Gamma)\to\ARI(\Gamma^N)$
which  is  given by
\begin{equation}
i_N(M)^m\varia{x_1,\ \dots,\ x_m}{\sigma_1,\ \dots,\ \sigma_m}
=
M^m\varia{x_1,\ \dots,\ x_m}{\sigma_1,\ \dots,\ \sigma_m}
\end{equation}
and also the map $m_N:\ARI(\Gamma)\to\ARI(\Gamma^N)$
which  is  given by
\begin{equation}
m_N(M)^m\varia{x_1,\ \dots,\ x_m}{\sigma_1,\ \dots,\ \sigma_m}
=
\sum_{\tau_i^N=\sigma_i}
M^m\varia{Nx_1,\ \dots,\ Nx_m}{\quad \tau_1,\ \dots,\ \tau_m}.
\end{equation}
We define  the following  $\mathbb Q$-linear subspaces
which are subject to the distribution relations:
\begin{align*}
&\ARID(\Gamma):=\{M\in\ARI(\Gamma)\ |\ i_N(M)=m_N(M)
\text{ for all } N\geqslant 1
\text{ with }
N\bigm |  |\Gamma|
\}, \\
&\ARID(\Gamma)_{\underline\al/\underline\al}:=\ARID(\Gamma)\cap
\ARI(\Gamma)_{\underline\al/\underline\al}.
\end{align*}
\end{defn}

\begin{prop} \label{prop:ARID Lie algebra}
The $\mathbb Q$-linear subspace $\ARID(\Gamma)$ forms a filtered Lie subalgebra of 
$\ARI(\Gamma)$ under the $\ari_u$-bracket.
\end{prop}

\begin{proof}
First we prove that $m_N$ is a Lie algebra homomorphism, that is, 
\begin{equation*}
	m_N(\ari_u(A,B))=\ari_u(m_N(A),m_N(B))
\end{equation*}
for $A,B\in\ARI(\Gamma)$. Let $m\geqslant0$. We have
\begin{align*}
	m_N(A\times B)^m(\vecx_m)
	&= \sum_{\substack{
		\tau_i^N=\sigma_i \\
		1\leqslant i\leqslant m}}
	(A\times B)^m\varia{Nx_1,\ \dots,\ Nx_m}{\quad \tau_1,\ \dots,\ \tau_m} \\
	&= \sum_{\substack{
		\tau_i^N=\sigma_i \\
		1\leqslant i\leqslant m}}
	\sum_{j=0}^m
	A^i\varia{Nx_1,\ \dots,\ Nx_j}{\quad \tau_1,\ \dots,\ \tau_j}
	B^{m-i}\varia{Nx_{j+1},\ \dots,\ Nx_m}{\quad \tau_{j+1},\ \dots,\ \tau_m} \\
	&= \sum_{j=0}^m
	\left\{
	\sum_{\substack{
		\tau_i^N=\sigma_i \\
		1\leqslant i\leqslant j}}A^i\varia{Nx_1,\ \dots,\ Nx_j}{\quad \tau_1,\ \dots,\ \tau_j}
	\right\}
	\left\{
	\sum_{\substack{
		\tau_i^N=\sigma_i \\
		j+1\leqslant i\leqslant m}}	B^{m-i}\varia{Nx_{j+1},\ \dots,\ Nx_m}{\quad \tau_{j+1},\ \dots,\ \tau_m}
	\right\} \\
	&= \sum_{j=0}^m
	m_N(A)^i\varia{x_1,\ \dots,\ x_j}{\sigma_1,\ \dots,\ \sigma_j}
	m_N(B)^{m-i}\varia{x_{j+1},\ \dots,\ x_m}{\sigma_{j+1},\ \dots,\ \sigma_m} \\
	&= (m_N(A)\times m_N(B))^m(\vecx_m).
\end{align*}
Similarly, we have
\begin{align*}
	&m_N(\arit_u(B)(A))^m(\vecx_m) \\
	&=\sum_{\substack{
		\nu_i^N=\sigma_i \\
		1\leqslant i\leqslant m}}
	\arit(B)(A)^m\varia{Nx_1,\ \dots,\ Nx_m}{\quad \nu_1,\ \dots,\ \nu_m} \\
	&= \sum_{\substack{
		\nu_i^N=\sigma_i \\
		1\leqslant i\leqslant m}}
	\left\{
	\sum_{1\leqslant k\leqslant l<m}
	A{\scriptsize\left(\begin{array}{rrrcrrr}
	 	Nx_1,& \dots,& Nx_{k-1},& N(x_k+\cdots+x_{l+1}),& Nx_{l+2},& \dots,& Nx_{m} \\
	 	\nu_1,& \dots,& \nu_{k-1},& \nu_{l+1},& \nu_{l+2},& \dots,& \nu_{m}
	\end{array}\right)}
	\right. \\
	&\hspace{5cm}\cdot \left.
	\vphantom{\sum_{1\leqslant k\leqslant l<m}}
	B{\scriptsize\left(\begin{array}{rrr}
		Nx_k,& \dots,& Nx_l \\
		\nu_k\nu_{l+1}^{-1},& \dots,&\nu_l\nu_{l+1}^{-1} \\
	\end{array}\right)}
	\right\} \\
	&\quad -\sum_{\substack{
		\nu_i^N=\sigma_i \\
		1\leqslant i\leqslant m}}
		\left\{
	\sum_{1< k\leqslant l\leqslant m}
	A{\scriptsize\left(\begin{array}{rrrcrrr}
	 	Nx_1,& \dots,& Nx_{k-2},& N(x_{k-1}+\cdots+x_{l}),& Nx_{l+1},& \dots,& Nx_{m} \\
	 	\nu_1,& \dots,& \nu_{k-2},& \tau_{k-1},& \nu_{l+1},& \dots,& \nu_{m}
	\end{array}\right)}
	\right. \\
	&\hspace{5cm}\cdot \left.
	\vphantom{\sum_{1\leqslant k\leqslant l<m}}
	B{\scriptsize\left(\begin{array}{rrr}
		Nx_k,& \dots,& Nx_l \\
		\nu_k\nu_{k-1}^{-1},& \dots,& \nu_l\nu_{k-1}^{-1} \\
	\end{array}\right)}
	\right\}.
\end{align*}
Here, we put $S_{k,l}:=\{i\in\N\ |\ 1\leqslant i\leqslant k-1\}\cup\{i\in\N\ |\ l+1\leqslant i\leqslant m\}$ for $1\leqslant k\leqslant l\leqslant m$. Then we can divide the set of variables of the summation in the following :
\begin{equation*}
	\{\nu_i\ |\ 1\leqslant i\leqslant m\}
	=\{\nu_i\ |\ i\in S_{k,l}\}
	\cup\{\nu_i\ |\ k\leqslant i\leqslant l\}.
\end{equation*}
So by substituting $\nu_i:=\tau_i$ for $i\in S_{k,l}$ and $\nu_i:=\tau_i\tau_{l+1}$ for $k\leqslant i\leqslant l$ for the first summation, and by substituting $\nu_i:=\tau_i$ for $i\in S_{k,l}$ and $\nu_i:=\tau_i\tau_{k-1}$ for $k\leqslant i\leqslant l$ for the second summation, we calculate
\begin{align*}
	&m_N(\arit_u(B)(A))^m(\vecx_m) \\
	&=  \sum_{1\leqslant k\leqslant l<m}
	\left\{
	\sum_{\substack{
		\tau_i^N=\sigma_i \\
		i\in S_{k,l}}}
	A{\scriptsize\left(\begin{array}{rrrcrrr}
	 	Nx_1,& \dots,& Nx_{k-1},& N(x_k+\cdots+x_{l+1}),& Nx_{l+2},& \dots,& Nx_{m} \\
	 	\tau_1,& \dots,& \tau_{k-1},& \tau_{l+1},& \tau_{l+2},& \dots,& \tau_{m}
	\end{array}\right)}
	\right\} \\
	&\hspace{5cm}\cdot \left\{
	\sum_{\substack{
		\tau_i^N=\sigma_i\sigma_{l+1}^{-1} \\
		k\leqslant i\leqslant l}}
	B{\scriptsize\left(\begin{array}{crc}
		Nx_k,& \dots,& Nx_l \\
		\tau_k,& \dots,&\tau_l \\
	\end{array}\right)}
	\right\} \\
	&\quad -\sum_{1< k\leqslant l\leqslant m}
	\left\{
	\sum_{\substack{
		\tau_i^N=\sigma_i \\
		i\in S_{k,l}}}
	A{\scriptsize\left(\begin{array}{rrrcrrr}
	 	Nx_1,& \dots,& Nx_{k-2},& N(x_{k-1}+\cdots+x_{l}),& Nx_{l+1},& \dots,& Nx_{m} \\
	 	\tau_1,& \dots,& \tau_{k-2},& \tau_{k-1},& \tau_{l+1},& \dots,& \tau_{m}
	\end{array}\right)}
	\right\} \\
	&\hspace{5cm}\cdot \left\{
	\sum_{\substack{
		\tau_i^N=\sigma_i\sigma_{k-1}^{-1} \\
		k\leqslant i\leqslant l}}
	B{\scriptsize\left(\begin{array}{crc}
		Nx_k,& \dots,& Nx_l \\
		\tau_k,& \dots,& \tau_l \\
	\end{array}\right)}
	\right\} \\
	&= \sum_{1\leqslant k\leqslant l<m}
	m_N(A){\scriptsize\left(\begin{array}{rrrcrrr}
	 	x_1,& \dots,& x_{k-1},& x_k+\cdots+x_{l+1},& x_{l+2},& \dots,& x_{m} \\
	 	\sigma_1,& \dots,& \sigma_{k-1},& \sigma_{l+1},& \sigma_{l+2},& \dots,& \sigma_{m}
	\end{array}\right)} \\
	&\hspace{5cm}\cdot m_N(B){\scriptsize\left(\begin{array}{rrr}
		x_k,& \dots,& x_l \\
		\sigma_k\sigma_{l+1}^{-1},& \dots,&\sigma_l\sigma_{l+1}^{-1} \\
	\end{array}\right)} \\
	&\quad -\sum_{1< k\leqslant l\leqslant m}
	m_N(A){\scriptsize\left(\begin{array}{rrrcrrr}
	 	x_1,& \dots,& x_{k-2},& x_{k-1}+\cdots+x_{l},& x_{l+1},& \dots,& x_{m} \\
	 	\sigma_1,& \dots,& \sigma_{k-2},& \sigma_{k-1},& \sigma_{l+1},& \dots,& \sigma_{m}
	\end{array}\right)} \\
	&\hspace{5cm}\cdot m_N(B){\scriptsize\left(\begin{array}{rrr}
		x_k,& \dots,& x_l \\
		\sigma_k\sigma_{k-1}^{-1},& \dots,& \sigma_l\sigma_{k-1}^{-1} \\
	\end{array}\right)}. \\
	\intertext{By putting $\alpha=\varia{x_1,\ \dots,\ x_{k-1}}{\sigma_1,\ \dots,\ \sigma_{k-1}}$, $\beta=\varia{x_k,\ \dots,\ x_l}{\sigma_k,\ \dots,\ \sigma_l}$ and $\gamma=\varia{x_{l+1},\ \dots,\ x_m}{\sigma_{l+1},\ \dots,\ \sigma_m}$, we get}
	&= \sum_{\substack{\vecx_m=\alpha\beta\gamma \\
		\beta,\gamma\neq\emptyset}}
	m_N(A)(\alpha\urflex{\beta}{\gamma})
	m_N(B)(\llflex{\beta}{\gamma})
	-\sum_{\substack{\vecx_m=\alpha\beta\gamma \\
		\alpha,\beta\neq\emptyset}}
	m_N(A)(\ulflex{\alpha}{\beta}\gamma)
	m_N(B)(\lrflex{\alpha}{\beta}) \\
	&= \arit_u(m_N(B))(m_N(A))^m(\vecx_m).
\end{align*}
Hence we get
\begin{align*}
	m_N(\ari_u(A,B))
	&= m_N(\arit_u(B)(A))-m_N(\arit_u(A)(B))+m_N(A\times B)-m_N(B\times A) \\
	&= \arit_u(m_N(B))(m_N(A))-\arit_u(m_N(A))(m_N(B)) \\
	&\hspace{1cm}+m_N(A)\times m_N(B)-m_N(B)\times m_N(A) \\
	&= \ari_u(m_N(A),m_N(B)).
\end{align*}
Therefore, $m_N$ is a Lie algebra homomorphism.

It is obvious that the map $i_N$ is a Lie algebra homomorphism. 
Since both $m_N$ and $i_N$ are Lie algebra homomorphisms,
$\ker(i_N-m_N)$ forms Lie algebra.

By the definition of $\ARID(\Gamma)$, we have 
$\ARID(\Gamma)=\underset{N\bigm |  |\Gamma|}{\bigcap}\ker(i_N-m_N)$. 
Therefore $\ARID(\Gamma)$ forms a Lie subalgebra under the $\ari_u$-bracket.
\end{proof}

\begin{cor}\label{cor:ARID alal Lie algebra}
The $\mathbb Q$-linear subspace $\ARID(\Gamma)_{\underline\al/\underline\al}$ forms a filtered Lie subalgebra of 
$\ARI(\Gamma)$ under the $\ari_u$-bracket.
\end{cor}

\begin{proof}
It is a direct consequence of Proposition \ref{ARIalal Lie algebra} and
Proposition \ref{prop:ARID Lie algebra}.
\end{proof}

\section{Kashiwara-Vergne Lie algebra}\label{sec:Kashiwara-Vergne Lie algebra}

We introduce 
the Kashiwara-Vergne bigraded Lie algebra $\lkrv(\Gamma)_{\bullet\bullet}$
associated with a finite abelian group $\Gamma$
and give its  mould theoretical interpretation by using
$\ARI(\Gamma)_{\push/ \pusnu}$.

\subsection{$\Gamma$-variant of  the KV condition}\label{subsec:KV condition}
We investigate a variant of the defining conditions of  Kashiwara-Vergne graded Lie algebra associated with a finite abelian group $\Gamma$
(cf. Definition \ref{defn: KV condition})
and explain its mould theoretical interpretation in Theorem \ref{thm:reform:krv}.

Let $\mathbb L=\oplus_{w\geqslant 1}\mathbb L_w$ be the free graded Lie $\mathbb Q$-algebra generated by $N+1$ variables $x$ and $y_\sigma$ ($\sigma\in\Gamma$) with $\deg x=\deg y_\sigma=1$.
Here $\mathbb L_w$ is the $\mathbb Q$-linear space generated by Lie monomials whose total degree is $w$.
Occasionally we regard $\mathbb L$ as a bigraded Lie algebra 
$\mathbb L_{\bullet\bullet}=\oplus_{w,d}\mathbb L_{w,d}$,
where $\mathbb L_{w,d}$ is the $\mathbb Q$-linear space generated by Lie monomials whose {\it weight} (the total degree) is $w$ and {\it depth} (the degree with respect to all $y_\sigma$) is $d$.
We encode $\mathbb L$ with  a structure of filtered graded  Lie algebra
by the filtration $\Fil_{\D}^d \mathbb L_w:=\oplus_{N\geqslant d}\mathbb L_{w,N}$ for $d>0$
and denote the associated bigraded Lie algebra by 
$\gr_{\D}\mathbb L=\oplus_{w,d}\gr^d_{\D}\mathbb L_w$
with $\gr_{\D}^d\mathbb L_w=
\Fil_{\D}^d \mathbb L_w/\Fil_{\D}^{d+1} \mathbb L_w$.

The $N+1$-variable  non-commutative polynomial algebra $\mathbb A=\mathbb Q\langle x,y_\sigma;\sigma\in\Gamma\rangle$ is regarded as 
the universal enveloping algebra of $\mathbb L$ and is encoded with the induced degree. 
Similarly $\mathbb A$ is encoded with a structure of bigraded algebra;
$\mathbb A_{\bullet\bullet}=\oplus_{w,d}\mathbb A_{w,d}$.
By putting $\Fil_{\D}^d \mathbb A_w:=\oplus_{N\geqslant d}\mathbb A_{w,N}$
for $d>0$,
we also encode $\mathbb A$ with  a structure of filtered  graded algebra.
We define the action of $\tau\in \Gamma$ on $\mathbb A$ (hence on $\mathbb L$) by
$$
\tau(x)=x \text{ and }\tau(y_\sigma)=y_{\tau\sigma}.
$$
For any $h\in\mathbb A$, we denote
$$h=h_xx+\sum_\alpha h_{y_\alpha} y_\alpha=xh^x+\sum_\alpha y_\alpha h^{y_\alpha}.$$

We denote $\pi_Y$ to be  the composition of the natural projection and inclusion: 
$$\pi_Y:\mathbb A\twoheadrightarrow \mathbb A/\mathbb A\cdot x
\simeq \mathbb Q\oplus (\oplus_{\sigma\in\Gamma}\mathbb A y_\sigma)
\hookrightarrow\mathbb A$$
and the $\mathbb Q$-linear isomorphism $\q$ on $\mathbb A$ defined by
$$
\q(x^{e_0}y_{\sigma_1}x^{e_1}y_{\sigma_2}\cdots x^{e_{r-1}}y_{\sigma_{r}}x^{e_r})
=x^{e_0}y_{\sigma_1}x^{e_1}y_{\sigma_2\sigma_1^{-1}}\cdots
x^{e_{r-1}} y_{\sigma_r\sigma_{r-1}^{-1}}x^{e_r}
$$
(cf. \cite{R}).
The $\mathbb Q$-linear endomorphism 
$$\anti:\mathbb A\to \mathbb A$$
is the palindrome (backwards-writing) operator in $\mathbb A_w$  (cf. \cite[Definition 1.3]{S}).

We put $\mathrm{Cyc}(\mathbb A)$ to be 
$\mathbb Q$-linear space generated by cyclic words of $\mathbb A$ 
and  $\tr:\mathbb A\twoheadrightarrow \mathrm{Cyc}(\mathbb A)$
to be the trace map, the natural projection to
$\mathrm{Cyc}(\mathbb A)$ (cf. \cite{AT}).

\begin{defn}
\label{defn: KV condition}
We define 
the graded $\mathbb Q$-linear space $\krv(\Gamma)_\bullet=\oplus_{w> 1}\krv(\Gamma)_w$,
where
its degree $w$-part $\krv(\Gamma)_w$ is defined to be the set of Lie elements  $F\in\mathbb L_w$ such that
there exists $G=G(F)$ in $\mathbb L_w$ with
\begin{align}\tag{KV1}
\label{KV1}
&[x,G]+\sum_{\tau\in\Gamma}[y_\tau,\tau(F)]=0,\\
\tag{KV2}
\label{KV2}
& \tr\circ\q\circ\pi_Y(F(z;(y_\sigma)))
=0
 \end{align}
with $z=-x-\sum_{\sigma\in\Gamma} y_\sigma$.
\end{defn}

We note that such $G=G(F)$  uniquely exists when $w>1$. 
For $d\geqslant 1$, we put $\Fil_{\D}^d\krv(\Gamma)_w$ to be the subspace of $\krv(\Gamma)_w$ 
consisting of $F\in \Fil_{\D}^d\mathbb L_w$. 
By  \eqref{KV2}, 
$$\krv(\Gamma)_w=\Fil_{\D}^2\krv(\Gamma)_w.$$

\begin{lem}
Assume  that $\Gamma=\{e\}$.
Let $F\in\mathbb L_w$ satisfying \eqref{KV1} with $w\geqslant 2$.
Then \eqref{KV2} for $F$ is equivalent to 
\begin{equation}\label{eq:original KV2}
\tr(G_xx+F_yy)=0.
\end{equation}
\end{lem}

\begin{proof}
By
$
\tr\circ\q\circ\pi_Y(F(z,y))=\tr\circ\pi_Y(F(z,y))=
\tr(-F_x(z,y)y+F_y(z,y)y),
$
the condition
\eqref{KV2} is equivalent to 
$$\tr((F_y-F_x)y)=0.$$

By \eqref{KV1}, we have
$Gx-xG=yF-Fy$.
So $F_y=F^y$ and $G_y=F^x$.
By $F\in\mathbb L_w$, $F_x=(-1)^{w-1}\anti(F^x)$ and
$F_y=(-1)^{w-1}\anti(F^y)$.
By $G\in\mathbb L_w$, we have $\tr(G_xx+G_yy)=0$.
Therefore
$$\tr((F_y-F_x)y)=(-1)^{w-1}\tr((F^y-F^x)y)
=(-1)^{w-1}\tr(F_yy-G_yy)=(-1)^{w-1}\tr(F_yy+G_xx),
$$
whence we get the claim.
\end{proof}

\begin{rem}
Since \eqref{eq:original KV2} agrees with original defining condition (KV2)
in \cite{AT},
we see that our $\krv(\Gamma)_\bullet$ with $\Gamma=\{e\}$
recovers the original Kashiwara-Vergne Lie algebra denoted by $\krv^0_{2}$ 
in \cite{AT} and
the depth>1-part of $\krv$ in \cite[Definition 3]{RS}.
We do not know how their Lie algebras 
$\krv^0_{n+1}$ ($n\geqslant 0$) in \cite{AT} and also
$\mathrm{krv}_{n+1}$ in \cite{AKKN}
are related to our $\krv(\Gamma)_\bullet$.
It is also not clear if our $\krv(\Gamma)_\bullet$ forms a Lie algebra or not.
\end{rem}


%
%

Let $h\in\mathbb A_w$ be a degree $w$ homogeneous polynomial
with
$h=\sum_{r=0}^wh^r$ 
and
\begin{equation}\label{eq: word expansion}
h^r=  \sum_{(\sigma_1,\dots,\sigma_r)\in\Gamma^{\oplus r}}\sum_{(e_0,\dots,e_r)\in\mathcal{E}_w^r}
a(h:{}_{\sigma_1,\dots,\sigma_r}^{e_0,\dots,e_r})x^{e_0}y_{\sigma_1}\cdots y_{\sigma_r}x^{e_r}
\in\mathbb A_{w,r}
\end{equation}
where $\mathcal{E}_w^r=\{(e_0,\dots,e_r)\in\N_0^{r+1} \mid \sum_{i=0}^r e_i=w-r \}$.
\begin{defn}\label{def:ma}
By following \cite[Appendix A]{S},
we associate a mould 
$$\ma_h=(\ma^0_h,\ma^1_h,\ma^2_h,\dots,\ma^w_h,0,0,\dots) \in\mathcal M(\mathcal F;\Gamma)$$
which is defined by 
$\ma^r_h=\{\ma^r_h({}_{\sigma_1,\dots,\sigma_r}^{u_1,\dots,u_r})\}_{(\sigma_1,\dots,\sigma_r)\in\Gamma^{\oplus r}}$
with
\begin{align*}
&
\ma^r_h({}_{\sigma_1,\dots,\sigma_r}^{u_1,\dots,u_r})
=\vimo^r_h({}_{\qquad \sigma_1,\dots,\sigma_r}^{0,u_1,u_1+u_2,\dots,u_1+\cdots+u_r}), \\
&
\vimo^r_h({}^{z_0,\dots,z_r}_{\sigma_1,\dots,\sigma_r})
=\sum_{(e_0,\dots,e_r)\in\mathcal{E}_w^r}
a(h:{}_{\sigma_1^{-1},\dots,\sigma_r^{-1}}^{e_0,\dots,e_r})
z_0^{e_0}z_1^{e_1}z_2^{e_2}\cdots z_r^{e_r}.
\end{align*}
\end{defn}
We start with the following  technical lemma which  is required to our later arguments.

\begin{lem}
When $h\in\mathbb L_{w,r}$, we have
\begin{equation}\label{eq:translation invariance of vimo}
\vimo^r_h({}^{z_0,\dots,z_r}_{\sigma_1,\dots,\sigma_r})
=\vimo^r_h({}^{0,z_1-z_0,\dots,z_r-z_0}_{\qquad \sigma_1,\dots,\sigma_r}),
\end{equation}
\begin{equation}\label{eq:mantar invariance for ma}
\mantar\circ\ma_h^r=\ma_h^r.
\end{equation}
\end{lem}

\begin{proof}
The proof of
\eqref{eq:translation invariance of vimo}
can be done by induction on degree
in the same way to the arguments in 
\cite{S} p.71. 
Assume $h\in\mathbb L_{w,r}$ is given by $[f,g]$ for some 
$f$ and $g\in\mathbb L$ with  depth $s$ and $t$ respectively.
Then by definition, we have
$$
\vimo_h^r({}^{z_0,\dots,z_r}_{\sigma_1,\dots,\sigma_r})
=\vimo_f^s({}^{z_0,\dots,z_s}_{\sigma_1,\dots,\sigma_s})
\vimo_g^t({}^{z_{s},\dots,z_{s+t}}_{\sigma_{s+1},\dots,\sigma_{s+t}})
-
\vimo_g^t({}^{z_0,\dots,z_t}_{\sigma_1,\dots,\sigma_t})
\vimo_f^s({}^{z_{t},\dots,z_{s+t}}_{\sigma_{t+1},\dots,\sigma_{s+t}}).
$$
While by our induction assumption we have
\begin{align*}
&\vimo_h^r({}^{0,z_1-z_0,\dots,z_r-z_0}_{\qquad \sigma_1,\dots,\sigma_r}) \\
&=\vimo_f^s({}^{0,z_1-z_0,\dots,z_s-z_0}_{\qquad\sigma_1,\dots,\sigma_s})
\vimo_g^t({}^{z_{s}-z_0,\dots,z_{s+t}-z_0}_{\qquad\sigma_{s+1},\dots,\sigma_{s+t}}) \\
&\qquad\qquad \qquad \qquad \qquad  \qquad \qquad 
-
\vimo_g^t({}^{0,z_1-z_0,\dots,z_t-z_0}_{\qquad\sigma_1,\dots,\sigma_t})
\vimo_f^s({}^{z_{s}-z_0,\dots,z_{s+t}-z_0}_{\qquad\sigma_{t+1},\dots,\sigma_{s+t}})       \\
&=\vimo_f^s({}^{z_0,\dots,z_s}_{\sigma_1,\dots,\sigma_s})
\vimo_g^t({}^{z_{s},\dots,z_{s+t}}_{\sigma_{s+1},\dots,\sigma_{s+t}})
-
\vimo_g^t({}^{z_0,\dots,z_t}_{\sigma_1,\dots,\sigma_t})
\vimo_f^s({}^{z_{t},\dots,z_{s+t}}_{\sigma_{t+1},\dots,\sigma_{s+t}})  \\
&=\vimo_h^r({}^{z_0,\dots,z_r}_{\sigma_1,\dots,\sigma_r}).
\end{align*}

The equation \eqref{eq:mantar invariance for ma} is a formal generalization of
\cite[Lemma A.2]{S}. We give a short proof below.
Since $h\in\mathbb L_w$, we have
$h=(-1)^{w-1}\anti(h)$.
Thus by definition, we have
$
\vimo_h^r({}^{z_0,\dots,z_r}_{\sigma_1,\dots,\sigma_r})
=(-1)^{w-1}\vimo_h^r({}^{z_r,\dots,z_0}_{\sigma_r,\dots,\sigma_1})
$.
Therefore
\begin{align*}
\ma_h^r&({}^{z_1,\dots,z_r}_{\sigma_1,\dots,\sigma_r})
=\vimo_h^r({}^{0,z_1,z_1+z_2,\dots,z_1+\cdots+z_r}_{\qquad \sigma_1,\dots,\sigma_r})
=(-1)^{w-1}\vimo_h^r({}^{z_1+\cdots+z_r, \dots, z_1+z_2,z_1,0}_{\qquad \sigma_r,\dots,\sigma_1})\\
&=(-1)^{r-1}\vimo_h^r({}^{-z_1-\cdots-z_r, \dots, -z_1-z_2,-z_1,0}_{\qquad \quad\sigma_r,\dots,\sigma_1})
=(-1)^{r-1}\vimo_h^r({}^{0, z_r, z_r+z_{r-1},\dots, z_r+\cdots+z_1}_{\qquad \quad\sigma_r,\dots,\sigma_1}) \\
&=(-1)^{r-1}\ma_h^r({}^{z_r,\dots,z_1}_{\sigma_r,\dots,\sigma_1})
=\mantar\circ\ma_h^r({}^{z_1,\dots,z_r}_{\sigma_1,\dots,\sigma_r}).
\end{align*}
Here in the third equality we use that $\vimo_h^r$ is homogeneous with degree $w-r$
and the fourth equality is by \eqref{eq:translation invariance of vimo}.
\end{proof}

\begin{defn}
Let $\tder$ be the set of tangential derivation of $\mathbb L$,
the derivation  $D_{\{F_\sigma\}_{\sigma}, G}$ of $\mathbb L$
defined by $x\mapsto [x,G]$ and $y_\sigma\mapsto [y_\sigma,F_\sigma]$ for some $F_\sigma,G\in \mathbb L$.
It forms a Lie algebra by the bracket
\begin{equation}\label{bracket}
[D_{\{F^{(1)}_{\sigma}\},G^{(1)}},D_{\{F_\sigma^{(2)}\},G^{(2)}}]
=D_{\{F^{(1)}_{\sigma}\},G^{(1)}}\circ D_{\{F^{(2)}_{\sigma}\},G^{(2)}}
-D_{\{F^{(2)}_{\sigma}\},G^{(2)}}\circ D_{\{F^{(1)}_{\sigma}\},G^{(1)}}.
\end{equation}
The action of $\Gamma$ on $\mathbb L$ induces the $\Gamma$-action on $\tder$.
We denote its invariant part by $\tder^\Gamma$.
We mean $\sder$ to be the set of special derivations,
tangential derivations such that $D_{\{F_{\sigma}\},G}(x+\sum_\sigma y_\sigma)=0$
and $\sder^\Gamma$ to be 
its intersection with $\tder^\Gamma$,
both of which forms a Lie subalgebra of $\tder$.
We put
$\mt:=\oplus_{(w,d)\neq (1,0)}\mathbb L_{w,d}$. It forms a Lie algebra
by the bracket
\begin{equation}\label{eq:mt bracket}
\{f_1,f_2\}=D_{\{\sigma(f_1)\},0}(f_2)-D_{\{\sigma(f_2)\},0}(f_1)+[f_1,f_2],
\end{equation}
in other words, $D_{\{\sigma(\{f_1,f_2\})\},0}:=[D_{\{\sigma(f_1)\},0},D_{\{\sigma(f_2)\},0}]$.
\end{defn}

We occasionally regard  $\mt$  as a Lie subalgebra of $\tder$ 
by $f\mapsto D_{\{\sigma(f)\},0}$.
We note that the condition \eqref{KV1} is equivalent to $D_{\{\sigma(F)\},G}\in\sder^\Gamma$.

We regard
$\sder^\Gamma$ and $\mt$ as filtered  graded Lie algebras
by encoding them with
$\Fil_\D^d\sder^\Gamma_w=\{D_{\{\sigma(F)\},G}\in\sder^\Gamma \bigm|
F\in\Fil_\D^d\mathbb L_w \}$ and
$\Fil_\D^d\mt=\{D_{\{\sigma(F)\},0}\in\mt \bigm|
F\in\Fil_\D^d\mathbb L_w \}$.
The following is required in the next section.
\begin{lem}
We have  a  natural graded Lie algebra homomorphism
\begin{equation}\label{eq:res}
\res:\gr_\D(\sder^\Gamma)\to\mt
\end{equation}
sending $D_{\{\tau(F)\},G}\mapsto D_{\{\tau(\bar F)\},0}$.
\end{lem}

\begin{proof}
Let $D_i=D_{\{\tau(F_i)\},G_i}\in \Fil_\D^{d_i}\sder^\Gamma_{w_i}$ for $i=1,2$.
Put $D_3=[D_1,D_2]$. Since it belongs to $\sder^\Gamma$,
$D_3$ is expressed as $D_{\{\tau(F_3)\},G_3}$ for some $F_3\in\mathbb L$
and $G_3\in\mathbb L$.

By \eqref{KV1}, we have  $\tau(F_i)\in\Fil^{d_i}_{\D}\mathbb L_{w_i}$,
$G_i\in\Fil^{d_i+1}_{\D}\mathbb L_{w_i}$ ($i=1,2$).
Since
$F_3$ is calculated to be
$$
F_3=D_{\{\tau(F_1)\},G_1}(F_2)-D_{\{\tau(F_2)\},G_2}(F_1)+[F_1,F_2]
$$
by \eqref{bracket},
we see that $F_3\in\Fil^{d_1+d_2}_{\D}\mathbb L_{w_1+w_2}$.
The residue class $\bar F_3\in \gr^{d_1+d_2}_{\D}\mathbb L_{w_1+w_2}$ is 
calculated as
\begin{equation}\label{eq:abuse bracket}
\bar F_3=D_{\{\tau(\bar F_1)\},0}(\bar F_2)-D_{\{\tau(\bar F_2)\},0}(\bar F_1)+[\bar F_1, \bar F_2],
\end{equation}
where  $\bar F_1\in\gr^{d_1}_{\D}\mathbb L_{w_1}$ and $\bar F_2\in\gr^{d_2}_{\D}\mathbb L_{w_2}$ are the residue classes of $F_1$ and $F_2$.
Therefore our map is a Lie algebra homomorphism.
\end{proof}

\begin{prop}
\label{prop:MT=ARIal}
The map $\ma$ sending $h\to\ma_h$ induces 
a filtered graded Lie algebra isomorphism
\begin{equation}\label{eq:hom ma}
\ma: \mt\simeq \ARI(\Gamma)^{\fin,\pol}_{\al}
\end{equation}
where 
$\ARI(\Gamma)^{\fin,\pol}_{\al}$ (cf. Definition \ref{defn:fin,pol})
is the Lie algebra (cf. Proposition \ref{ARIal Lie algebra})
equipped with the $\ari_u$-bracket.
\end{prop}

\begin{proof}
It can be proved completely in a same way to that of {\cite[Theorem 3.4.2]{S-ARIGARI}}.
\end{proof}

We prepare the following technical lemma which is  
required to  the proof of a reformulation of \eqref{KV1}
in Lemma \ref{lem:refo1:KV1}.

\begin{lem}\label{lem:aux H=[x,G]}
Let $H\in\mathbb L_w$ with $w\geqslant 2$.
Assume that $H$ has no words starting with any $y_\sigma$ and ending in any $y_\tau$.
Then  there exists  $G\in\mathbb L_{w-1}$ such that $H=[x,G]$.
\end{lem}

\begin{proof}
The proof goes on the same way to the proof of  \cite[Proposition 2.2]{S}.
Define  the derivation $\partial_x$ of $\mathbb A$ sending $x\mapsto 1$ and
$y_\sigma\mapsto 0$ and the $\mathbb Q$-linear endomorphism $\sec$ of $\mathbb A$ by
$\sec(h):=\sum_{i\geqslant 0}\frac{(-1)^i}{i!}\partial_x^i(h)x^i$
for $h\in\mathbb A$.
Let us write  $H=H_xx+\sum_\sigma H_{y_\sigma}y_\sigma$.
Then by our assumption,  we have $P\in\mathbb A$ such that  $xP=\sum_\sigma H_{y_\sigma}y_\sigma$.
Then by \cite{R} Proposition 4.2.2 and 
$\partial_x^i(xP)=i\partial_x^{i-1}(P)+x\partial_x^{i}(P)$, we have
\begin{align*}
H&=\sec(\sum_\sigma H_{y_\sigma}y_\sigma)=\sec(xP)
=\sum_{i\geqslant 0}\frac{(-1)^i}{i!}\partial_x^i(xP)x^i\\ 
&=\sum_{i\geqslant 1}\frac{(-1)^i}{(i-1)!}\partial_x^{i-1}(P)x^i
+x\sum_{i\geqslant 0}\frac{(-1)^i}{i!}\partial_x^i(P)x^i
=-\sec(P)x+x\sec(P)=[x,\sec(P)].
\end{align*}
It remains to show that $G=\sec(P)$ is in $\mathbb L$,
which follows from exactly the same  arguments  to the last half of
the proof of  \cite[Proposition 2.2]{S}.
\end{proof}

Let $F=F(x;(y_\sigma))\in\mathbb L$. As in \cite{S}, we put
\begin{equation}\label{eq:f and F}
f(x;(y_\sigma))=F(z;(y_\sigma)) \quad \text{ and } \quad\tilde f(x;(y_\sigma))=f(x;(-y_\sigma))
\end{equation}
with $z=-x-\sum_{\sigma}y_\sigma$.

%

The following generalizes the equivalence between (i) and (v)
in \cite[Theorem 2.1]{S}.

\begin{lem}
\label{lem:refo1:KV1}
Let $F\in\mathbb L_w$ with $w\geqslant 1$. 
Then saying \eqref{KV1} for $F$ is equivalent to 
saying that
\begin{equation}\label{eq:refo1:KV1}
\{\tilde f_{y_{\gamma}}+\tilde f_x\}
=(-1)^{w-1}\gamma\cdot\anti\{\tilde f_{y_{\gamma^{-1}}}+\tilde f_x\}
\end{equation}
for all $\gamma\in\Gamma$.
\end{lem}

\begin{proof}
Firstly, we show that \eqref{KV1} for $F$ is equivalent to 
$\alpha(F)_{y_\beta}=\beta(F)^{y_\alpha}$ for any $\alpha, \beta\in \Gamma$.
Set $H=\sum_\sigma[y_\sigma,\sigma(F)]$.
Then  we have
\begin{equation}\label{eqn for H}
H=\sum_{\alpha,\beta} 
y_\alpha\left\{
\alpha (F)_{y_\beta} - \beta(F)^{y_\alpha}
\right\}y_\beta
+ \sum_{\alpha}y_\alpha \alpha(F)_xx
- \sum_{\beta}x\beta(F)^xy_\beta. 
\end{equation}

\begin{itemize}
\item
Assume \eqref{KV1} for $F$.
Then 
\begin{equation}\label{eqn for H 2}
H=Gx-xG. 
\end{equation}
So $H$ has no words starting with any $y_\sigma$ and ending in any $y_\tau$.
By \eqref{eqn for H}, we have $\alpha(F)_{y_\beta}=\beta(F)^{y_\alpha}$. 
\item
Conversely assume $\alpha(F)_{y_\beta}=\beta(F)^{y_\alpha}$. 
Then  $H$ has no words starting with any $y_\sigma$ and ending in any $y_\tau$.
By Lemma \ref{lem:aux H=[x,G]},
there is a $G\in\mathbb L_{w}$
such that $H=[G,x]$.
Whence we get \eqref{KV1}.
\end{itemize}

Secondly,
by $\alpha(F)\in\mathbb L_w$, we have $\alpha(F)^{y_\beta}=(-1)^{w-1}\anti(\alpha(F)_{y_\beta})$.
Therefore $\alpha(F)_{y_\beta}=\beta(F)^{y_\alpha}$ 
is equivalent to 
$\alpha(F)_{y_\beta}=(-1)^{w-1}\anti(\beta(F)_{y_\alpha})$.

Lastly,
by \eqref{eq:f and F}
we have
$
f_{y_\alpha}(z;(y_\sigma))-f_x(z;(y_\sigma))=F_{y_\alpha}.
$
Hence
\begin{equation*}
\alpha(F)_{y_\beta}=\alpha(F_{y_{\alpha^{-1}\beta}})=
\alpha\left\{f_{y_{\alpha^{-1}\beta}}(z;(y_\sigma))-f_x(z;(y_\sigma))\right\} 
\end{equation*}
Whence
$\alpha(F)_{y_\beta}=(-1)^{w-1}\anti(\beta(F)_{y_\alpha})$
is equivalent to 
$$
\alpha\left\{f_{y_{\alpha^{-1}\beta}}(z;(y_\sigma))-f_x(z;(y_\sigma))\right\} 
=(-1)^{w-1}\anti\left(
\beta\left\{f_{y_{\alpha\beta^{-1}}}(z;(y_\sigma))-f_x(z;(y_\sigma))\right\} 
\right)
$$
from which our claim follows.
\end{proof}

The following reformulation is suggested by the arguments in \cite[Appendix A]{S}.

\begin{lem}\label{lem:refo2:KV1}
Let $\tilde f\in\mathbb L_w$ with $w>1$. 
Then 
\eqref{eq:refo1:KV1} for all $\gamma\in\Gamma$
is equivalent to  the following 
{\it senary relation}
\footnote{
It is because it consists of 6 terms, 3 terms on each hand sides.
}
(cf. \cite[(3.64)]{E-flex})
\begin{equation}\label{senary relation}
\teru( M)^r=\push\circ\mantar\circ\teru\circ\mantar( M)^r
\end{equation}
for $1\leqslant r \leqslant w$
with $ M=\ma_{\tilde f}$.

\end{lem}

\begin{proof}
%

Because $M=\ma_{\tilde f}$ is $\mantar$ invariant for $\tilde f\in\mathbb L_w$
by \eqref{eq:mantar invariance for ma},
the senary relation \eqref{senary relation} is equivalent to
\begin{equation}\label{senary relation of mantar invariant}
\swap\circ\teru( M)^r
=\swap\circ\push\circ\mantar\circ\teru( M)^r.
\end{equation}
For simplicity, we denote $\sigma_i\cdots\sigma_j$ by $\sigma_{i,j}$ for $1\leqslant i\leqslant j\leqslant r$. By definition, its left hand side is calculated to be
\begin{align}\label{LHS of senary of mantar invariant}
	&\vimo^r_{\tilde f}
	{\scriptsize\left(\begin{array}{rrrrr}
		0,& v_r,& \dots,& v_2,& v_1 \\
		& \sigma_{1,r},& \dots,& \sigma_{1,2},& \sigma_1
	\end{array}\right)} \\
	&\hspace{0cm}+\frac{1}{v_1-v_2}\left\{\vimo^{r-1}_{\tilde f}
		{\scriptsize\left(\begin{array}{rrrrr}
			0,& v_r,& \dots,& v_3,& v_1 \\
			& \sigma_{1,r},& \dots,& \sigma_{1,3},& \sigma_{1,2}
		\end{array}\right)}
	-\vimo^{r-1}_{\tilde f}
	{\scriptsize\left(\begin{array}{rrrrr}
		0,& v_r,& \dots,& v_3,& v_2 \\
		& \sigma_{1,r},& \dots,& \sigma_{1,3},& \sigma_{1,2}
	\end{array}\right)}\right\}. \nonumber
\end{align}
By 
$\vimo_{\tilde f^r}^r\varia{z_0,\dots,z_r}{\sigma_1,\dots,\sigma_r} =(-1)^{w-r}\vimo_{\tilde f^r}^r\varia{-z_0,\dots,-z_r}{\ \ \sigma_1,\dots,\sigma_r}$
 and definition, its right hand side is calculated to be
\begin{align}\label{RHS of senary of mantar invariant}
	&(-1)^{w-1}\left[\vimo^r_{\tilde f}
	{\scriptsize\left(\begin{array}{rrrrrr}
		0,& v_3-v_2,& \dots,& v_r-v_2,& -v_2,& v_1-v_2 \\
		& \sigma_{2,2},& \dots,& \sigma_{2,r-1},& \sigma_{2,r},& \sigma_1^{-1}
	\end{array}\right)}\right. \\
	&\hspace{0cm}\left.+\frac{1}{v_1}\left\{\vimo^{r-1}_{\tilde f}
		{\scriptsize\left(\begin{array}{rrrrr}
			0,& v_3-v_2,& \dots,& v_r-v_2,& v_1-v_2 \\
			& \sigma_{2,2},& \dots,& \sigma_{2,r-1},& \sigma_{2,r}
		\end{array}\right)}
	-\vimo^{r-1}_{\tilde f}
	{\scriptsize\left(\begin{array}{rrrrr}
		0,& v_3-v_2,& \dots,& v_r-v_2,& -v_2 \\
		& \sigma_{2,2},& \dots,& \sigma_{2,r-1},& \sigma_{2,r}
	\end{array}\right)}\right\}\right]. \nonumber
\end{align}
By $\vimo_{\tilde f^r}^r\varia{z_0,\dots,z_r}{\sigma_1,\dots,\sigma_r} =\vimo_{\tilde f^r}^r\varia{0,z_1-z_0,\dots,z_r-z_0}{\qquad\sigma_1,\dots,\sigma_r}$ 
in \eqref{eq:translation invariance of vimo}
and the change of variables 
$z_0=-v_1$, $z_1=v_r-v_1$, $\dots$, $z_{r-1}=v_2-v_1$ and $\gamma=\sigma_1$, $\tau_1=\sigma_{1,r}$, $\dots$, $\tau_{r-1}=\sigma_{1,2}$, the formula \eqref{LHS of senary of mantar invariant} is equal to
\begin{align}\label{LHS of senary of mantar invariant-2}
	\vimo^{r}_{{\tilde f}^{r}}\varia{z_0,\dots,z_{r-1},0}{\tau_1,\dots,\tau_{r-1},\gamma}
	+\frac{1}{z_{r-1}}\left\{
		\vimo^{r-1}_{\tilde f^{r-1}}\varia{z_0,\dots,z_{r-1}}{\tau_1,\dots,\tau_{r-1}}
		-\vimo^{r-1}_{\tilde f^{r-1}}\varia{z_0,\dots,z_{r-2},0}{\tau_1,\dots,\tau_{r-2},\tau_{r-1}}
	\right\},
\end{align}
and \eqref{RHS of senary of mantar invariant} is equal to
\begin{align}\label{RHS of senary of mantar invariant-2}
	&(-1)^{w-1}\left[
	\vimo^{r}_{{\tilde f}^{r}}\varia{z_{r-1},\dots,z_0,0}{\gamma^{-1}\tau_{r-1},\dots,\gamma^{-1}\tau_1,\gamma^{-1}} \right. \\
	&\hspace{1cm}\left.+\frac{1}{z_0}\left\{
		\vimo^{r-1}_{\tilde f^{r-1}}\varia{z_{r-1},\dots,z_1,z_0}{\gamma^{-1}\tau_{r-1},\dots,\gamma^{-1}\tau_1}
		-\vimo^{r-1}_{\tilde f^{r-1}}\varia{z_{r-1},\dots,z_1,0}{\gamma^{-1}\tau_{r-1},\dots,\gamma^{-1}\tau_1}
	\right\}
	\right]. \nonumber
\end{align}
Whence \eqref{senary relation} is equivalent to \eqref{LHS of senary of mantar invariant-2}=\eqref{RHS of senary of mantar invariant-2}
for $1\leqslant r \leqslant w$ and $\tau_1,\dots,\tau_{r-1},\gamma\in\Gamma$.

On the other hand, \eqref{eq:refo1:KV1} for all $\gamma\in\Gamma$
is equivalent to
\footnote{
Here ${\tilde f_x^r}$ and ${\tilde f_y^r}$ mean $({\tilde f_x})^r$
and $({\tilde f_y})^r$.
} 
\begin{align}\label{eq:refo1:KV1-ver:2}
	&\vimo^r_{\tilde f_{y_{\gamma}}^r}\varia{z_0,\dots,z_r}{\tau_1,\dots,\tau_r}
		+\vimo^r_{\tilde f_x^r}\varia{z_0,\dots,z_r}{\tau_1,\dots,\tau_r} \\
	&=(-1)^{w-1}\left[\vimo^r_{\gamma\cdot\anti(\tilde f_{y_{\gamma^{-1}}}^r)}
		\varia{z_0,\dots,z_r}{\tau_1,\dots,\tau_r}
		+\vimo^r_{\gamma\cdot\anti(\tilde f_x^r)}
		\varia{z_0,\dots,z_r}{\tau_1,\dots,\tau_r}\right] \nonumber
\end{align}
for $\gamma\in\Gamma$ and $1\leqslant r \leqslant w$.
By definition, we have
\begin{align*}
	\vimo^{r+1}_{(\tilde f_{y_{\gamma}})^ry_{\gamma}}\varia{z_0,\dots,z_{r+1}}{\tau_1,\dots,\tau_{r+1}}
	=\delta_{\tau_{r+1},\gamma}\vimo^{r+1}_{{\tilde f}^{r+1}}\varia{z_0,\dots,z_r,0}{\tau_1,\dots,\tau_r,\gamma}
\end{align*}
where $\delta_{\tau_{r+1},\gamma}
:=\left\{\begin{array}{ll}
	1 & (\tau_{r+1}=\gamma), \\
	0 & (\mbox{otherwise}).
\end{array}\right.$
So by using this, we get
\begin{align}\label{vimo y formula}
	\vimo^r_{\tilde f_{y_{\gamma}}^r}\varia{z_0,\dots,z_{r}}{\tau_1,\dots,\tau_{r}}
	=\vimo^{r+1}_{(\tilde f_{y_{\gamma}})^ry_{\gamma}}\varia{z_0,\dots,z_r,0}{\tau_1,\dots,\tau_r,\gamma}
	=\vimo^{r+1}_{{\tilde f}^{r+1}}\varia{z_0,\dots,z_r,0}{\tau_1,\dots,\tau_r,\gamma}.
\end{align}
By ${\tilde f}^r={(\tilde f_{x})}^rx+\sum_{\gamma\in\Gamma}{(\tilde f_{y_{\gamma}})}^{r-1}y_{\gamma}$ and $\vimo^r_{(\tilde f_x)^rx}\varia{z_0,\dots,z_r}{\tau_1,\dots,\tau_r}=z_r\cdot\vimo^r_{\tilde f_x^r}\varia{z_0,\dots,z_r}{\tau_1,\dots,\tau_r}$, we have
\begin{align}\label{vimo x formula}
	\vimo^r_{\tilde f_x^r}\varia{z_0,\dots,z_r}{\tau_1,\dots,\tau_r}
	=\frac{1}{z_r}\left\{
		\vimo^r_{\tilde f^r}\varia{z_0,\dots,z_r}{\tau_1,\dots,\tau_r}
		-\vimo^r_{\tilde f^r}\varia{z_0,\dots,z_{r-1},0}{\tau_1,\dots,\tau_{r-1},\tau_r}
	\right\}.
\end{align}
Hence, by \eqref{vimo y formula} and \eqref{vimo x formula}, the left hand side of \eqref{eq:refo1:KV1-ver:2} is calculated to be
\begin{align}\label{LHS of eq:refo1:KV1-ver:2}
	\vimo^{r+1}_{{\tilde f}^{r+1}}\varia{z_0,\dots,z_r,0}{\tau_1,\dots,\tau_r,\gamma}
	+\frac{1}{z_r}\left\{
		\vimo^r_{\tilde f^r}\varia{z_0,\dots,z_r}{\tau_1,\dots,\tau_r}
		-\vimo^r_{\tilde f^r}\varia{z_0,\dots,z_{r-1},0}{\tau_1,\dots,\tau_{r-1},\tau_r}
	\right\},
\end{align}
and the left hand side of \eqref{eq:refo1:KV1-ver:2} is calculated to be
\begin{align}\label{RHS of eq:refo1:KV1-ver:2}
	&(-1)^{w-1}\left[
	\vimo^{r+1}_{{\tilde f}^{r+1}}\varia{z_r,\dots,z_0,0}{\gamma^{-1}\tau_r,\dots,\gamma^{-1}\tau_1,\gamma^{-1}} \right. \\
	&\hspace{1cm}\left.+\frac{1}{z_0}\left\{
		\vimo^r_{\tilde f^r}\varia{z_r,\dots,z_1,z_0}{\gamma^{-1}\tau_r,\dots,\gamma^{-1}\tau_1}
		-\vimo^r_{\tilde f^r}\varia{z_r,\dots,z_1,0}{\gamma^{-1}\tau_r,\dots,\gamma^{-1}\tau_1}
	\right\}
	\right]. \nonumber
\end{align}
So \eqref{eq:refo1:KV1} for all $\gamma\in\Gamma$
is equivalent to \eqref{LHS of eq:refo1:KV1-ver:2}=\eqref{RHS of eq:refo1:KV1-ver:2} for $1\leqslant r+1 \leqslant w$ and $\tau_1,\dots,\tau_{r},\gamma\in\Gamma$. This is nothing but \eqref{LHS of senary of mantar invariant-2}=\eqref{RHS of senary of mantar invariant-2} for $r-1$ instead of $r$.
Therefore we get the equivalence between \eqref{eq:refo1:KV1} 
and
\eqref{senary relation}.
\end{proof}

Next we consider the condition \eqref{KV2}.

\begin{lem}\label{lem:refo10:kv2}
Let $F\in\mathbb L_w$. 
Then \eqref{KV2} for $F$ is equivalent to 
\begin{equation}\label{eq:refo10:kv2}
\sum_{i\in\Z/r\Z}
a(\q\circ\pi_Y(f):{}^{e_{i},e_{i+1},\dots,e_{i+r-1},0}_{\sigma_{i+1},\sigma_{i+2},\dots,\sigma_{i+r}})=0
\end{equation}
for each $1\leqslant r\leqslant w$, $(\sigma_1,\dots,\sigma_r)\in\Gamma^r$ and
$(e_0,\dots,e_r)\in\mathcal{E}_w^r$
with $f$ as in \eqref{eq:f and F}.
\end{lem}

\begin{proof}
It is immediate that \eqref{KV2} for $F$ is equivalent to 
\begin{equation}\label{reformualtion-quasi-push-inv}
\tr(\q\circ\pi_Y(f))=0.
\end{equation}

For $W=u_1\cdots u_w\in\mathbb A_w$ with $u_i=x$ or $y_\sigma$ ($\sigma\in\Gamma$),
we define  its cyclic permutation by $\cp(W):=u_2\cdots u_wu_1\in\mathbb A_w$.
We put $c(W)$ to be the number of its cycles. It divides $w$.
Then the natural pairing in $\mathrm{Cyc}(\mathbb A)$ is calculated by the one in $\mathbb A$
as follows: 
When $W=x^{e_0}y_{\sigma_1}\cdots x^{e_{r-1}}y_{\sigma_r}$,
we have
\begin{align*}
\langle \tr &(
\q\circ\pi_Y(f)
)
\Bigm|\tr (W)\rangle_{\mathrm{Cyc}(\mathbb A)} 
=\bigl\langle
\q\circ\pi_Y(f)
\Bigm|
\sum_{i=0}^{c(W)-1}\cp^i(W)
\bigr\rangle_{\mathbb A} \\
&=\frac{c(W)}{w}\bigl\langle
\q\circ\pi_Y(f)
\bigm|
\sum_{i=0}^{w-1}\cp^i(W)
\bigr\rangle_{\mathbb A} 
=\frac{c(W)}{w}\sum_{i\in\Z/r\Z}a(\q\circ\pi_Y(f)
:{}^{e_{i},\dots,e_{i+r-1},0}_{\sigma_{i+1},\dots,\sigma_{i+r}}).
\end{align*}
So we get the claim.
\end{proof}

\begin{lem}\label{lem:refo2:KV2}
Let $f\in\mathbb L_w$. 
Then \eqref{eq:refo10:kv2} for $f$ is equivalent to
the pus-neutrality \eqref{pus-neutral} for $\swap(\ma_{\tilde f})$, i.e.
\begin{equation}\label{reform:KV2}
\sum_{i\in\Z/r\Z}
\pus^i\circ\swap( M)^r
\varia{\sigma_{1},\dots,\sigma_{r}}{v_{1},\dots,v_{r}}
=0
\end{equation}
with $ M=\ma_{\tilde f}$
for all $1\leqslant r\leqslant w$
and 
$\sigma_i\in\Gamma$ ($1\leqslant i\leqslant r$).
\end{lem}

\begin{proof}
We decompose $\tilde f=\sum_r \tilde f^r$ and describe
$\tilde f^r\in  \mathbb L_{w,r}$ as in \eqref{eq: word expansion}.
Then
\begin{equation*}
\ma^r_{\tilde f}({}^{u_1,\dots,u_r}_{\sigma_1,\dots,\sigma_r})
=\underset{e_0=0}{\sum_{(e_0,\dots,e_r)\in\mathcal{E}_w^r}}
a(\tilde f:{}_{\sigma_1^{-1},\dots,\sigma_r^{-1}}^{e_0,\dots,e_r})
u_1^{e_{1}}(u_1+u_2)^{e_{2}}\cdots (u_1+\cdots+ u_r)^{e_{r}}.
\end{equation*}
By $\tilde f^r\in\mathbb L_{w,r}$, we have 
$$
a(\tilde f:{}^{e_0,\dots,e_{r}}_{\sigma_1,\dots,\sigma_r})
=(-1)^{w+1} a(\tilde f:{}^{e_r,\dots,e_{0}}_{\sigma_r,\dots,\sigma_1})
$$
because $\anti(\tilde f^r)=(-1)^{w+1}\tilde f^r$. So
\begin{align*}
\swap\circ \ma^r_{\tilde f}({}_{v_1,\dots,v_r}^{\sigma_1,\dots,\sigma_r})
&=\underset{e_0=0}{\sum_{(e_0,\dots,e_r)\in\mathcal{E}_w^r}}
a(\tilde f:{}_{(\sigma_1\cdots\sigma_r)^{-1},\dots,(\sigma_1\sigma_2)^{-1},\sigma_1^{-1}}^{\qquad e_0,e_1,\dots,e_r})
v_r^{e_{1}}v_{r-1}^{e_{2}}\cdots v_1^{e_{r}} \\
&=(-1)^{w+1}\underset{e_0=0}{\sum_{(e_0,\dots,e_r)\in\mathcal{E}_w^r}}
a(\tilde f:{}_{\sigma_1^{-1}, (\sigma_1\sigma_2)^{-1},\dots,(\sigma_1\cdots\sigma_r)^{-1}}^{\qquad e_r,\dots,e_1,e_0})
v_1^{e_{r}}v_2^{e_{r-1}}\cdots v_r^{e_{1}} \\
&=(-1)^{w+1}\underset{e_r=0}{\sum_{(e_0,\dots,e_r)\in\mathcal{E}_w^r}}
a(\tilde f:{}_{\sigma_1^{-1}, (\sigma_1\sigma_2)^{-1},\dots,(\sigma_1\cdots\sigma_r)^{-1}}^{\qquad\quad e_0,\dots,e_{r-1},e_r})
v_1^{e_{0}}v_2^{e_{1}}\cdots v_r^{e_{r-1}} \\
&=(-1)^{w+r+1}\underset{e_r=0}{\sum_{(e_0,\dots,e_r)\in\mathcal{E}_w^r}}
a(f:{}_{\sigma_1^{-1}, (\sigma_1\sigma_2)^{-1},\dots,(\sigma_1\cdots\sigma_r)^{-1}}^{\qquad e_0,\dots,e_{r-1},e_r})
v_1^{e_{0}}v_2^{e_{1}}\cdots v_r^{e_{r-1}} \\
&=(-1)^{w+r+1}\underset{e_r=0}{\sum_{(e_0,\dots,e_r)\in\mathcal{E}_w^r}}
a(\q\circ \pi_Y(f):{}_{\sigma_1^{-1}, \sigma_2^{-1},\dots,\sigma_r^{-1}}^{e_0,\dots,e_{r-1},0})
v_1^{e_{0}}v_2^{e_{1}}\cdots v_r^{e_{r-1}}. 
\end{align*}
Here in the last equality we use 
$$
a(\q\circ\pi_Y(h):{}_{\quad \tau_1,\dots,\tau_r}^{e_0,\dots,e_{r-1},0})
=a(h:{}_{\tau_1,\tau_1\tau_2,\dots,\tau_1\cdots\tau_r}^{\ e_0,\dots,e_{r-1},0})
$$
for any $h\in\pi ({\mathbb A})$.
Therefore we obtain 
\begin{align*}
\sum_{i\in\Z/r\Z}\pus^i & \circ\swap\circ\ma^r_{\tilde f}({}_{v_1,\dots,v_r}^{\sigma_1,\dots,\sigma_r}) \\&
=(-1)^{w+r+1}\sum_{i\in\Z/r\Z}\underset{e_r=0}{\sum_{(e_0,\dots,e_r)\in\mathcal{E}_w^r}}
a(\q\circ \pi_Y(f):{}_{\sigma_{i+1}^{-1},\dots,
\sigma_{i+r}^{-1}}^{ e_i,\dots,e_{i+r-1}, \ 0})
v_1^{e_{0}}v_2^{e_{1}}\cdots v_r^{e_{r-1}}.
\end{align*}
%
It is immediate to see that it is equivalent to \eqref{eq:refo10:kv2}.
\end{proof}

The following definition of
the mould version of $\krv_\bullet$ is suggested by our previous lemmas.

\begin{defn}\label{def:ARI-pspushpusnu}
We define the $\mathbb Q$-linear space
$\ARI(\Gamma)_{\pspush/\pusnu}$
to be the subset of moulds $ M$ in $\ARI$
which satisfy the senary relation \eqref{senary relation} and 
whose $\swap$ satisfy the pus-neutrality 
\eqref{reform:KV2}. 
\end{defn}

The $\mathbb Q$-linear space $\ARI(\Gamma)_{\pspush/\pusnu}$
is filtered by the depth filtration $\{\Fil^m_\D\}_m$.
We note that $\Fil^2_\D\ARI(\Gamma)_{\pspush/\pusnu}\cap \ARI(\Gamma)_\al^{\fin,\pol}$
is identified with the finite and depth >1-part of
$\ARI^\pol_{\al+\mathrm{sen*circconst}}$
in \cite[Proposition 28]{RS}
when $\Gamma=\{e\}$.
It is because their formula (79), actually (81), agrees with
\eqref{senary relation} for a mould in $\ARI(\Gamma)_\al^{\fin,\pol}$
by Lemma \ref{lem:Embd:1}.

There is an embedding of $\mathbb Q$-linear space 
\begin{equation}\label{eq:inclusion grARI}
\gr_{\D}\ARI(\Gamma)_{\pspush/\pusnu}\hookrightarrow
\ARI(\Gamma)_{\push/\pusnu}
\end{equation}
that is, the associated graded quotient $\gr_{\D}\ARI(\Gamma)_{\pspush/\pusnu}$
of the filtered $\mathbb Q$-linear space
$\ARI(\Gamma)_{\pspush/\pusnu}$ 
is embedded to $\ARI(\Gamma)_{\push/\pusnu}$ introduced in Definition \ref{def:ARIpushpusnu}.
It is because 
the push-invariance \eqref{push-invariant} is associated graded with
\eqref{senary relation}
by $\gr_{\D}(\teru)=\id$ and $\mantar\circ\mantar=\id$.


\begin{thm}\label{thm:reform:krv}
The map sending $F\in\mathbb A\mapsto\ma_{\tilde f}\in\mathcal M({\mathcal F};\Gamma)$ induces an isomorphism 
of filtered $\mathbb Q$-linear spaces
\begin{equation}\label{eq: KRV=ARI al+sena/pusnu}
\krv(\Gamma)_\bullet\simeq
\ARI(\Gamma)_{\pspush/\pusnu}\cap\ARI(\Gamma)_\al^{\fin,\pol}.
\end{equation}
\end{thm}

\begin{proof}
The restriction of our map 
decomposes as
\begin{equation}\label{eq:krv into ARIal}
F\in\krv(\Gamma)_\bullet 
\mapsto {\tilde f}\in\mt \mapsto \ma_{\tilde f}\in\ARI(\Gamma)_\al^{\fin,\pol}.
\end{equation}
By 
Proposition \ref{prop:MT=ARIal},
we see that it gives an isomorphism $\mt\simeq\ARI(\Gamma)_\al^{\fin,\pol}$ as (actually filtered) $\mathbb Q$-linear spaces.
Thus we get the claim by our previous lemmas.
\end{proof}

\begin{rem}
When $\Gamma=\{e\}$, the above isomorphism \eqref{eq: KRV=ARI al+sena/pusnu} can be recovered by the composition 
of the isomorphisms (29),  (71) and Proposition 28
\footnote{
It looks that there is a slip on the isomorphism on \cite[Proposition 28]{RS} 
because its right hand side is completed {\it by depth} while its left hand side is not.
}
in \cite{RS}.
\end{rem}

%
%
The authors are not sure
if the bigger space $\ARI(\Gamma)_{\pspush/\pusnu}$ is equipped with a structure of Lie algebra under the $\ari_u$-bracket
or not
although we show that a related space $\ARI(\Gamma)_{\push/\pusnu}$ forms a Lie algebra
in Theorem \ref{thm: ARIpushpusnu Lie}.
%

\subsection{Kashiwara-Vergne bigraded Lie algebra}\label{sec:Kashiwara-Vergne bigraded Lie algebra}
Based on our arguments in the previous subsection,
we introduce
a $\Gamma$-variant $\lkrv(\Gamma)_{\bullet\bullet}$
of the Kashiwara-Vergne bigraded Lie algebra $\lkrv$  (\cite{RS})
in Definition \ref{defn:lrkv}
and give a mould theoretical interpretation
in Theorem \ref{thm:reform:krv:bigrade}.
As a corollary we show that it forms a Lie algebra in Theorem \ref{thm:lkrv Lie alg}.

\begin{defn}\label{defn:lrkv}
Kashiwara-Vergne bigraded Lie algebra is defined to be the bigraded $\mathbb Q$-linear space 
$\lkrv(\Gamma)_{\bullet\bullet}=\oplus_{w>1,d>0}\lkrv(\Gamma)_{w,d}$,
where $\lkrv(\Gamma)_{w,d}$ is the $\mathbb Q$-linear space
consisting of $\bar F\in \gr^d_{\D}\mathbb L_w$ whose lift
$F\in\Fil^d_{\D}\mathbb L_w$ satisfies the following relations,
\lq the {leading-terms}' of \eqref{KV1} and \eqref{KV2}, 
\begin{align}\tag{LKV1}
\label{DKV1}
& [x,G]+\sum_{\tau\in\Gamma}[y_\tau,\tau(F)] \equiv 0 \mod  \Fil_{\D}^{d+2} \mathbb L_{w+1},\\
\tag{LKV2}
\label{DKV2}
& \tr\circ\q\circ\pi_Y(F)
\equiv 0 \mod  \tr(\Fil_{\D}^{d+1} \mathbb A_{w}) 
\end{align}
for a certain $G=G(F)\in\mathbb L_w$. 
\end{defn}
%
We note that such $G\in\mathbb L_{w}$ is in $\Fil_{\D}^{d+1} \mathbb L_{w}$ and
is uniquely determined modulo $\Fil_{\D}^{d+2} \mathbb L_{w}$
by \eqref{DKV1}.
We note that, by  \eqref{DKV2} and $\dim\mathbb L_{w,1}=1$, we have
 $$\lkrv_{w,d}=\{0\} \qquad  \text{ for  } d=1.$$

\begin{rem}
When $\Gamma=\{e\}$, our definition of $\lkrv(\Gamma)_{\bullet\bullet}$  does not agree with
that of $\mathfrak{lkv}$ in \cite[Definition 5 and 10]{RS},
which is constructed as a Lie polynomial version of $\ARI(\Gamma)_{\push/\pusnu}$.
However their depth>1-parts  are eventually isomorphic to
our depth>1-parts 
via Theorem \ref{thm:reform:krv:bigrade}. 
\end{rem}

By definition, there is an inclusion of bigraded $\mathbb Q$-linear spaces
\begin{equation}\label{eq:inclusion grKRV}
\gr_{\D}\krv(\Gamma)_\bullet \hookrightarrow
\lkrv(\Gamma)_{\bullet\bullet},
\end{equation}
which generalizes \cite[Proposition 2]{RS},
that is, 
the associated graded $\mathbb Q$-linear space $\gr_{\D}\krv(\Gamma)_\bullet$
of the filtered Lie algebra $\krv(\Gamma)_\bullet$ 
is embedded to $\lkrv(\Gamma)_{\bullet\bullet}$.
We do not know if it is an isomorphism.


For $F\in\Fil^d_{\D}\mathbb L_w$,
we put $\f$ to be the element in $\mathbb L_{w,d}$ corresponding to
$(-1)^{w-d}\bar F\in\gr^m_{\D}\mathbb L_w$ under the natural identification
$\gr^d_{\D}\mathbb L_w\simeq \mathbb L_{w,d}$.
By abuse of notation,
\begin{equation}\label{eq:new f}
\f=(-1)^{w-d}\bar F.
\end{equation}
We write $\f=\f_xx+\sum_\sigma\f_{y_\sigma}y_\sigma$.
We also put
$$\tilde \f=\f(x,(-y_\sigma)_\sigma)$$
and write $\tilde\f=\tilde\f_xx+\sum_\sigma\tilde\f_{y_\sigma}y_\sigma$.
%
%

The following is a bigraded variant of Lemma \ref{lem:refo1:KV1}.

\begin{lem}\label{lem:refo1:DKV1}
Let  $F\in\Fil^d_{\D}\mathbb L_w$ 
for $w\geqslant 1$. 
Then \eqref{DKV1} for $F$ is equivalent to 
\begin{equation}\label{eq::refo1:LKV1}
\tilde\f_{y_\gamma}= (-1)^{w-1}\gamma(\anti(\tilde\f_{y_{{\gamma}^{-1}}}))
\end{equation}
for any 
$\gamma\in\Gamma$.
\end{lem}

\begin{proof}
The proof goes  similarly to  Lemma \ref{lem:refo1:KV1}.

Firstly, we show that \eqref{DKV1} for $F$ is equivalent to
$\alpha(F)_{y_\beta}\equiv \beta(F)^{y_\alpha} \mod \Fil_{\D}^{d}\mathbb L_{w-1}$.
Set $H=\sum_\sigma [y_\sigma,\sigma(F)]\in  \Fil_{\D}^{d+1}\mathbb L_{w+1}$.

Assume \eqref{DKV1} for $F$.
Then $H\equiv Gx-xG\mod \Fil_{\D}^{d+2}\mathbb L_{w+1}$. 
So the depth $(d+1)$-part of $H$ has no words starting and ending in any  $y_\sigma$.
By \eqref{eqn for H}, we have 
$\alpha(F)_{y_\beta}\equiv \beta(F)^{y_\alpha}
\mod \Fil_{\D}^{d}\mathbb L_{w-1}$.

Conversely assume $\alpha(F)_{y_\beta}\equiv \beta(F)^{y_\alpha}
\mod \Fil_{\D}^{d}\mathbb L_{w-1}$.
Then the depth $(d+1)$-part of $H$ has no words starting and ending in any $y_\sigma$.
By Lemma \ref{lem:aux H=[x,G]}, there is a $G\in\mathbb L_{w}$
which express this part as $[G,x]$, i.e.
$H \equiv [G,x]\mod \Fil_{\D}^{d+2}\mathbb L_{w+1}$.
Whence we get \eqref{DKV1}.

Secondly,
by $\alpha(F)\in\mathbb L$, we have $\alpha(F)^{y_\beta}=(-1)^{w-1}\anti(\alpha (F)_{y_\beta})$.
Therefore $\alpha(F)_{y_\beta}\equiv \beta(F)^{y_\alpha}
\mod \Fil_{\D}^{d}\mathbb L_{w-1}$
is equivalent to 
$$\alpha(F)_{y_\beta}\equiv (-1)^{w-1}\anti(\beta(F)_{y_\alpha}) \mod \Fil_{\D}^{d}\mathbb L_{w-1}.$$

Lastly,
by \eqref{eq:new f} and $\alpha(F)_{y_\beta}=\alpha(F_{y_{\alpha^{-1}\beta}})$,
it is equivalent to
$\alpha(\f_{y_{\alpha^{-1}\beta}})= (-1)^{w-1}\anti(\beta(\f_{y_{\alpha\beta^{-1}}}))$,
so for $\tilde \f_y$.
\end{proof}

The following might be regarded as  a bigraded variant of  Lemma \ref{lem:refo2:KV1}.

\begin{lem}
\label{lem:refo2:DKV1}
Let $\tilde \f\in\mathbb L_{w,d}$ with $w>1$. 
Then \eqref{eq::refo1:LKV1} for any $\gamma\in\Gamma$
is equivalent to $\push$-invariance \eqref{push-invariant}
for $ M=\ma_{\tilde \f}$.
\end{lem}


\begin{proof}
The proof goes similarly to Lemma \ref{lem:refo2:KV1}.
The condition  \eqref{eq::refo1:LKV1} for $\gamma\in\Gamma$
is reformulated to 
\begin{equation*}
\vimo^{r}_{\tilde \f_{y_{\gamma}}^{r}}\varia{z_0,\dots,z_{r}}{\tau_1,\dots,\tau_{r}}
=(-1)^{w-1}\vimo^{r}_{\gamma\cdot\anti(\tilde \f_{y_{\gamma^{-1}}}^{r})}
		\varia{z_0,\dots,z_{r}}{\tau_1,\dots,\tau_{r}}
\end{equation*}
for $1\leqslant r \leqslant w$, 
which is equivalent to
\begin{equation}\label{eqn:another refo:anti-pal}
\vimo^{r+1}_{\tilde \f^{r+1}}\varia{z_0,\dots,z_{r},0}{\tau_1,\dots,\tau_{r},\gamma}
=(-1)^{w-1}\vimo^{r+1}_{\tilde \f^{r+1}}
		\varia{\qquad z_{r},\dots,z_{0},0}{\gamma^{-1}\tau_{r},\dots,\gamma^{-1}\tau_{1},\gamma^{-1}}
\end{equation}
for $1\leqslant r+1 \leqslant w$.

On the other hand, \eqref{push-invariant} can be shown to be equivalent to 
$$
\vimo^{r}_{\tilde \f^{r}}\varia{z_0,\dots,z_{r-1},0}{\tau_1,\dots,\tau_{r-1},\gamma}
=(-1)^{w-1}\vimo^{r}_{\tilde \f^{r}}
		\varia{\qquad z_{r-1},\dots,z_{0},0}{\gamma^{-1}\tau_{r-1},\dots,\gamma^{-1}\tau_{1},\gamma^{-1}}
$$
for $1\leqslant r \leqslant w$,
which turns to be equivalent to
\eqref{eqn:another refo:anti-pal}.

\end{proof}

%
%

The following is a bigraded variant of Lemma \ref{lem:refo2:KV2}.

\begin{lem}\label{lem:refo2:DKV2}
Then \eqref{DKV2} for $F\in\Fil^d_{\D}\mathbb L_w$ is equivalent to 
the $\pus$-neutrality \eqref{pus-neutral} for $ M=\swap(\ma_{\tilde \f})$,
i.e. \eqref{reform:KV2} for all $r\geqslant 1$,
\end{lem}

\begin{proof}
We note that actually only the terms for $r=d$ contributes in the above equation.
Decompose $\tilde \f$ as in
\eqref{eq: word expansion}.
Then by the arguments in the proof of  Lemma \ref{lem:refo2:KV2},
\eqref{reform:KV2} is equivalent to
$$
\sum_{i\in\Z/r\Z}
a(\q\circ \pi_Y(\f):{}_{\sigma_{i+1}^{-1},\dots,
\sigma_{i+r}^{-1}}^{ e_i,\dots,e_{i+r-1}, \ 0})=0
$$
for all $(\sigma_1,\dots,\sigma_r)\in\Gamma^r$ and
$(e_0,\dots,e_{r-1},0)\in\mathcal{E}_w^r$.
It is nothing but
$$
\tr\circ\q\circ\pi_Y(\f)=0,
$$
which is equivalent to \eqref{DKV2}.
\end{proof}


\begin{thm}\label{thm:reform:krv:bigrade}
The map sending $\bar F\in\gr_\D\mathbb L\mapsto\ma_{\tilde\f}\in\mathcal M(\mathcal F;\Gamma)$ induces an isomorphism
of  bigraded $\mathbb Q$-linear spaces
$$
\lkrv(\Gamma)_{\bullet\bullet}\simeq
\ARI(\Gamma)_{\push/ \pusnu}\cap \ARI(\Gamma)_\al^{\fin,\pol}.
$$
\end{thm}

\begin{proof}
Our map decomposes as
\begin{equation}\label{eq:gr krv to ARIal}
\bar F\in\lkrv(\Gamma)_{\bullet\bullet}\mapsto {\tilde \f}\in\mt
\mapsto \ma_{\tilde\f}\in\ARI(\Gamma)_\al^{\fin,\pol}.
\end{equation}
The first map is injective and the second one is isomorphic by 
Proposition \ref{prop:MT=ARIal}.
Our claim follows by our previous lemmas.
\end{proof}

As a generalization of \cite[Proposition 1]{RS}, we obtain the following.

\begin{thm}\label{thm:lkrv Lie alg}
The space $\lkrv(\Gamma)_{\bullet\bullet}$ forms a 
bigraded Lie algebra under the bracket \eqref{eq:mt bracket}.
\end{thm}

\begin{proof}
As for the morphism \eqref{eq:gr krv to ARIal},
the second map is Lie algebra homomorphism by Proposition \ref{prop:MT=ARIal}.
It is easy to see that the first map  forms a Lie algebra homomorphism
when we encode $\lkrv(\Gamma)_{\bullet\bullet}$ with the bracket \eqref{eq:mt bracket}.
Thus our claim follows because it is shown that 
%
%
$\ARI(\Gamma)_{\push/ \pusnu}$ forms a Lie algebra by Theorem \ref{thm: ARIpushpusnu Lie}
and so $\ARI(\Gamma)_\al$  does by Lemma \ref{ARIal Lie algebra}.
\end{proof}

\begin{rem}
By  \eqref{eq:inclusion grARI}, \eqref{eq:inclusion grKRV}, Theorem \ref{thm:reform:krv} and Theorem \ref{thm:reform:krv:bigrade},
we obtain the following commutative diagram 
of Lie algebras.
\begin{equation}\label{CD:krv}
\xymatrix{ 
\gr_{\D}\krv(\Gamma)_{\bullet}\ar^{\!\!\!\!\!\!\!\!\!\!\!\!\!\!\!\!\!\!\!\!\!\!\!\!\!\!\!\!\!\!\!\!\!\!\!\!\!\!\simeq}[r]\ar@{^{(}->}[d]& \gr_{\D}(\ARI(\Gamma)_{\pspush/ \pusnu}\cap \ARI(\Gamma)_\al^{\fin,\pol})\ar@{^{(}->}[d]\\ 
\lkrv(\Gamma)_{\bullet\bullet}
\ar^{\!\!\!\!\!\!\!\!\!\!\!\!\!\!\!\!\!\!\!\!\!\!\!\!\!\!\!\!\!\!\!\!\!\!\!\simeq}[r]&\ARI(\Gamma)_{\push/ \pusnu}\cap \ARI(\Gamma)_\al^{\fin,\pol}
}
\end{equation}
The diagram is commutative because 
\eqref{eq:gr krv to ARIal} is associated graded with \eqref{eq:krv into ARIal}.
\end{rem}
%

Similarly to  Definition \ref{def:ARID}, we impose a distribution relation on 
$\lkrv(\Gamma)$.

\begin{defn}\label{def:LKRVD}
For $N\geqslant 1$ and $\Gamma^N=\{g^N\in\Gamma\ |\  g\in\Gamma\}$,
we consider the map $i_N:\lkrv(\Gamma)_{\bullet\bullet}\to\lkrv(\Gamma^N)_{\bullet\bullet}$
which  is  induced by 
\begin{equation*}
i_N(x)=x, \qquad
i_N(y_{\tau})=
\begin{cases}
y_{\tau}& \text{ when } \tau\in\Gamma^N, \\
0 & \text{ when } \tau\not\in\Gamma^N
\end{cases}
\end{equation*}
and also the map $m_N:\lkrv(\Gamma)_{\bullet\bullet}\to\lkrv(\Gamma^N)_{\bullet\bullet}$
which  is  induced by
\begin{equation*}
m_N(x)=Nx, \qquad
m_N(y_{\sigma})=y_{\sigma^N}
\end{equation*}
for $\sigma\in\Gamma$.
We define the following $\mathbb Q$-linear subspace
$$
\lkrvd(\Gamma)_{\bullet\bullet}:=\{\bar F\in\lkrv(\Gamma)_{\bullet\bullet}\bigm | i_N(\bar F)=m_N(\bar F) \text{ for all }
N\geqslant 1\}.
$$
\end{defn}

As a corollary of Theorems \ref{thm:reform:krv:bigrade} and \ref{thm:lkrv Lie alg},
we obtain the following corollary.

\begin{cor}\label{cor:reform:lkrvd}
The space $\lkrvd(\Gamma)$ forms a bigraded Lie algebra and
we have the following isomorphism of bigraded Lie algebras
$$
\lkrvd(\Gamma)_{\bullet\bullet}\simeq
\ARI(\Gamma)_{\push/ \pusnu}\cap \ARID(\Gamma)_\al^{\fin,\pol}.
$$
\end{cor}

\begin{proof}
By definition  we see that both $i_N$ and $m_N$ form Lie algebra homomorphisms
and by Proposition \ref{ARI Lie algebra}
$\ARI(\Gamma)$ forms a Lie algebra.
Therefore our claim follows from
the commutativity of the following diagrams and 
Proposition \ref{prop:ARID Lie algebra}:
\begin{equation*}
\xymatrix{ 
\lkrv(\Gamma)_{\bullet\bullet}\ar[r]\ar[d]^{i_N}& \ARI(\Gamma)\ar[d]^{i_N}\\ 
\lkrv(\Gamma^N)_{\bullet\bullet}\ar[r]&\ARI(\Gamma^N),
}
\qquad
\xymatrix{ 
\lkrv(\Gamma)_{\bullet\bullet}\ar[r]\ar[d]^{m_N}& \ARI(\Gamma)\ar[d]^{m_N}\\ 
\lkrv(\Gamma^N)_{\bullet\bullet}\ar[r]&\ARI(\Gamma^N).
}
\end{equation*}
\end{proof}

\section{Dihedral Lie algebra}\label{sec:dihedral Lie algebra}
By using mould theoretic interpretations of 
the bigraded Lie algebra 
${\mathbb D}(\Gamma)_{\bullet\bullet}$ ($\Gamma$: a finite abelian group)
with a dihedral symmetry 
and of the Kashiwara-Vergne bigraded Lie algebra 
$\lkrv(\Gamma)_{\bullet\bullet}$,
we realize an embedding $\Fil_{\D}^{2}{\mathbb D}(\Gamma)_{\bullet\bullet}\hookrightarrow\lkrv(\Gamma)_{\bullet\bullet}$
which extends the result of \cite{RS}.

\subsection{Dihedral bigraded Lie algebra}
We recall the definition of the dihedral bigraded Lie algebra 
$D(\Gamma)_{\bullet\bullet}$ introduced by Goncharov \cite{G}
and introduce a related Lie algebra ${\mathbb D}(\Gamma)_{\bullet\bullet}$
which contains $D(\Gamma)_{\bullet\bullet}$
in Definitions \ref{defn: Goncharov dihedral bigraded Lie algebra} and \ref{defn: alal dihedral Lie algebra}. 

We call an element $f=f(t_1,\dots,t_{m+1})$ in $\mathbb Q[t_1,\dots,t_{m+1}]$ {\it translation invariant} when 
$f(t_1,\dots,t_{m+1})=f(t_1+c,\dots,t_{m+1}+c)$ for any $c\in\mathbb Q$.
We often denote this  by $f(t_1:\dots:t_{m+1})$.
We consider a set of collections
\begin{equation}\label{eq:undertilde Z}
\undertilde Z=\left\{\undertilde Z(g_1, \dots, g_m,g_{m+1} |t_1: \dots :t_{m+1})\right\}_{g_1,\dots,g_m\in\Gamma}
\end{equation}
with 
$$g_{m+1}=(g_1\cdots g_m)^{-1}$$
of 
translation invariant element  in $\mathbb Q[t_1,\dots,t_{m+1}]$. 
For such $\undertilde Z$, we associate
\begin{align}\label{eq:Goncharov swap}
\tilde Z&(g_1:\cdots:g_{m+1}|t_1,\dots,t_{m+1}):= \\ 
&\undertilde Z(g_1^{-1}g_2,g_2^{-1}g_3,\dots,g_m^{-1}g_{m+1},g_{m+1}^{-1}g_1
|t_1:t_1+t_2: \dots :t_1+\cdots+t_m:0)  \notag
\end{align}
with any $g_1,\dots, g_{m+1}\in\Gamma$ and $t_1+\cdots+t_{m+1}=0$ and also
\begin{align}
\notag
Z(g_1:&\cdots:g_{m+1}|t_1:\dots:t_{m+1}) \\
\label{eq: Z1}
&:=\tilde Z(g_1:\cdots:g_{m+1}|
t_1-t_{m+1},t_2-t_1,\dots ,t_m-t_{m-1}, t_{m+1}-t_m) \\ 
\label{eq: Z2}
& \ =\undertilde Z(g_1^{-1}g_2,g_2^{-1}g_3,\dots,g_m^{-1}g_{m+1},g_{m+1}^{-1}g_1
|t_1:t_2: \dots :t_m:t_{m+1})  
\end{align}
with any $g_1,\dots, g_{m+1}\in\Gamma$ and any $t_1, \dots, t_{m+1}$.

\begin{defn}[\cite{G}]
\label{defn: Goncharov dihedral bigraded Lie algebra}
Set-theoretically 
the {\it dihedral bigraded Lie algebra}
means the $\mathbb Q$-linear bigraded space
$$D(\Gamma)_{\bullet\bullet}=\bigoplus_{w,m}D(\Gamma)_{w,m},$$
where the bidegree $(w,m)$-part $D(\Gamma)_{w,m}$ is defined to be
the set $\undertilde Z$ in \eqref{eq:undertilde Z} 
of 
translation invariant elements 
in $\mathbb Q[t_1,\dots,t_{m+1}]$ 
with total degree $w-m$
satisfying 

\noindent
(a). {\it the double shuffle relation}, that is,

(a-i).
{\it the harmonic product}
\begin{equation}\label{harmonic product}
\sum\nolimits_{\sigma\in\Sha_{p,q}}\undertilde Z(
g_{\sigma(1)},\dots,g_{\sigma(m)},g_{m+1}|
t_{\sigma(1)}:\dots:t_{\sigma(m)}:t_{m+1})=0
\end{equation}

(a-ii). {\it the shuffle product}
\begin{equation}\label{shuffle product}
\sum\nolimits_{\sigma\in\Sha_{p,q}}
\tilde Z(g_{\sigma(1)}:\dots:g_{\sigma(m)}:g_{m+1}|
t_{\sigma(1)},\dots,t_{\sigma(m)}, t_{m+1})=0
\end{equation}
for any $p,q\geqslant 1$ with $p+q=m$.


\noindent
(b). {\it the distribution relation} for $N\in\mathbb Z$
such that $|N|$ divides $|\Gamma|$,
\begin{equation}\label{distribution relation}
Z(g_1: \dots :g_{m+1} |t_1:\dots:t_{m+1})
=\frac{1}{|\Gamma_N|}\sum_{h_i^N=g_i}Z(h_1: \dots: h_{m+1} |Nt_1:\dots:Nt_{m+1})
\end{equation}
except that  a constant in $t$ is allowed when $m=1$ and $g_1=g_2$.
Here $|\Gamma_N|$ is the  order of the $N$-torsion subgroup of $\Gamma$
and $\Sha_{p,q}$ is defined by \eqref{shuffle permutation}.

\noindent
(c). Additionally  we put 
\begin{equation}\label{eq:additional}
Z(e:e|0:0)=0
\end{equation}
for a technical reason.
\end{defn}

Its Lie algebra structure is explained in \cite{G} \S\S 4-5.
In \cite[Theorem 4.1]{G},
it is shown that 
the double shuffle relation implies the dihedral symmetry relations:

\begin{thm}[\cite{G} Theorem 4.1]
\label{ds to dihedral}
The double shuffle relation implies
{\it the dihedral symmetry relations},
which consist of 
{\it the cyclic symmetry relation}
\begin{equation*}\label{cyclic symmetry relation}
Z(g_1:g_2:\cdots :g_{m+1}|t_1:t_2:\dots:t_{m+1})=
Z(g_2:\dots:g_{m+1}:g_1|t_2:\dots:t_{m+1}:t_1),
\end{equation*}
{\it the inversion relation} (\lq the distribution relation for $N=-1$')
\begin{equation*}\label{inversion relation}
Z(g_1: \dots:g_{m+1} |t_1: \dots :t_{m+1})
=Z(g_1^{-1}: \dots:g_{m+1}^{-1} |-t_1: \dots :-t_{m+1}).
\end{equation*}
{\it the reflection relation}
\begin{equation*}\label{reflection relation}
Z(g_1:g_2:\cdots :g_{m+1}|t_1:\dots:t_m:t_{m+1})=
(-1)^{m+1}Z(g_{m+1}:\cdots : g_1|-t_m:\dots:-t_1:-t_{m+1}),
\end{equation*}
for $m\geqslant 2$.
\end{thm}

The following reformulation of the above dihedral symmetry relations
is useful in our later arguments.
\begin{rem}
The transformation \eqref{eq: Z2} allows us to rewrite
the above dihedral symmetry relations as follows:

\noindent
{\it the cyclic symmetry relation}
\begin{equation}\label{refo cyclic symmetry relation}
\undertilde Z(g_1,g_2,\dots ,g_{m+1}|t_1:t_2:\dots:t_{m+1})=
\undertilde Z(g_2,\dots,g_{m+1},g_1|t_2:\dots:t_{m+1}:t_1),
\end{equation}
{\it the inversion relation}
\begin{equation}\label{refo inversion relation}
\undertilde Z(g_1,g_2,\dots ,g_{m+1}|t_1:t_2:\dots:t_{m+1})=
\undertilde Z(g_1^{-1},g_2^{-1},\dots,g_{m+1}^{-1}|-t_1:-t_2:\dots:-t_{m+1}),
\end{equation}
{\it the reflection relation}
\begin{equation}\label{refo reflection relation}
\undertilde Z(g_1,\dots, g_m ,g_{m+1}|t_1:\dots:t_m:t_{m+1})=
(-1)^{m+1}\undertilde Z(g_m^{-1},\dots , g_1^{-1},g_{m+1}^{-1}|-t_{m}:\dots:-t_1:-t_{m+1})
\end{equation}
with $g_1 g_2\cdots g_{m+1}=1$.
\end{rem}

\begin{defn}\label{defn: alal dihedral Lie algebra}
Goncharov \cite[\S 4.5 and \S 5.2]{G} also introduced a related Lie algebra 
$$D'(\Gamma)_{\bullet\bullet}=\oplus_{w,m}{D'}(\Gamma)_{w,m}$$
which consists of the collections $\undertilde Z$ 
satisfying the shuffle product \eqref{shuffle product},
the cyclic symmetry relation \eqref{cyclic symmetry relation} and
the additional condition \eqref{eq:additional} (cf. \cite[Proposition 4.6]{G}).
For our purpose, we consider its $\mathbb Q$-linear subspace
$${\mathbb D}(\Gamma)_{\bullet\bullet}=\oplus_{w,m}{\mathbb D}(\Gamma)_{w,m}$$
which consists of elements in $D'(\Gamma)_{\bullet\bullet}$ 
satisfying the double shuffle relations (a).
\end{defn}

By Theorem \ref{ds to dihedral}, we have
$$
D'(\Gamma)_{\bullet\bullet}\supset
\mathbb D(\Gamma)_{\bullet\bullet}\supset
D(\Gamma)_{\bullet\bullet}.
$$

\subsection{Mould theoretic reformulation}
We explain a reformulation of ${\mathbb D}(\Gamma)_{\bullet\bullet}$ and
$D(\Gamma)_{\bullet\bullet}$ in terms of moulds in 
Theorem \ref{thm:reform:dihedral}
and Corollary \ref{cor:reform:dihedral} respectively.
%

Let $\undertilde Z$ be a collection \eqref{eq:undertilde Z}
of translation invariant elements.
We associate a mould
$${M}_{\undertilde Z}=({M}_{\undertilde Z}^i)_{i\in\mathbb Z_{\geqslant 0}}\in \mathcal M(\mathcal F;\Gamma)$$
by
${M}_{\undertilde Z}^i=0$ for $i\neq m$ and
$${M}_{\undertilde Z}^m({}^{u_1,\dots,u_m}_{g_1,\cdots,g_m})=\tilde Z(g_1:\cdots:g_{m}:1|u_1,\dots,u_{m+1}).$$

%

The following is a generalization of the results in \cite{M}, \cite{S-ARIGARI}
which treat the case of $\Gamma=\{e\}$.

\begin{thm}
\label{thm:reform:dihedral}
The map sending $M:\undertilde Z\in D'(\Gamma)_{\bullet\bullet}
\mapsto M_{\undertilde Z}\in \ARI(\Gamma)$ 
forms a Lie algebra homomorphism
and  it induces an isomorphism between
\begin{equation}\label{eq:Fil2D=Fil2ARIalal}
\Fil^2_\D {\mathbb D}(\Gamma)_{\bullet\bullet}\simeq
\Fil^2_\D \ARI(\Gamma)_{\underline{\al}/ \underline{\al}}^{\fin,\pol}.
\end{equation}
Here the left hand side means the depth>1-part of ${\mathbb  D}(\Gamma)_{\bullet\bullet}$
and the right hand side means the finite polynomial-valued part of
the subset of 
$\ARI(\Gamma)_{\underline{\al}/\underline{\al}}$
(cf. Definition \ref{def:ARIalal})
consisting of $M$ with depth>1.
\end{thm}

\begin{proof}
It is immediate that  that \eqref{shuffle product} is equivalent to
the condition for ${M}_{\undertilde Z}$ being in $\ARI(\Gamma)_\al$.
By \eqref{eq:Goncharov swap}, we have
\begin{align*}
\swap(& M^m_{\undertilde Z})\varia{g_1,\dots,g_m}{v_1,\dots,v_m}\\
&=\tilde Z(g_1\cdots g_m:g_1\cdots g_{m-1}:\dots: g_1:1|v_m,v_{m-1}-v_m,\dots,v_2-v_3,v_1-v_2,-v_1)\\
&=\undertilde Z(g_m^{-1},g_{m-1}^{-1},\dots,g_1^{-1},g_1\cdots g_m\ |\ v_m:\dots:v_2:v_1:0).
\end{align*}
Thus we see that \eqref{harmonic product} is equivalent to
the condition for $\swap({M}_{\undertilde Z})$ being in $\overline{\ARI}(\Gamma)_\al$.
Therefore our map forms a $\mathbb Q$-linear isomorphism \eqref{eq:Fil2D=Fil2ARIalal}
by Theorem \ref{ds to dihedral}.

Since $M_{\undertilde Z}$ is in $\ARI(\Gamma)_\al^{\fin,\pol}$
when $\undertilde Z\in D'(\Gamma)_{\bullet\bullet}$,
by Proposition \ref{prop:MT=ARIal}
there is an $h\in\mathbb L$ with depth $m$ such that
$$\ma(h)=M_{\undertilde Z},$$
that is,
$$
\ma_h^m({}^{u_1,\dots,u_m}_{g_1,\cdots,g_m})=\tilde Z(g_1:\cdots:g_{m}:1|u_1,\dots,u_{m+1})
$$
with $u_1+\cdots+ u_m+u_{m+1}=0$.
By \eqref{eq:translation invariance of vimo} and \eqref{eq: Z1}, we have
$$
\vimo^m_h({}^{u_0,\dots,u_m}_{g_1,\cdots,g_m})=
Z(g_1:\cdots:g_{m}:1|u_1:\dots :u_{m}:u_0).
$$

In \cite[Theorem 5.2]{G},  $D'(\Gamma)_{\bullet\bullet}$ is realized as  a Lie subalgebra
of $\sder^\Gamma$ under the morphism 
\begin{equation}\label{eq:map xi prime}
\xi'_\Gamma: D'(\Gamma)_{\bullet\bullet}\to \gr_\D\sder^\Gamma
\end{equation}
sending each 
$\undertilde Z=\left\{\undertilde Z(g_1,\dots,g_{m+1}|t_1:\dots:t_{m+1})\right\}_{g_1,\dots,g_{m+1}\in\Gamma}
\in D_{w,m}(\Gamma)$
to the residue class of $D_{\{\sigma(F)\}_\sigma,G(F)}\in\sder^\Gamma$
in $\gr_\D^m\sder^\Gamma$
with 
\begin{align*}
F&=|\Gamma|^{-1}
\sum_{n_1,\dots,n_{m+1}>0
\atop g_1,\dots,g_{m+1}\in\Gamma}
I_{n_1,\dots,n_{m+1}}(g_1:\cdots:g_{m+1})X^{n_1-1}Y_{g_1^{-1}g_2}\cdots 
X^{n_{m}-1}
Y_{g_1^{-1}g_{m+1}}X^{n_{m+1}-1} \\
&=
\sum_{n_0,\dots,n_{m}>0
\atop g_1,\dots,g_{m}\in\Gamma}
I_{n_0,\dots,n_{m}}(1:g_1:\cdots:g_{m})X^{n_0-1}Y_{g_1}\cdots 
X^{n_{m-1}-1}
Y_{g_{m}}X^{n_{m}-1}\in\mathbb L_{w,m}
\end{align*}
when the associated element
$Z=\left\{Z(g_1:\cdots:g_{m+1}|t_1:\dots:t_{m+1})\right\}_{g_1,\dots,g_{m+1}\in\Gamma}
$ given by \eqref{eq: Z2} is expressed as 
$$
Z(g_1:\cdots:g_{m+1}|t_1:\dots:t_{m+1}):
=\sum_{n_1,\dots,n_{m+1}>0}I_{n_1,\dots,n_{m+1}}(g_1:\cdots: g_{m+1})t_1^{n_1-1}\cdots t_{m+1}^{n_{m+1}-1}
$$
in $\mathbb Q[t_1,\dots,t_{m+1}]$.
We note that $I_{n_1,\dots,n_{m+1}}(g_1:\cdots: g_{m+1})$ is uniquely determined.
By combining the Lie algebra homomorphisms $\xi'_\Gamma$ with \eqref{eq:res} and \eqref{eq:hom ma},
we obtain
$$
\ma\circ\res\circ\xi'_\Gamma(\undertilde Z)=\ma(F).
$$

By definition, we have
\begin{align*}
\vimo_F^m({}^{u_1,\dots,u_m}_{g_1^{-1},\dots,g_m^{-1}})
&=
\sum_{n_1,\dots,n_{m+1}>0}I_{n_0,\dots,n_{m}}(1:g_1:\cdots:g_{m})
u_0^{n_0-1}\cdots u_m^{n_m-1} \\
&=
Z(1:g_1:\cdots:g_m|u_0:\cdots: u_m). \\
\intertext{By the cyclic symmetry relation \eqref{cyclic symmetry relation}, }
&=
Z(g_1:\cdots:g_m:1|u_1:\cdots: u_m:u_0) 
=
\vimo_h^m({}^{u_1,\dots,u_m}_{g_1,\dots,g_m}).
\end{align*}
Therefore we have
$
F(x, (y_\sigma)_\sigma)=
h(x,(y_{\sigma^{-1}})_\sigma).
$
So 
$$M(\undertilde Z)=\ma(h)=
\iota\circ\ma\circ\res\circ\xi'_\Gamma(\undertilde Z)$$
where $\iota$ is the map defined by 
$(\iota M)^m({}^{u_1,\dots,u_m}_{g_1,\dots,g_m})=
M^m({}^{u_1,\dots,u_m}_{g_1^{-1},\dots,g_m^{-1}})$
which forms a Lie algebra homomorphism.
Since $\iota$, $\ma$ in \eqref{eq:hom ma}, 
$\res$ in \eqref{eq:res} and 
$\xi'_\Gamma$ in \eqref{eq:map xi prime}
are all Lie algebra homomorphisms,
we see that $M$ is so. Whence we get our claims.
\end{proof}

Since 
$\ARI(\Gamma)_{\underline{\al}/ \underline{\al}}$,
and hence the right hand side of \eqref{eq:Fil2D=Fil2ARIalal},
forms a Lie algebra by Proposition \ref{ARIalal Lie algebra},
we learn that 
$\Fil^2_\D {\mathbb D}(\Gamma)_{\bullet\bullet}$
forms a Lie subalgebra of $D'(\Gamma)_{\bullet\bullet}$.

\begin{cor}\label{cor:reform:dihedral}
The map $M$ in Theorem \ref{thm:reform:dihedral} induces 
a Lie algebra isomorphism between
\begin{equation}\label{eq:D=ARIDalal}
\Fil^2_\D {D}(\Gamma)_{\bullet\bullet}\simeq
\Fil^2_\D \ARID(\Gamma)_{\underline{\al}/ \underline{\al}}^{\fin,\pol}.
\end{equation}
Here the left hand side means the depth>1-part of ${D}(\Gamma)_{\bullet\bullet}$
and the right hand side means the finite polynomial-valued part of
the subset of 
$\ARID(\Gamma)_{\underline{\al}/\underline{\al}}$
(cf. Definition \ref{def:ARID})
consisting of $M$ with depth>1.

\end{cor}

\begin{proof}
Since the distribution relation
\eqref{distribution relation} 
is equivalent to
$$
\undertilde Z(g_1, \dots ,g_{m+1} |t_1:\dots:t_{m+1})
=
\sum_{h_i^N=g_i}\undertilde Z(h_1, \dots, h_{m+1} |Nt_1:\dots:Nt_{m+1}),
$$
which corresponds to 
$i_N(M_{\undertilde Z})=m_N(M_{\undertilde Z})$,
that is, $M_{\undertilde Z}\in\ARID(\Gamma)$.
Since 
$\ARID(\Gamma)_{\underline{\al}/ \underline{\al}}$,
and hence the right hand side of \eqref{eq:D=ARIDalal},
forms a Lie algebra by 
Corollary \ref{cor:ARID alal Lie algebra}, and
$\Fil^2_\D {\mathbb D}(\Gamma)_{\bullet\bullet}$
forms a Lie subalgebra of $D(\Gamma)_{\bullet\bullet}$ by \cite[Theorem 5.2]{G},
our claim follows.
\end{proof}
%
%
\subsection{Embedding}
In this subsection, we construct an embedding 
$\Fil_{\D}^{2}\ARI(\Gamma)_{\underline{\al}/ \underline{\al}}\hookrightarrow
\ARI(\Gamma)_{\push/ \pusnu}$ in Theorem \ref{thm:embedding}.
As a corollary, we get an embedding
$\Fil_{\D}^{2}{\mathbb D}(\Gamma)_{\bullet\bullet}\hookrightarrow\lkrv(\Gamma)_{\bullet\bullet}$
in Corollary \ref{cor:embedding}.
%
%

The following generalizes {\cite[Lemma 2.5.3]{S-ARIGARI}}.

\begin{lem} 
\label{lem:Embd:1}
Any mould $M\in\ARI(\Gamma)_{\al}$ is $\mantar$-invariant 
(cf. Notation \ref{nota:mould operations}), namely,  
for $m\geqslant1$ and $\sigma_1,\dots,\sigma_m\in\Gamma$, we have
\begin{equation}\label{mantar invariant}
	M^m\varia{x_1,\ \dots,\ x_m}{\sigma_1,\ \dots,\ \sigma_m} = (-1)^{m-1}M^m\varia{x_m,\ \dots,\ x_1}{\sigma_m,\ \dots,\ \sigma_1}.
\end{equation}
\end{lem}
\begin{proof}
For simplicity, we denote $\omega_i:=\binom{x_i}{\sigma_i}$. By using alternality of $M$, we have
\begin{align*}
	\sum_{i=1}^{m-1}(-1)^{i-1}
	\left\{ \sum_{\alpha\in X_{\Z}^\bullet}\Sh{(\omega_i,\dots,\omega_1)}{(\omega_{i+1},\dots,\omega_m)}{\alpha}M^m(\alpha) \right\}
	=0.
\end{align*}
Here, we calculate the left hand side as follows:
\begin{align*}
	\sum_{i=1}^{m-1}&(-1)^{i-1}
		\left\{ \sum_{\alpha\in X_{\Z}^\bullet}\Sh{(\omega_i,\dots,\omega_1)}{(\omega_{i+1},\dots,\omega_m)}{\alpha}M^m(\alpha) \right\} \\
	=&\sum_{i=1}^{m-1}(-1)^{i-1}
		\left\{ \sum_{\alpha\in X_{\Z}^\bullet}\Sh{(\omega_{i-1},\dots,\omega_1)}{(\omega_{i+1},\dots,\omega_m)}{\alpha}M^m(\omega_i,\alpha) \right. \\
	&\left.+ \sum_{\alpha\in X_{\Z}^\bullet}\Sh{(\omega_i,\dots,\omega_1)}{(\omega_{i+2},\dots,\omega_m)}{\alpha}M^m(\omega_{i+1},\alpha) \right\} \\
	=&\sum_{i=1}^{m-1}(-1)^{i-1}
		\sum_{\alpha\in X_{\Z}^\bullet}\Sh{(\omega_{i-1},\dots,\omega_1)}{(\omega_{i+1},\dots,\omega_m)}{\alpha}M^m(\omega_i,\alpha) \\
	&- \sum_{i=2}^{m}(-1)^{i-1}
		\sum_{\alpha\in X_{\Z}^\bullet}\Sh{(\omega_{i-1},\dots,\omega_1)}{(\omega_{i+1},\dots,\omega_m)}{\alpha}M^m(\omega_i,\alpha) \\
	=&M^m(\omega_1,\omega_2,\dots,\omega_m) 
		-(-1)^{m-1}
		M^m(\omega_m,\omega_{m-1},\dots,\omega_1).
\end{align*}
So we obtain \eqref{mantar invariant}.
\end{proof}

We define the parallel translation map $\h:\overline{\ARI}(\Gamma)\rightarrow \overline{\ARI}(\Gamma)$ by
\begin{align*}
	\h(M)^m(\vecy_m):=
	\left\{\begin{array}{ll}
	M^m(\vecy_m) & (m=0,1), \\
	M^{m-1}{\scriptsize\left(\begin{array}{rrr}
	\sigma_2,& \dots,& \sigma_m \\
	y_2-y_1,& \dots,& y_m-y_1
\end{array}\right)} & (m\geqslant2),
	\end{array}\right.
\end{align*}
for $M\in \overline{\ARI}(\Gamma)$. For our simplicity, we put
$$\tswap:=\h\circ\swap.$$

\begin{lem}\label{lem:Embd:2}
Let $M\in \ARI(\Gamma)_{\underline{\al}/ \underline{\al}}$ and $m\geqslant2$ and $\sigma_1,\dots,\sigma_m\in\Gamma$ with $\sigma_1\cdots\sigma_m=1$. Then we have
\begin{equation}\label{eqn:G70}
	\tswap(M)^m\varia{\sigma_1,\ \dots,\ \sigma_m}{x_1,\ \dots,\ x_m}
	= \tswap(M)^m\varia{\sigma_2^{-1},\ \dots,\ \ \sigma_m^{-1},\ \sigma_1^{-1}}{-x_2,\ \dots,\ -x_m,\ -x_1}.
\end{equation}
\end{lem}

\begin{proof}
We have
\begin{align*}
	\tswap(M)^m&\varia{\sigma_1,\ \dots,\ \sigma_m}{x_1,\ \dots,\ x_m}  \\
	=&\swap(M)^{m-1}\varia{\ \quad\sigma_2,\qquad\sigma_3,\ \dots,\ \quad \sigma_{m-1},\ \quad \sigma_m}
	{x_2-x_1,\ x_3-x_1,\ \dots,\ x_{m-1}-x_1,\ x_m-x_1} \\
	=&M^{m-1}\varia{x_m-x_1,\ x_{m-1}-x_m,\ \dots,\ x_2-x_3}
	{\sigma_2\cdots\sigma_m,\ \sigma_2\cdots\sigma_{m-1},\ \dots,\quad \sigma_2}.
	\intertext{By using \eqref{mantar invariant}, we get}
	=&(-1)^{m-2}M^{m-1}\varia{x_2-x_3,\ \dots,\ x_{m-1}-x_m,\ x_m-x_1}
	{\ \quad \sigma_2,\ \dots,\ \sigma_2\cdots\sigma_{m-1},\ \sigma_2\cdots\sigma_m} \\ 
	=&(-1)^{m-2}\swap(M)^{m-1}\varia{\sigma_2\cdots\sigma_m,\quad \sigma_m^{-1},\ \dots,\quad \sigma_3^{-1}}
	{x_2-x_1,\ x_2-x_m,\ \dots,\ x_2-x_3}. 
	\intertext{By \eqref{mantar invariant}, $\sigma_1\cdots\sigma_m=1$ and
	$M\in\ARI(\Gamma)_{\underline{\al}/ \underline{\al}}$, we calculate}
	=&\swap(M)^{m-1}\varia{\quad \sigma_3^{-1},\ \dots,\ \quad \sigma_m^{-1},\quad\sigma_1^{-1}}
	{x_2-x_3,\ \dots,\ x_2-x_m,\ x_2-x_1} \\
	=&\tswap(M)^m\varia{\sigma_2^{-1},\ \sigma_3^{-1},\ \dots,\ \ \sigma_m^{-1},\ \sigma_1^{-1}}
	{-x_2,\ -x_3,\ \dots,\ -x_m,\ -x_1}.
\end{align*}
Therefore, we obtain the claim.
\end{proof}

The following three relations should be called as the cyclic symmetry relation, the inversion relation and the reflection relation respectively (compare them with \eqref{refo cyclic symmetry relation}, \eqref{refo inversion relation} and \eqref{refo reflection relation}
with $m+1$ replaced with $m$).

\begin{lem}\label{lem:Embd:3}
Let $m\geqslant3$ and $M\in\ARI(\Gamma)_{\underline{\al}/ \underline{\al}}$. Then, for $\sigma_1,\dots,\sigma_m\in\Gamma$ with $\sigma_1\cdots\sigma_m=1$, we have the following:
\begin{align}
	\label{eqn:cyc}&\tswap(M)^m\varia{\sigma_1,\ \sigma_2,\ \dots,\ \sigma_m}{x_1,\ x_2,\ \dots,\ x_m}
	= \tswap(M)^m\varia{\sigma_2,\ \dots,\ \sigma_m,\ \sigma_1}{x_2,\ \dots,\ x_m,\ x_1}, \\
	\label{eqn:inv}&\tswap(M)^m\varia{\sigma_1,\ \sigma_2,\ \dots,\ \sigma_m}{x_1,\ x_2,\ \dots,\ x_m}
	 = \tswap(M)^m\varia{\sigma_1^{-1},\ \sigma_2^{-1},\ \dots,\ \sigma_m^{-1}}{-x_1,\ -x_2,\ \dots,\ -x_m}, \\
	\label{eqn:ref}&
	\tswap(M)^m\varia{\sigma_1,\ \dots,\ \sigma_{m-1},\ \sigma_m}{x_1,\ \dots,\ x_{m-1},\ x_m} 
	= (-1)^{m}\tswap(M)^{m}\varia{\ \ \sigma_{m-1}^{-1},\ \dots,\ \sigma_{1}^{-1},\ \ \sigma_m^{-1}}{-x_{m-1},\ \dots,\ -x_{1},\ -x_m}.
\end{align}
\end{lem}

\begin{proof}
It is easy to get \eqref{eqn:inv} from \eqref{eqn:G70} and \eqref{eqn:cyc} and to get \eqref{eqn:ref} from \eqref{mantar invariant}, \eqref{eqn:cyc} and \eqref{eqn:inv}.
So it is enough to prove \eqref{eqn:cyc}. 

Firstly, we have
\begin{align*}
	&\sum_{\alpha\in Y_{\Z}^\bullet}
	\Sh{\varia{\sigma_1}{x_1}}{\varia{\sigma_2,\ \dots,\ \sigma_m}{x_2,\ \dots,\ x_m}}{\alpha}\tswap(M)^m(\alpha) \\
	& =\tswap(M)^m\varia{\sigma_1,\ \sigma_2,\ \dots,\ \sigma_m}{x_1,\ x_2,\ \dots,\ x_m}
	+\sum_{\alpha\in Y_{\Z}^\bullet}
	\Sh{\varia{\sigma_1}{x_1}}{\varia{\sigma_3,\ \dots,\ \sigma_m}{x_3,\ \dots,\ x_m}}{\alpha}\tswap(M)^m(\varia{\sigma_2}{x_2},\alpha) \\
	& =\tswap(M)^m\varia{\sigma_1,\ \dots,\ \sigma_m}{x_1,\ \dots,\ x_m} \\
	&\hspace{2cm}+\sum_{\alpha\in Y_{\Z}^\bullet}
	\Sh{\varia{\ \ \ \sigma_1}{x_1-x_2}}{\varia{\ \quad\sigma_3,\ \dots,\qquad \sigma_m}{x_3-x_2,\ \dots,\ x_m-x_2}}{\alpha}\swap(M)^{m-1}(\alpha).
\end{align*}
So by using alternality of $\swap(M)$, we obtain
\begin{equation}\label{first cyclic term}
	\sum_{\alpha\in Y_{\Z}^\bullet}
	\Sh{\varia{\sigma_1}{x_1}}{\varia{\sigma_2,\ \dots,\ \sigma_m}{x_2,\ \dots,\ x_m}}{\alpha}
	\tswap(M)^m(\alpha)
	=\tswap(M)^m\varia{\sigma_1,\ \dots,\ \sigma_m}{x_1,\ \dots,\ x_m}.
\end{equation}
Secondly, we have
\begin{align*}
	&\sum_{\alpha\in Y_{\Z}^\bullet}
	\Sh{\varia{\sigma_1}{x_1}}{\varia{\sigma_2,\ \dots,\ \sigma_m}{x_2,\ \dots,\ x_m}}{\alpha}
	\tswap(M)^m(\alpha) \\
	& =\tswap(M)^m\varia{\sigma_2,\ \dots,\ \sigma_m,\ \sigma_1}{x_2,\ \dots,\ x_m,\ x_1}
	+\sum_{\alpha\in Y_{\Z}^\bullet}
	\Sh{\varia{\sigma_1}{x_1}}{\varia{\sigma_2,\ \dots,\ \sigma_{m-1}}{x_2,\ \dots,\ x_{m-1}}}{\alpha}\tswap(M)^m(\alpha,\varia{\sigma_m}{x_m}).
	\intertext{By applying Lemma \ref{lem:Embd:2} to the second term, we get}
	& =\tswap(M)^m\varia{\sigma_2,\ \dots,\ \sigma_m,\ \sigma_1}{x_2,\ \dots,\ x_m,\ x_1} \\
	&\hspace{2cm}+\sum_{\alpha\in Y_{\Z}^\bullet}
	\Sh{\varia{\sigma_1}{x_1}}{\varia{\sigma_2,\ \dots,\ \sigma_{m-1}}{x_2,\ \dots,\ x_{m-1}}}{\alpha}
	\tswap(M)^m(\varia{\ \sigma_m^{-1}}{-x_m},\alpha_-) \\
	& =\tswap(M)^m\varia{\sigma_2,\ \dots,\ \sigma_m,\ \sigma_1}{x_2,\ \dots,\ x_m,\ x_1} \\
	&\hspace{2cm}+\sum_{\alpha\in Y_{\Z}^\bullet}
	\Sh{\varia{\quad\sigma_1}{x_1-x_m}}{\varia{\ \ \quad\sigma_2,\ \dots,\qquad \sigma_{m-1}}{x_2-x_m,\ \dots,\ x_{m-1}-x_m}}{\alpha}\swap(M)^{m-1}(\alpha_-).
\end{align*}
Here, the symbol $\alpha_-$ is $\varia{\tau_1^{-1},\dots,\tau_m^{-1}}{-u_1,\dots,-u_m}$ for $\alpha=\varia{\tau_1,\dots,\tau_m}{u_1,\dots,u_m}$ and for $\binom{\tau_1}{u_1},\dots,\binom{\tau_m}{u_m}\in Y_{\Z}$. So by using alternality of $\swap(M)$, we obtain
\begin{equation}\label{second cyclic term}
	\sum_{\alpha\in Y_{\Z}^\bullet}
	\Sh{\varia{\sigma_1}{x_1}}{\varia{\sigma_2,\ \dots,\ \sigma_m}{x_2,\ \dots,\ x_m}}{\alpha}
	\tswap(M)^m(\alpha)
	=\tswap(M)^m\varia{\sigma_2,\ \dots,\ \sigma_m,\ \sigma_1}{x_2,\ \dots,\ x_m,\ x_1}.
\end{equation}
Therefore, by combining \eqref{first cyclic term} and \eqref{second cyclic term}, we obtain \eqref{eqn:cyc}.
\end{proof}

\begin{lem}
For any mould $M\in\ARI(\Gamma)_{\underline\al/\underline\al}$,
	its $\swap(M)$ is {\it neg-invariant} and {\it mantar-invariant}
	(cf. Notation \ref{nota:mould operations}) 
	namely we have  
	\begin{align}
	\label{eqn:inv2}&\swap(M)^m\varia{\sigma_1,\ \dots,\ \sigma_m}{x_1,\ \dots,\ x_m}
	 = \swap(M)^m\varia{\sigma_1^{-1},\ \dots,\ \sigma_m^{-1}}{-x_1,\ \dots,\ -x_m}, \\
	\label{eqn:ref2}&\swap(M)^m\varia{\sigma_1,\ \dots,\ \sigma_m}{x_1,\ \dots,\ x_m}
	 = (-1)^{m-1}\swap(M)^m\varia{\sigma_m^{-1},\ \dots,\ \sigma_1^{-1}}{-x_m,\ \dots,\ -x_1},
\end{align}
for $m\geqslant 0$.
\end{lem}

\begin{proof}
For $m\geqslant 2$, the first equation follows from \eqref{eqn:inv} and
the second one follows from \eqref{eqn:cyc} and \eqref{eqn:ref}.
For $m=0,1$,  they follow from the definition of $\ARI(\Gamma)_{\underline\al/\underline\al}$. 
\end{proof}

The following can be also found in \cite[Lemma 2.5.5]{S-ARIGARI} 
but we give a different proof below.

\begin{prop}\label{prop:DKV1}
	Any mould $ M\in\ARI(\Gamma)_{\underline\al/\underline\al}$ is push-invariant \eqref{push-invariant}. 
\end{prop}

\begin{proof}
	For $m=0,1$, it is obvious because we have $M^1\varia{x_1}{\sigma_1}=M^1\varia{-x_1}{\sigma_1^{-1}}$.
	Assume $m\geqslant2$. By using \eqref{eqn:cyc}, we have
	\begin{equation*}
		\tswap(M)^{m+1}\varia{\tau_1,\ \tau_2,\ \dots,\ \tau_{m+1}}{y_1,\ y_2,\ \dots,\ y_{m+1}}
		 = \tswap(M)^{m+1}\varia{\tau_2,\ \dots,\ \tau_{m+1},\ \tau_1}{y_2,\ \dots,\ y_{m+1},\ y_1},
	\end{equation*}
	for $\binom{\tau_1}{y_1},\dots, \binom{\tau_{m+1}}{y_{m+1}}\in Y_{\Z}$ with $\tau_1\cdots\tau_{m+1}=1$. We calculate the left hand side
	\begin{align*}
		\tswap(M)^{m+1}\varia{\tau_1,\ \tau_2,\ \dots,\ \tau_{m+1}}{y_1,\ y_2,\ \dots,\ y_{m+1}}
		& = \swap(M)^m\varia{\ \quad\tau_2,\ \dots,\ \quad \tau_{m+1}}
		{y_2-y_1,\ \dots,\ y_{m+1}-y_1} \\
		& = M^m
		{\scriptsize\left(\begin{array}{rrrr}
			y_{m+1}-y_1,& y_m-y_{m+1},& \dots,& y_2-y_3 \\
			\tau_2\cdots\tau_{m+1},& \tau_2\cdots\tau_{m},& \dots,& \tau_2
		\end{array}\right)}.
	\end{align*}
	Similarly, we calculate the right hand side
	\begin{align*}
		\tswap(M)^{m+1}\varia{\tau_2,\ \dots,\ \tau_{m+1},\ \tau_1}{y_2,\ \dots,\ y_{m+1},\ y_1}
		& = M^m
		{\scriptsize\left(\begin{array}{rrrrr}
			y_1-y_2,&  y_{m+1}-y_1,&  y_m-y_{m+1},& \dots,& y_3-y_4 \\
			\tau_3\cdots\tau_{m+1}\tau_1,& \tau_3\cdots\tau_{m+1},& \tau_3\cdots\tau_{m},& \dots,& \tau_3
		\end{array}\right)}.
	\end{align*}
	By substituting $y_1:=0$ and $y_i:=x_1+\cdots+x_{m+2-i}$ ($2\leqslant i\leqslant m+1$) and $\tau_1:=\sigma_1^{-1}$ and $\tau_2:=\sigma_m$ and $\tau_j:=\sigma_{m+2-j}\sigma_{m+3-j}^{-1}$ ($3\leqslant j\leqslant m+1$) for any $\sigma_1,\dots,\sigma_m\in\Gamma$, we have
	\begin{align*}
		M^m\varia{x_1,\ x_2,\ \dots,\ x_m}{\sigma_1,\ \sigma_2,\ \dots,\ \sigma_m}
		& = M^m{\scriptsize\left(\begin{array}{rrrr}
			-x_1-\cdots-x_m,& x_1,& \dots,& x_{m-1} \\
			\sigma_m^{-1},& \sigma_1\sigma_m^{-1},& \dots,& \sigma_{m-1}\sigma_m^{-1}
		\end{array}\right)} \\
		& = \push(M)^m\varia{x_1,\ x_2,\ \dots,\ x_m}{\sigma_1,\ \sigma_2,\ \dots,\ \sigma_m}.
	\end{align*}
	Whence we obtain the claim.
\end{proof}

The following generalizes 
\cite[Theorem 14]{RS}.

\begin{prop}\label{prop:DKV2}
	For any mould $ M\in\Fil_{\D}^{2}\ARI(\Gamma)_{\underline{\al}/ \underline{\al}}$, its swap $\swap( M)$
	is pus-neutral \eqref{pus-neutral}, that is, for all $m\geqslant 1$ and $\sigma_1,\dots,\sigma_m\in\Gamma$, 
	\begin{equation*}
		\sum_{i\in\Z/m\Z}\pus^i\circ\swap( M)^m
		\varia{\sigma_1,\dots,\sigma_m}{x_1,\dots,x_m}
		=0.
	\end{equation*}
\end{prop}

\begin{proof}
	For $m=0,1$, it is obvious because we have $M^1(x_1)=0$.
	Assume $m\geqslant2$. Let $\binom{\tau_0}{u_0},\binom{\tau_1}{u_1},\dots,\binom{\tau_m}{u_m}\in Y_{\Z}$ with $\tau_0\tau_1\cdots\tau_m=1$. By using alternality of $\swap(M)$, we have
	\begin{equation*}
		\sum_{\alpha\in Y_{\Z}^\bullet}
		\Sh{\scriptsize\left(\begin{array}{ccc}
			\tau_2,& \dots,& \tau_m \\
			u_2-u_1,& \dots,& u_m-u_1
		\end{array}\right)}{\scriptsize\left(\begin{array}{c}
			\tau_0 \\
			u_0-u_1
		\end{array}\right)}{\alpha}\swap(M)^m(\alpha)=0.
	\end{equation*}
	On the other hand, we calculate
	\begin{align*}
		&\sum_{\alpha\in Y_{\Z}^\bullet}
		\Sh{\scriptsize\left(\begin{array}{ccc}
			\tau_2,& \dots,& \tau_m \\
			u_2-u_1,& \dots,& u_m-u_1
		\end{array}\right)}{\scriptsize\left(\begin{array}{c}
		    \tau_0\\
			u_0-u_1 
		\end{array}\right)}{\alpha}
		\swap(M)^m(\alpha) \\
		=&\tswap(M)^{m+1}\varia{\tau_1,\ \tau_2,\ \dots,\ \tau_m,\ \tau_0}
			{u_1,\ u_2,\ \dots,\ u_m,\ u_0}
		+\tswap(M)^{m+1}\varia{\tau_1,\ \tau_2,\ \dots,\ \tau_{m-1},\ \tau_0,\ \tau_m}
			{u_1,\ u_2,\ \dots,\ u_{m-1},\ u_0,\ u_m} \\
		&\hspace{0cm}+\cdots+\tswap(M)^{m+1}\varia{\tau_1,\ \tau_2,\ \tau_0,\ \tau_3,\ \dots,\ \tau_m}
			{u_1,\ u_2,\ u_0,\ u_3,\ \dots,\ u_m}
		+\tswap(M)^{m+1}\varia{\tau_1,\ \tau_0,\ \tau_2,\ \dots,\ \tau_m}
			{u_1,\ u_0,\ u_2,\ \dots,\ u_m}.
		\intertext{Repeated applications of \eqref{eqn:cyc} yield}
		=&\tswap(M)^{m+1}\varia{\tau_0,\ \tau_1,\ \tau_2,\ \dots,\ \tau_m}
			{u_0,\ u_1,\ u_2,\ \dots,\ u_m}
		+\tswap(M)^{m+1}\varia{\tau_0,\ \tau_m,\ \tau_1,\ \tau_2,\ \dots,\ \tau_{m-1}}
			{u_0,\ u_m,\ u_1,\ u_2,\ \dots,\ u_{m-1}} \\
		&\hspace{0cm}+\cdots+\tswap(M)^{m+1}\varia{\tau_0,\ \tau_3,\ \dots,\ \tau_m,\ \tau_1,\ \tau_2}
			{u_0,\ u_3,\ \dots,\ u_m,\ u_1,\ u_2}
		+\tswap(M)^{m+1}\varia{\tau_0,\ \tau_2,\ \dots,\ \tau_m,\ \tau_1}
			{u_0,\ u_2,\ \dots,\ u_m,\ u_1} \\
		=&\sum_{i\in\Z/m\Z}\tswap(M)^{m+1}\varia{\tau_0,\ \tau_{i+1},\ \tau_{i+2},\ \dots,\ \tau_{i+m}}
			{u_0,\ u_{i+1},\ u_{i+2},\ \dots,\ u_{i+m}}.
	\end{align*}
	Therefore, by substituting $u_0=0$ and $u_i=x_i$ ($1\leqslant i\leqslant m$) and $\tau_0=(\sigma_1\cdots\sigma_m)^{-1}$ and $\tau_i=\sigma_i$ ($1\leqslant i\leqslant m$), we get the claim.
\end{proof}

\begin{thm}\label{thm:embedding}
There is an embedding
$$
\Fil_{\D}^{2}\ARI(\Gamma)_{\underline{\al}/ \underline{\al}}\hookrightarrow
\ARI(\Gamma)_{\push/ \pusnu}
$$
of  graded Lie algebras.
\end{thm}

\begin{proof}
	Let $ M\in\Fil_{\D}^{2}\ARI(\Gamma)_{\underline{\al}/ \underline{\al}}$. Then we have $M\in\ARI(\Gamma)_{\underline\al/\underline\al}$. So by Proposition \ref{prop:DKV1} and Proposition \ref{prop:DKV2},  $M$ is $\push$-invariant 
	and $\swap(M)$ is $\pus$-neutral. 
	Therefore, we obtain $ M\in\ARI(\Gamma)_{\push/ \pusnu}$.
	To see that it is a Lie algebra homomorphism is immediate.
\end{proof}

As a corollary, 
by taking an intersection with $\ARI(\Gamma)_\al^{\fin,\pol}$ in the embedding of the above theorem,
we obtain the following  inclusion 
which generalizes \cite[Theorem 3]{RS}.

\begin{cor}\label{cor:embedding}
There is an embedding
$$
\Fil_{\D}^{2}{\mathbb D}(\Gamma)_{\bullet\bullet}\hookrightarrow\lkrv(\Gamma)_{\bullet\bullet}
$$
of bigraded Lie algebras.
\end{cor}

\begin{proof}
It follows from  Theorem \ref{thm:reform:krv:bigrade},
\eqref{eq:Fil2D=Fil2ARIalal}
and Theorem \ref{thm:embedding}.
\end{proof}

By imposing the distribution relation, we also obtain the following.

\begin{cor}\label{cor:embedding2}
There is an embedding
$$
\Fil_{\D}^{2}{D}(\Gamma)_{\bullet\bullet}\hookrightarrow\lkrvd(\Gamma)_{\bullet\bullet}
$$
of bigraded Lie algebras.
\end{cor}

\begin{proof}
It follows from 
Corollary \ref{cor:reform:lkrvd},
Corollary \ref{cor:reform:dihedral}
and Theorem \ref{thm:embedding}.
\end{proof}

It looks interesting to see if the map is an isomorphism.

\begin{rem}
By \eqref{eq:Fil2D=Fil2ARIalal}, Theorem \ref{thm:reform:krv:bigrade},
Theorem \ref{thm:embedding} and Corollary \ref{cor:embedding},
we obtain the commutative diagram \eqref{CD:D} of Lie algebras:
\begin{equation}\label{CD:D}
\xymatrix{ 
\Fil_{\D}^{2}{\mathbb D}(\Gamma)_{\bullet\bullet}\ar^{\!\!\!\!\!\!\!\!\!\!\!\!\!\!\!\!\!\!\!\!\!\!\!\!\!\!\!\!\!\!\!\!\!\simeq}[r]\ar@{^{(}->}[d]& \Fil_{\D}^{2}\ARI(\Gamma)_{\underline{\al}/ \underline{\al}}^{\fin,\pol}\ar@{^{(}->}[d]\\ 
\lkrv(\Gamma)_{\bullet\bullet}
\ar^{\!\!\!\!\!\!\!\!\!\!\!\!\!\!\!\!\!\!\!\!\!\!\!\!\!\!\!\!\!\!\!\!\!\!\!\simeq}[r]&\ARI(\Gamma)_{\push/ \pusnu}\cap \ARI(\Gamma)_\al^{\fin,\pol}
}
\end{equation}
\end{rem}

%
%
\appendix
\section{On the $\ari$-bracket of  $\ARI(\Gamma)$}\label{sec:appendix}
In this appendix, we give self-contained proofs of
fundamental properties of the $\ari$-bracket of $\ARI(\Gamma)$, that is,
Proposition \ref{ARI Lie algebra}, \ref{ARIal Lie algebra} and \ref{ARIalal Lie algebra},
which are required  in this paper.

\subsection{Proof of Proposition \ref{ARI Lie algebra}}\label{sec:A.1}
We give a proof Proposition \ref{ARI Lie algebra} which claims that  $\ARI(\Gamma)$ forms a Lie algebra,
by showing that it actually forms a pre-Lie algebra.

 We start with three fundamental lemmas 
 which can be proved directly by simple computations.
\begin{lem}\label{lem:flexion's action}
	For $\alpha,\alpha_1,\alpha_2,\beta,\beta_1,\beta_2\in X_{\Z}^\bullet$, we have
	\begin{align*}
		\urflex{\alpha_1\alpha_2}{\beta}=\urflex{\alpha_1}{(\urflex{\alpha_2}{\beta})}, &\quad
		\ulflex{\beta}{\alpha_1\alpha_2}=\ulflex{(\ulflex{\beta}{\alpha_1})}{\alpha_2}, \\
		\lrflex{\alpha}{(\beta_1\beta_2)}=\lrflex{\alpha}{\beta_1}\lrflex{\alpha}{\beta_2}, &\quad
		\llflex{(\beta_1\beta_2)}{\alpha}=\llflex{\beta_1}{\alpha}\llflex{\beta_2}{\alpha}.
	\end{align*}
	Especially, if $\beta_1\neq\emptyset$, we have
	\begin{align*}
		\urflex{\alpha}{(\beta_1\beta_2)}=(\urflex{\alpha}{\beta_1}){\beta_2}, \quad
		\ulflex{(\beta_2\beta_1)}{\alpha}={\beta_2}(\ulflex{\beta_1}{\alpha}),
	\end{align*}
	and if $\alpha_2\neq\emptyset$, we have
	\begin{align*}
		\lrflex{\alpha_1\alpha_2}{\beta}=\lrflex{\alpha_2}{\beta}, &\quad
		\llflex{\beta}{\alpha_2\alpha_1}=\llflex{\beta}{\alpha_2}.
	\end{align*}
\end{lem}
\begin{proof}
Let $\alpha_1=\varia{a_1,\dots,a_l}{\sigma_1,\dots,\sigma_l},\alpha_2=\varia{a_{l+1},\dots,a_{l+m}}{\sigma_{l+1},\dots,\sigma_{l+m}},\beta=\varia{b_1,\dots,b_n}{\tau_1,\dots,\tau_n}$.
By $\alpha_1\alpha_2=\varia{a_1,\dots,a_{l+m}}{\sigma_1,\dots,\sigma_{l+m}}$, we have
\begin{equation*}
	\urflex{\alpha_1\alpha_2}{\beta}
	=\varia{a_1+\cdots+a_{l+m}+b_1,b_2,\dots,b_n}{\quad\qquad\qquad\tau_1,\tau_2,\dots,\tau_n}
	=\urflex{\alpha_1}{\varia{a_{l+1}+\cdots+a_{l+m}+b_1,b_2,\dots,b_n}{\qquad\qquad\qquad\tau_1,\tau_2,\dots,\tau_n}}
	=\urflex{\alpha_1}{(\urflex{\alpha_2}{\beta})}.
\end{equation*}
All the other cases can shown similarly.
\end{proof}

\begin{lem}\label{lem:tri-factorization}
	Let $a,b,c,d,e,f\in X_{\Z}^\bullet$.
	\begin{enumerate}
		\item If we have $a\urflex{b}{c}=def$, one of the following holds:
		\begin{equation*}
		\begin{array}{rl}
			{\rm (I):}& \mbox{There exists $a_1,a_2,a_3\in X_{\Z}^\bullet$ such that} \\
			&\mbox{\qquad $a=a_1a_2a_3$ and $(d,e,f)=(a_1,a_2,a_3\urflex{b}{c})$,} \\
			{\rm (II):}& \mbox{There exists $a_1,a_2,c_1,c_2\in X_{\Z}^\bullet$ such that} \\
			&\mbox{\qquad $a=a_1a_2$ and $c=c_1c_2$ and $(d,e,f)=(a_1,a_2\urflex{b}{c_1},c_2)$ and $c_1\neq\emptyset$,} \\
			{\rm (III):}& \mbox{There exists $c_1,c_2,c_3\in X_{\Z}^\bullet$ such that} \\
			&\mbox{\qquad $c=c_1c_2c_3$ and $(d,e,f)=(a\urflex{b}{c_1},c_2,c_3)$ and $c_1\neq\emptyset$.}
		\end{array}
		\end{equation*}
	\item If we have $\ulflex{a}{b}{c}=def$, one of the followings holds:
		\begin{equation*}
		\begin{array}{rl}
			{\rm (I):}& \mbox{There exists $a_1,a_2,a_3\in X_{\Z}^\bullet$ such that} \\
			&\mbox{\qquad $a=a_1a_2a_3$ and $(d,e,f)=(a_1,a_2,\ulflex{a_3}{b}{c})$ and $a_3\neq\emptyset$,} \\
			{\rm (II):}& \mbox{There exists $a_1,a_2,c_1,c_2\in X_{\Z}^\bullet$ such that} \\
			&\mbox{\qquad $a=a_1a_2$ and $c=c_1c_2$ and $(d,e,f)=(a_1,\ulflex{a_2}{b}{c_1},c_2)$ and $a_2\neq\emptyset$,} \\
			{\rm (III):}& \mbox{There exists $c_1,c_2,c_3\in X_{\Z}^\bullet$ such that} \\
			&\mbox{\qquad $c=c_1c_2c_3$ and $(d,e,f)=(\ulflex{a}{b}{c_1},c_2,c_3)$.}
		\end{array}
		\end{equation*}
	\end{enumerate}
\end{lem}

\begin{proof}
We present a proof for (1).
When $a\urflex{b}{c}=def$,
the following depict all the possible cases of  
the positions of $a$, $\urflex{b}{c}$, $d$, $e$ and $f$.
\begin{equation*}
	\hspace{-2.1cm}
	\overbrace{(\underbrace{\hspace{.4cm}}_{d}|\underbrace{\hspace{.4cm}}_{e}|\hspace{.4cm})}^{a}
		\overbrace{(\hspace{.6cm}\hspace{.6cm})}^{\urflex{b}{c}}
	\hspace{-2.cm}\underbrace{\hspace{1.8cm}}_{f}
		\hspace{.5em} ,\quad
	\overbrace{(\underbrace{\hspace{.6cm}}_{d}|\hspace{.6cm})}^{a}
		\overbrace{(\hspace{.6cm}|\underbrace{\hspace{.6cm}}_{f})}^{\urflex{b}{c}}
	\hspace{-2.5cm}\underbrace{\hspace{1.5cm}}_{e}
	\hspace{2.8em} ,\quad
	\overbrace{(\hspace{.6cm}\hspace{.6cm})}^{a}
		\overbrace{(\hspace{.4cm}|\underbrace{\hspace{.4cm}}_{e}|\underbrace{\hspace{.4cm}}_{f})}^{\urflex{b}{c}}.
	\hspace{-3.9cm}\underbrace{\hspace{1.8cm}}_{d}
\end{equation*}
The first, the second and the third  cases correspond (I), (II)  and (III) in (1) respectively
\footnote{Note that, in order to decompose $\urflex{b}{c}$, we need the condition $c_1\neq\emptyset$ for the second and third cases.}. 
The claim for (2) can be proved in the same way.
\end{proof}

\begin{lem}\label{lem:flexions property}
For $a,b,c\in X_{\Z}^\bullet$, the following hold:
\begin{enumerate}
	\item {\rm  (commutativity)}
	\begin{align*}
		&\urflex{a}{(\urflex{b}{c})}= \urflex{b}{(\urflex{a}{c})},\quad
		\ulflex{(\ulflex{c}{a})}{b}= \ulflex{(\ulflex{c}{b})}{a},\quad
		\lrflex{a}{(\lrflex{b}{c})}= \lrflex{b}{(\lrflex{a}{c})},\quad
		\llflex{(\llflex{c}{a})}{b}= \llflex{(\llflex{c}{b})}{a}, \\
		&\urflex{a}{(\ulflex{c}{b})}= \ulflex{(\urflex{a}{c})}{b},\quad
		\urflex{a}{(\lrflex{b}{c})}= \lrflex{b}{(\urflex{a}{c})},\quad
		\urflex{a}{(\llflex{c}{b})}= \llflex{(\urflex{a}{c})}{b},\quad
		\ulflex{(\lrflex{b}{c})}{a}= \lrflex{b}{(\ulflex{c}{a})}, \\
		&\ulflex{(\llflex{c}{b})}{a}= \llflex{(\ulflex{c}{a})}{b},\quad
		\lrflex{a}{(\llflex{c}{b})}= \llflex{(\lrflex{a}{c})}{b}.
	\end{align*}
	\item {\rm  (composition)}
	\begin{align*}
		&\urflex{(\ulflex{b}{a})}{c}=\urflex{(\urflex{a}{b})}{c}=\urflex{ab}{c},\quad 
		\lrflex{(\llflex{b}{a})}{(\llflex{c}{a})}=\lrflex{(\lrflex{a}{b})}{(\lrflex{a}{c})}=\lrflex{b}{c}, \\
		&\ulflex{c}{(\ulflex{b}{a})}=\ulflex{c}{(\urflex{a}{b})}=\ulflex{c}{ab},\quad 
		\llflex{(\llflex{c}{a})}{(\llflex{b}{a})}=\llflex{(\lrflex{a}{c})}{(\lrflex{a}{b})}=\llflex{c}{b}.
	\end{align*}
	\item {\rm  (independence)}
	\begin{align*}
		&\urflex{(\llflex{b}{a})}{c}=\urflex{(\lrflex{a}{b})}{c}=\urflex{b}{c},\quad 
		\lrflex{(\ulflex{b}{a})}{c}=\lrflex{(\urflex{a}{b})}{c}=\lrflex{b}{c}, \\
		&\ulflex{c}{(\llflex{b}{a})}=\ulflex{c}{(\lrflex{a}{b})}=\ulflex{c}{b},\quad 
		\llflex{c}{(\ulflex{b}{a})}=\llflex{c}{(\urflex{a}{b})}=\llflex{c}{b}.
	\end{align*}
\end{enumerate}
\end{lem}
\begin{proof}
Let $a=\varia{a_1,\dots,a_l}{\sigma_1,\dots,\sigma_l},b=\varia{b_1,\dots,b_m}{\tau_1,\dots,\tau_m},c=\varia{c_1,\dots,c_n}{\mu_1,\dots,\mu_n}$. 
We give proofs for specific cases because all the other cases
 can  be proved in a similar way.  \\
(1).  We calculate
\begin{align*}
	&\urflex{a}{(\urflex{b}{c})}
	=\urflex{a}{\varia{b_1+\cdots+b_m+c_1,c_2,\dots,c_n}{\qquad\qquad\mu_1,\mu_2,\dots,\mu_n}}
	=\varia{a_1+\cdots+a_l+b_1+\cdots+b_m+c_1,c_2,\dots,c_n}{\qquad\qquad\qquad\qquad\mu_1,\mu_2,\dots,\mu_n} \\
	&\hspace{6cm}=\urflex{b}{\varia{a_1+\cdots+a_l+c_1,c_2,\dots,c_n}{\qquad\qquad\mu_1,\mu_2,\dots,\mu_n}}
	=\urflex{b}{(\urflex{a}{c})}.
\end{align*}
(2). By
$\ulflex{b}{a}=\varia{b_1,\dots,b_{m-1},b_m+a_1,\dots,a_l}{\tau_1,\dots,\tau_{m-1},\tau_m}$
and $ab=\varia{a_1,\dots,a_l,b_1,\dots,b_m}{\sigma_1,\dots,\sigma_l,\tau_1,\dots,\tau_m}$,
we get
$$
\urflex{(\ulflex{b}{a})}{c}
=\varia{a_1+\cdots+a_l+b_1+\cdots+b_m+c_1,c_2,\dots,c_n}{\qquad\qquad\qquad\qquad\mu_1,\mu_2,\dots,\mu_n}
=\urflex{ab}{c}.
$$
On the other hand, by $\llflex{b}{a}=\varia{\ \quad b_1,\dots,\quad b_m}{\tau_1\sigma_1^{-1},\dots,\tau_m\sigma_1^{-1}}$
and $\llflex{c}{a}=\varia{\ \quad c_1,\dots,\quad c_n}{\mu_1\sigma_1^{-1},\dots,\mu_n\sigma_1^{-1}}$, we get
$$
\lrflex{(\llflex{b}{a})}{(\llflex{c}{a})}
=\varia{\ \quad c_1,\dots,\quad c_n}{\mu_1\tau_m^{-1},\dots,\mu_n\tau_m^{-1}}
=\lrflex{b}{c}.
$$
(3). By the above expression of $\llflex{b}{a}$, we get
$$
\urflex{(\llflex{b}{a})}{c}
=\varia{b_1+\cdots+b_m+c_1,c_2,\dots,c_n}{\qquad\qquad\mu_1,\mu_2,\dots,\mu_n}
=\urflex{b}{c}.
$$
\end{proof}

The following formula is essential to prove Proposition \ref{ARI Lie algebra}.
It is stated in \cite[(2.2.10)]{S-ARIGARI}
(where it looks that there is an error on the  signature) 
without a proof.

\begin{prop}\label{prop:arit-ari}
For any $A,B\in\ARI(\Gamma)$, we have
\begin{equation*}
	\aritu(B)\circ\aritu(A)-\aritu(A)\circ\aritu(B)=\aritu(\ariu(A,B)).
\end{equation*}
\end{prop}
\begin{proof}
Let $m\geqslant0$. Then we have
\begin{align*}
	&(\aritu(B)\circ\aritu(A)-\aritu(A)\circ\aritu(B))(C)(\vecx_m) \\
	&=\aritu(B)(\aritu(A)(C))(\vecx_m)-\aritu(A)(\aritu(B)(C))(\vecx_m) \\
	&=\sum_{\substack{
		\vecx_m=abc \\
		b,c\neq\emptyset}}
	\aritu(A)(C)(a \urflex{b}{c})B(\llflex{b}{c})
	-\sum_{
		\substack{\vecx_m=abc \\
		a,b\neq\emptyset}}
	\aritu(A)(C)(\ulflex{a}{b} c)B(\lrflex{a}{b}) \\
	&\quad -\sum_{\substack{
		\vecx_m=abc \\
		b,c\neq\emptyset}}
	\aritu(B)(C)(a \urflex{b}{c})A(\llflex{b}{c})
	+\sum_{
		\substack{\vecx_m=abc \\
		a,b\neq\emptyset}}
	\aritu(B)(C)(\ulflex{a}{b} c)A(\lrflex{a}{b}) \\
	&=\sum_{\substack{
		\vecx_m=abc \\
		b,c\neq\emptyset}}
	\left\{
		\sum_{\substack{
			a \urflex{b}{c}=def \\
			e,f\neq\emptyset}}
		C(d \urflex{e}{f})A(\llflex{e}{f})
		-\sum_{\substack{
			a \urflex{b}{c}=def \\
			d,e\neq\emptyset}}
		C(\ulflex{d}{e} f)A(\lrflex{d}{e})
	\right\}
	B(\llflex{b}{c}) \\
	&\quad -\sum_{
		\substack{\vecx_m=abc \\
		a,b\neq\emptyset}}
	\left\{
		\sum_{\substack{
			\ulflex{a}{b} c=def \\
			e,f\neq\emptyset}}
		C(d \urflex{e}{f})A(\llflex{e}{f})
		-\sum_{\substack{
			\ulflex{a}{b} c=def \\
			d,e\neq\emptyset}}
		C(\ulflex{d}{e} f)A(\lrflex{d}{e})
	\right\}
	B(\lrflex{a}{b}) \\
	&\quad -\sum_{\substack{
		\vecx_m=abc \\
		b,c\neq\emptyset}}
	\left\{
		\sum_{\substack{
			a \urflex{b}{c}=def \\
			e,f\neq\emptyset}}
		C(d \urflex{e}{f})B(\llflex{e}{f})
		-\sum_{\substack{
			a \urflex{b}{c}=def \\
			d,e\neq\emptyset}}
		C(\ulflex{d}{e} f)B(\lrflex{d}{e})
	\right\}
	A(\llflex{b}{c}) \\
	&\quad +\sum_{
		\substack{\vecx_m=abc \\
		a,b\neq\emptyset}}
	\left\{
		\sum_{\substack{
			\ulflex{a}{b} c=def \\
			e,f\neq\emptyset}}
		C(d \urflex{e}{f})B(\llflex{e}{f})
		-\sum_{\substack{
			\ulflex{a}{b} c=def \\
			d,e\neq\emptyset}}
		C(\ulflex{d}{e} f)B(\lrflex{d}{e})
	\right\}
	A(\lrflex{a}{b}).
\end{align*}
By using Lemma \ref{lem:tri-factorization}, we have
{\scriptsize
\begin{align*}
	&=\sum_{\substack{
		\vecx_m=abc \\
		b,c\neq\emptyset}}
	\left\{
		\sum_{\substack{
			a =a_1a_2a_3 \\
			a_2\neq\emptyset}}
		C({a_1} \urflex{a_2}{({a_3} \urflex{b}{c})})A(\llflex{a_2}{{a_3} \urflex{b}{c}})
		+\sum_{\substack{
			a=a_1a_2 \\
			c=c_1c_2 \\
			c_1,c_2\neq\emptyset}}
		C({a_1} \urflex{{a_2}\urflex{b}{c_1}}{c_2})A(\llflex{({a_2}\urflex{b}{c_1})}{c_2})
		+\sum_{\substack{
			c=c_1c_2c_3 \\
			c_1,c_2,c_3\neq\emptyset}}
		C(a\urflex{b}{c_1} \urflex{c_2}{c_3})A(\llflex{c_2}{c_3})
	\right. \\
	&\left.
	\quad
		-\sum_{\substack{
			c=c_1c_2c_3 \\
			c_1,c_2\neq\emptyset}}
		C(\ulflex{(a\urflex{b}{c_1})}{c_2} c_3)A(\lrflex{a\urflex{b}{c_1}}{c_2})
		-\sum_{\substack{
			a=a_1a_2 \\
			c=c_1c_2 \\
			a_1,c_1\neq\emptyset}}
		C(\ulflex{a_1}{{a_2}\urflex{b}{c_1}} c_2)A(\lrflex{a_1}{({a_2}\urflex{b}{c_1})})
		-\sum_{\substack{
			a =a_1a_2a_3 \\
			a_1,a_2\neq\emptyset}}
		C(\ulflex{a_1}{a_2} {a_3} \urflex{b}{c})A(\lrflex{a_1}{a_2})
	\right\}
	B(\llflex{b}{c}) \\
	&-\sum_{
		\substack{\vecx_m=abc \\
		a,b\neq\emptyset}}
	\left\{
		\sum_{\substack{
			a =a_1a_2a_3 \\
			a_2,a_3\neq\emptyset}}
		C({a_1} \urflex{a_2}{(\ulflex{a_3}{b}{c})})A(\llflex{a_2}{\ulflex{a_3}{b}{c}})
		+\sum_{\substack{
			a=a_1a_2 \\
			c=c_1c_2 \\
			a_2,c_2\neq\emptyset}}
		C({a_1} \urflex{\ulflex{a_2}{b}{c_1}}{c_2})A(\llflex{(\ulflex{a_2}{b}{c_1})}{c_2})
		+\sum_{\substack{
			c=c_1c_2c_3 \\
			c_2,c_3\neq\emptyset}}
		C(\ulflex{a}{b}{c_1} \urflex{c_2}{c_3})A(\llflex{c_2}{c_3})
	\right. \\
	&\left.
	\quad
		-\sum_{\substack{
			c=c_1c_2c_3 \\
			c_2\neq\emptyset}}
		C(\ulflex{(\ulflex{a}{b}{c_1})}{c_2} c_3)A(\lrflex{\ulflex{a}{b}{c_1}}{c_2})
		-\sum_{\substack{
			a=a_1a_2 \\
			c=c_1c_2 \\
			a_1,a_2\neq\emptyset}}
		C(\ulflex{a_1}{\ulflex{a_2}{b}{c_1}} c_2)A(\lrflex{a_1}{(\ulflex{a_2}{b}{c_1})})
		-\sum_{\substack{
			a =a_1a_2a_3 \\
			a_1,a_2,a_3\neq\emptyset}}
		C(\ulflex{a_1}{a_2} \ulflex{a_3}{b}{c})A(\lrflex{a_1}{a_2})
	\right\}
	B(\lrflex{a}{b}) \\
	&-\sum_{\substack{
		\vecx_m=abc \\
		b,c\neq\emptyset}}
	\left\{
		\sum_{\substack{
			a =a_1a_2a_3 \\
			a_2\neq\emptyset}}
		C({a_1} \urflex{a_2}{({a_3} \urflex{b}{c})})B(\llflex{a_2}{{a_3} \urflex{b}{c}})
		+\sum_{\substack{
			a=a_1a_2 \\
			c=c_1c_2 \\
			c_1,c_2\neq\emptyset}}
		C({a_1} \urflex{{a_2}\urflex{b}{c_1}}{c_2})B(\llflex{({a_2}\urflex{b}{c_1})}{c_2})
		+\sum_{\substack{
			c=c_1c_2c_3 \\
			c_1,c_2,c_3\neq\emptyset}}
		C(a\urflex{b}{c_1} \urflex{c_2}{c_3})B(\llflex{c_2}{c_3})
	\right. \\
	&\left.
	\quad
		-\sum_{\substack{
			c=c_1c_2c_3 \\
			c_1,c_2\neq\emptyset}}
		C(\ulflex{(a\urflex{b}{c_1})}{c_2} c_3)B(\lrflex{a\urflex{b}{c_1}}{c_2})
		-\sum_{\substack{
			a=a_1a_2 \\
			c=c_1c_2 \\
			a_1,c_1\neq\emptyset}}
		C(\ulflex{a_1}{{a_2}\urflex{b}{c_1}} c_2)B(\lrflex{a_1}{({a_2}\urflex{b}{c_1})})
		-\sum_{\substack{
			a =a_1a_2a_3 \\
			a_1,a_2\neq\emptyset}}
		C(\ulflex{a_1}{a_2} {a_3} \urflex{b}{c})B(\lrflex{a_1}{a_2})
	\right\}
	A(\llflex{b}{c}) \\
	&+\sum_{
		\substack{\vecx_m=abc \\
		a,b\neq\emptyset}}
	\left\{
		\sum_{\substack{
			a =a_1a_2a_3 \\
			a_2,a_3\neq\emptyset}}
		C({a_1} \urflex{a_2}{(\ulflex{a_3}{b}{c})})B(\llflex{a_2}{\ulflex{a_3}{b}{c}})
		+\sum_{\substack{
			a=a_1a_2 \\
			c=c_1c_2 \\
			a_2,c_2\neq\emptyset}}
		C({a_1} \urflex{\ulflex{a_2}{b}{c_1}}{c_2})B(\llflex{(\ulflex{a_2}{b}{c_1})}{c_2})
		+\sum_{\substack{
			c=c_1c_2c_3 \\
			c_2,c_3\neq\emptyset}}
		C(\ulflex{a}{b}{c_1} \urflex{c_2}{c_3})B(\llflex{c_2}{c_3})
	\right. \\
	&\left.
	\quad
		-\sum_{\substack{
			c=c_1c_2c_3 \\
			c_2\neq\emptyset}}
		C(\ulflex{(\ulflex{a}{b}{c_1})}{c_2} c_3)B(\lrflex{\ulflex{a}{b}{c_1}}{c_2})
		-\sum_{\substack{
			a=a_1a_2 \\
			c=c_1c_2 \\
			a_1,a_2\neq\emptyset}}
		C(\ulflex{a_1}{\ulflex{a_2}{b}{c_1}} c_2)B(\lrflex{a_1}{(\ulflex{a_2}{b}{c_1})})
		-\sum_{\substack{
			a =a_1a_2a_3 \\
			a_1,a_2,a_3\neq\emptyset}}
		C(\ulflex{a_1}{a_2} \ulflex{a_3}{b}{c})B(\lrflex{a_1}{a_2})
	\right\}
	A(\lrflex{a}{b}).
\end{align*}}

By  using Lemma \ref{lem:flexion's action}
\footnote{We apply Lemma \ref{lem:flexion's action} to the 4th, 7th, 16th and 19th terms.}
and
Lemma \ref{lem:flexions property}.(2), (3)
\footnote{Especially, we apply Lemma \ref{lem:flexions property}.(2) to the middle terms of each lines, and apply Lemma \ref{lem:flexions property}.(3) to the first terms of each lines.}
and changing variables, we calculate
{\scriptsize
\begin{align*}
	&=\sum_{\substack{
			\vecx_m=abcde \\
			b,d,e\neq\emptyset}}
		C({a} \urflex{b}{({c} \urflex{d}{e})})A(\llflex{b}{ce})B(\llflex{d}{e})
		+\sum_{\substack{
			\vecx_m=abcde \\
			c,d,e\neq\emptyset}}
		C({a} \urflex{bcd}{e})A(\llflex{({b}\urflex{c}{d})}{e})B(\llflex{c}{de})
		+\sum_{\substack{
			\vecx_m=abcde \\
			b,c,d,e\neq\emptyset}}
		C(a\urflex{b}{c} \urflex{d}{e})A(\llflex{d}{e})B(\llflex{b}{cde}) \\
	&\quad
	-\sum_{\substack{
			\vecx_m=abcde \\
			b,c,d\neq\emptyset}}
		C(\ulflex{a\urflex{b}{c}}{d} e)A(\lrflex{ac}{d})B(\llflex{b}{cde})
		-\sum_{\substack{
			\vecx_m=abcde \\
			a,c,d\neq\emptyset}}
		C(\ulflex{a}{bcd} e)A(\lrflex{a}{({b}\urflex{c}{d})})B(\llflex{c}{de})
		-\sum_{\substack{
			\vecx_m=abcde \\
			a,b,d,e\neq\emptyset}}
		C(\ulflex{a}{b} {c} \urflex{d}{e})A(\lrflex{a}{b})B(\llflex{d}{e}) \\
	&\quad
	-\sum_{\substack{
			\vecx_m=abcde \\
			b,c,d\neq\emptyset}}
		C({a} \urflex{b}{\ulflex{c}{d}{e}})A(\llflex{b}{ce})B(\lrflex{abc}{d})
		-\sum_{\substack{
			\vecx_m=abcde \\
			b,c,e\neq\emptyset}}
		C({a} \urflex{bcd}{e})A(\llflex{(\ulflex{b}{c}{d})}{e})B(\lrflex{ab}{c})
		-\sum_{\substack{
			\vecx_m=abcde \\
			a,b,d,e\neq\emptyset}}
		C(\ulflex{a}{b}{c} \urflex{d}{e})A(\llflex{d}{e})B(\lrflex{a}{b}) \\
	&\quad
	+\sum_{\substack{
			\vecx_m=abcde \\
			a,b,d\neq\emptyset}}
		C(\ulflex{(\ulflex{a}{b}{c})}{d} e)A(\lrflex{ac}{d})B(\lrflex{a}{b})
		+\sum_{\substack{
			\vecx_m=abcde \\
			a,b,c\neq\emptyset}}
		C(\ulflex{a}{bcd} e)A(\lrflex{a}{(\ulflex{b}{c}{d})})B(\lrflex{ab}{c})
		+\sum_{\substack{
			\vecx_m=abcde \\
			a,b,c,d\neq\emptyset}}
		C(\ulflex{a}{b} \ulflex{c}{d}{e})A(\lrflex{a}{b})B(\lrflex{abc}{d}) \\
	&-\sum_{\substack{
			\vecx_m=abcde \\
			b,d,e\neq\emptyset}}
		C({a} \urflex{b}{({c} \urflex{d}{e})})B(\llflex{b}{ce})A(\llflex{d}{e})
		-\sum_{\substack{
			\vecx_m=abcde \\
			c,d,e\neq\emptyset}}
		C({a} \urflex{bcd}{e})B(\llflex{({b}\urflex{c}{d})}{e})A(\llflex{c}{de})
		-\sum_{\substack{
			\vecx_m=abcde \\
			b,c,d,e\neq\emptyset}}
		C(a\urflex{b}{c} \urflex{d}{e})B(\llflex{d}{e})A(\llflex{b}{cde}) \\
	&\quad
	+\sum_{\substack{
			\vecx_m=abcde \\
			b,c,d\neq\emptyset}}
		C(\ulflex{a\urflex{b}{c}}{d} e)B(\lrflex{ac}{d})A(\llflex{b}{cde})
		+\sum_{\substack{
			\vecx_m=abcde \\
			a,c,d\neq\emptyset}}
		C(\ulflex{a}{bcd} e)B(\lrflex{a}{({b}\urflex{c}{d})})A(\llflex{c}{de})
		+\sum_{\substack{
			\vecx_m=abcde \\
			a,b,d,e\neq\emptyset}}
		C(\ulflex{a}{b} {c} \urflex{d}{e})B(\lrflex{a}{b})A(\llflex{d}{e}) \\
	&\quad
	+\sum_{\substack{
			\vecx_m=abcde \\
			b,c,d\neq\emptyset}}
		C({a} \urflex{b}{\ulflex{c}{d}{e}})B(\llflex{b}{ce})A(\lrflex{abc}{d})
		+\sum_{\substack{
			\vecx_m=abcde \\
			b,c,e\neq\emptyset}}
		C({a} \urflex{bcd}{e})B(\llflex{(\ulflex{b}{c}{d})}{e})A(\lrflex{ab}{c})
		+\sum_{\substack{
			\vecx_m=abcde \\
			a,b,d,e\neq\emptyset}}
		C(\ulflex{a}{b}{c} \urflex{d}{e})B(\llflex{d}{e})A(\lrflex{a}{b}) \\
	&\quad
	-\sum_{\substack{
			\vecx_m=abcde \\
			a,b,d\neq\emptyset}}
		C(\ulflex{(\ulflex{a}{b}{c})}{d} e)B(\lrflex{ac}{d})A(\lrflex{a}{b})
		-\sum_{\substack{
			\vecx_m=abcde \\
			a,b,c\neq\emptyset}}
		C(\ulflex{a}{bcd} e)B(\lrflex{a}{(\ulflex{b}{c}{d})})A(\lrflex{ab}{c})
		-\sum_{\substack{
			\vecx_m=abcde \\
			a,b,c,d\neq\emptyset}}
		C(\ulflex{a}{b} \ulflex{c}{d}{e})B(\lrflex{a}{b})A(\lrflex{abc}{d}).
\end{align*}}
Cancellation
\footnote{The cancellations occur on the four pairs:  4th and 19th, 6th and 21th, 7th and 16th, 9th and 18th.}
yields
{\scriptsize
\begin{align*}
	&=\sum_{\substack{
			\vecx_m=abcde \\
			b,d,e\neq\emptyset}}
		C({a} \urflex{b}{({c} \urflex{d}{e})})A(\llflex{b}{ce})B(\llflex{d}{e})
		+\sum_{\substack{
			\vecx_m=abcde \\
			c,d,e\neq\emptyset}}
		C({a} \urflex{bcd}{e})A(\llflex{({b}\urflex{c}{d})}{e})B(\llflex{c}{de})
		+\sum_{\substack{
			\vecx_m=abcde \\
			b,c,d,e\neq\emptyset}}
		C(a\urflex{b}{c} \urflex{d}{e})A(\llflex{d}{e})B(\llflex{b}{cde}) \\
	&\quad
	\hphantom{-\sum_{\substack{
			\vecx_m=abcde \\
			b,c,d\neq\emptyset}}
		C(\ulflex{a\urflex{b}{c}}{d} e)A(\lrflex{ac}{d})B(\llflex{b}{cde})}
		-\sum_{\substack{
			\vecx_m=abcde \\
			a,c,d\neq\emptyset}}
		C(\ulflex{a}{bcd} e)A(\lrflex{a}{({b}\urflex{c}{d})})B(\llflex{c}{de})
		\hphantom{-\sum_{\substack{
			\vecx_m=abcde \\
			a,b,d,e\neq\emptyset}}
		C(\ulflex{a}{b} {c} \urflex{d}{e})A(\lrflex{a}{b})B(\llflex{d}{e})} \\
	&\quad
	\hphantom{-\sum_{\substack{
			\vecx_m=abcde \\
			b,c,d\neq\emptyset}}
		C({a} \urflex{b}{\ulflex{c}{d}{e}})A(\llflex{b}{ce})B(\lrflex{abc}{d})}
		-\sum_{\substack{
			\vecx_m=abcde \\
			b,c,e\neq\emptyset}}
		C({a} \urflex{bcd}{e})A(\llflex{(\ulflex{b}{c}{d})}{e})B(\lrflex{ab}{c})
		\hphantom{-\sum_{\substack{
			\vecx_m=abcde \\
			a,b,d,e\neq\emptyset}}
		C(\ulflex{a}{b}{c} \urflex{d}{e})A(\llflex{d}{e})B(\lrflex{a}{b})} \\
	&\quad
	+\sum_{\substack{
			\vecx_m=abcde \\
			a,b,d\neq\emptyset}}
		C(\ulflex{(\ulflex{a}{b}{c})}{d} e)A(\lrflex{ac}{d})B(\lrflex{a}{b})
		+\sum_{\substack{
			\vecx_m=abcde \\
			a,b,c\neq\emptyset}}
		C(\ulflex{a}{bcd} e)A(\lrflex{a}{(\ulflex{b}{c}{d})})B(\lrflex{ab}{c})
		+\sum_{\substack{
			\vecx_m=abcde \\
			a,b,c,d\neq\emptyset}}
		C(\ulflex{a}{b} \ulflex{c}{d}{e})A(\lrflex{a}{b})B(\lrflex{abc}{d}) \\
	&-\sum_{\substack{
			\vecx_m=abcde \\
			b,d,e\neq\emptyset}}
		C({a} \urflex{b}{({c} \urflex{d}{e})})B(\llflex{b}{ce})A(\llflex{d}{e})
		-\sum_{\substack{
			\vecx_m=abcde \\
			c,d,e\neq\emptyset}}
		C({a} \urflex{bcd}{e})B(\llflex{({b}\urflex{c}{d})}{e})A(\llflex{c}{de})
		-\sum_{\substack{
			\vecx_m=abcde \\
			b,c,d,e\neq\emptyset}}
		C(a\urflex{b}{c} \urflex{d}{e})B(\llflex{d}{e})A(\llflex{b}{cde}) \\
	&\quad
	\hphantom{+\sum_{\substack{
			\vecx_m=abcde \\
			b,c,d\neq\emptyset}}
		C(\ulflex{a\urflex{b}{c}}{d} e)B(\lrflex{ac}{d})A(\llflex{b}{cde})}
		+\sum_{\substack{
			\vecx_m=abcde \\
			a,c,d\neq\emptyset}}
		C(\ulflex{a}{bcd} e)B(\lrflex{a}{({b}\urflex{c}{d})})A(\llflex{c}{de})
		\hphantom{+\sum_{\substack{
			\vecx_m=abcde \\
			a,b,d,e\neq\emptyset}}
		C(\ulflex{a}{b} {c} \urflex{d}{e})B(\lrflex{a}{b})A(\llflex{d}{e})} \\
	&\quad
	\hphantom{+\sum_{\substack{
			\vecx_m=abcde \\
			b,c,d\neq\emptyset}}
		C({a} \urflex{b}{\ulflex{c}{d}{e}})B(\llflex{b}{ce})A(\lrflex{abc}{d})}
		+\sum_{\substack{
			\vecx_m=abcde \\
			b,c,e\neq\emptyset}}
		C({a} \urflex{bcd}{e})B(\llflex{(\ulflex{b}{c}{d})}{e})A(\lrflex{ab}{c})
		\hphantom{+\sum_{\substack{
			\vecx_m=abcde \\
			a,b,d,e\neq\emptyset}}
		C(\ulflex{a}{b}{c} \urflex{d}{e})B(\llflex{d}{e})A(\lrflex{a}{b})} \\
	&\quad
	-\sum_{\substack{
			\vecx_m=abcde \\
			a,b,d\neq\emptyset}}
		C(\ulflex{(\ulflex{a}{b}{c})}{d} e)B(\lrflex{ac}{d})A(\lrflex{a}{b})
		-\sum_{\substack{
			\vecx_m=abcde \\
			a,b,c\neq\emptyset}}
		C(\ulflex{a}{bcd} e)B(\lrflex{a}{(\ulflex{b}{c}{d})})A(\lrflex{ab}{c})
		-\sum_{\substack{
			\vecx_m=abcde \\
			a,b,c,d\neq\emptyset}}
		C(\ulflex{a}{b} \ulflex{c}{d}{e})B(\lrflex{a}{b})A(\lrflex{abc}{d}).
\end{align*}}

\noindent
Here, for the first term, we have
\begin{align*}
&\sum_{\substack{
	\vecx_m=abcde \\
	b,d,e\neq\emptyset}}
C({a} \urflex{b}{({c} \urflex{d}{e})})A(\llflex{b}{ce})B(\llflex{d}{e}) \\
&=\sum_{\substack{
	\vecx_m=abcde \\
	b,d,e\neq\emptyset \\
	c=\emptyset}}
C({a} \urflex{b}{({c} \urflex{d}{e})})A(\llflex{b}{ce})B(\llflex{d}{e})
+\sum_{\substack{
	\vecx_m=abcde \\
	b,d,e\neq\emptyset \\
	c\neq\emptyset}}
C({a} \urflex{b}{({c} \urflex{d}{e})})A(\llflex{b}{ce})B(\llflex{d}{e}).
\intertext{By using Lemma \ref{lem:flexion's action}, we get}
&=\sum_{\substack{
	\vecx_m=abcde \\
	b,d,e\neq\emptyset \\
	c=\emptyset}}
C({a} \urflex{b}{({c} \urflex{d}{e})})A(\llflex{b}{ce})B(\llflex{d}{e})
+\sum_{\substack{
	\vecx_m=abcde \\
	b,c,d,e\neq\emptyset}}
C({a} \urflex{b}{c}\urflex{d}{e})A(\llflex{b}{c})B(\llflex{d}{e}).
\end{align*}
Therefore, cancellation
\footnote{The cancellations occur on the four pairs when $c\neq\emptyset$:
1st and 11th, 3rd and 9th, 6th and 16th, 8th and 14th.}
yields
{\scriptsize
\begin{align*}
&(\aritu(B)\circ\aritu(A)-\aritu(A)\circ\aritu(B))(C)(\vecx_m) \\
	&=\sum_{\substack{
			\vecx_m=abcde \\
			b,d,e\neq\emptyset \\
			c=\emptyset}}
		C({a} \urflex{b}{({c} \urflex{d}{e})})A(\llflex{b}{ce})B(\llflex{d}{e})
		+\sum_{\substack{
			\vecx_m=abcde \\
			c,d,e\neq\emptyset}}
		C({a} \urflex{bcd}{e})A(\llflex{({b}\urflex{c}{d})}{e})B(\llflex{c}{de})
		\hphantom{
		+\sum_{\substack{
			\vecx_m=abcde \\
			b,c,d,e\neq\emptyset}}
		C(a\urflex{b}{c} \urflex{d}{e})A(\llflex{d}{e})B(\llflex{b}{cde})
		} \\
	&\quad
	\hphantom{
		-\sum_{\substack{
			\vecx_m=abcde \\
			a,b,d,e\neq\emptyset}}
		C(\ulflex{a}{b} {c} \urflex{d}{e})A(\lrflex{a}{b})B(\llflex{d}{e})
		}
		-\sum_{\substack{
			\vecx_m=abcde \\
			a,c,d\neq\emptyset}}
		C(\ulflex{a}{bcd} e)A(\lrflex{a}{({b}\urflex{c}{d})})B(\llflex{c}{de})
		\hphantom{
		-\sum_{\substack{
			\vecx_m=abcde \\
			b,c,d\neq\emptyset}}
		C(\ulflex{a\urflex{b}{c}}{d} e)A(\lrflex{ac}{d})B(\llflex{b}{cde})
		} \\
	&\quad
	\hphantom{
		-\sum_{\substack{
			\vecx_m=abcde \\
			b,c,d\neq\emptyset}}
		C({a} \urflex{b}{\ulflex{c}{d}{e}})A(\llflex{b}{ce})B(\lrflex{abc}{d})}
		-\sum_{\substack{
			\vecx_m=abcde \\
			b,c,e\neq\emptyset}}
		C({a} \urflex{bcd}{e})A(\llflex{(\ulflex{b}{c}{d})}{e})B(\lrflex{ab}{c})
		\hphantom{
		-\sum_{\substack{
			\vecx_m=abcde \\
			a,b,d,e\neq\emptyset}}
		C(\ulflex{a}{b}{c} \urflex{d}{e})A(\llflex{d}{e})B(\lrflex{a}{b})
		} \\
	&\quad
	+\sum_{\substack{
			\vecx_m=abcde \\
			a,b,d\neq\emptyset \\
			c=\emptyset}}
		C(\ulflex{(\ulflex{a}{b}{c})}{d} e)A(\lrflex{ac}{d})B(\lrflex{a}{b})
		+\sum_{\substack{
			\vecx_m=abcde \\
			a,b,c\neq\emptyset}}
		C(\ulflex{a}{bcd} e)A(\lrflex{a}{(\ulflex{b}{c}{d})})B(\lrflex{ab}{c})
		\hphantom{
		+\sum_{\substack{
			\vecx_m=abcde \\
			a,b,c,d\neq\emptyset}}
		C(\ulflex{a}{b} \ulflex{c}{d}{e})A(\lrflex{a}{b})B(\lrflex{abc}{d})
		} \\
	&-\sum_{\substack{
			\vecx_m=abcde \\
			b,d,e\neq\emptyset \\
			c=\emptyset}}
		C({a} \urflex{b}{({c} \urflex{d}{e})})B(\llflex{b}{ce})A(\llflex{d}{e})
		-\sum_{\substack{
			\vecx_m=abcde \\
			c,d,e\neq\emptyset}}
		C({a} \urflex{bcd}{e})B(\llflex{({b}\urflex{c}{d})}{e})A(\llflex{c}{de})
		\hphantom{
		-\sum_{\substack{
			\vecx_m=abcde \\
			b,c,d,e\neq\emptyset}}
		C(a\urflex{b}{c} \urflex{d}{e})B(\llflex{d}{e})A(\llflex{b}{cde})
		} \\
	&\quad
	\hphantom{
		+\sum_{\substack{
			\vecx_m=abcde \\
			a,b,d,e\neq\emptyset}}
		C(\ulflex{a}{b} {c} \urflex{d}{e})B(\lrflex{a}{b})A(\llflex{d}{e})
		}
		+\sum_{\substack{
			\vecx_m=abcde \\
			a,c,d\neq\emptyset}}
		C(\ulflex{a}{bcd} e)B(\lrflex{a}{({b}\urflex{c}{d})})A(\llflex{c}{de})
		\hphantom{
		+\sum_{\substack{
			\vecx_m=abcde \\
			b,c,d\neq\emptyset}}
		C(\ulflex{a\urflex{b}{c}}{d} e)B(\lrflex{ac}{d})A(\llflex{b}{cde})
		} \\
	&\quad
	\hphantom{
		+\sum_{\substack{
			\vecx_m=abcde \\
			b,c,d\neq\emptyset}}
		C({a} \urflex{b}{\ulflex{c}{d}{e}})B(\llflex{b}{ce})A(\lrflex{abc}{d})
		}
		+\sum_{\substack{
			\vecx_m=abcde \\
			b,c,e\neq\emptyset}}
		C({a} \urflex{bcd}{e})B(\llflex{(\ulflex{b}{c}{d})}{e})A(\lrflex{ab}{c})
		\hphantom{
		+\sum_{\substack{
			\vecx_m=abcde \\
			a,b,d,e\neq\emptyset}}
		C(\ulflex{a}{b}{c} \urflex{d}{e})B(\llflex{d}{e})A(\lrflex{a}{b})
		} \\
	&\quad
	-\sum_{\substack{
			\vecx_m=abcde \\
			a,b,d\neq\emptyset \\
			c=\emptyset}}
		C(\ulflex{(\ulflex{a}{b}{c})}{d} e)B(\lrflex{ac}{d})A(\lrflex{a}{b})
		-\sum_{\substack{
			\vecx_m=abcde \\
			a,b,c\neq\emptyset}}
		C(\ulflex{a}{bcd} e)B(\lrflex{a}{(\ulflex{b}{c}{d})})A(\lrflex{ab}{c}).
		\hphantom{
		-\sum_{\substack{
			\vecx_m=abcde \\
			a,b,c,d\neq\emptyset}}
		C(\ulflex{a}{b} \ulflex{c}{d}{e})B(\lrflex{a}{b})A(\lrflex{abc}{d})
		}
\end{align*}}
By rearranging each terms and by calculating the terms with
$c=\emptyset$, we have
{\footnotesize
\begin{align*}
	&=\sum_{\substack{
			\vecx_m=abde \\
			b,d,e\neq\emptyset}}
		C({a} \urflex{bd}{e})A(\llflex{b}{e})B(\llflex{d}{e})
		-\sum_{\substack{
			\vecx_m=abde \\
			b,d,e\neq\emptyset}}
		C({a} \urflex{bd}{e})B(\llflex{b}{e})A(\llflex{d}{e}) \\
	&\quad
	+\sum_{\substack{
			\vecx_m=abde \\
			a,b,d\neq\emptyset}}
		C(\ulflex{a}{bd} e)A(\lrflex{a}{d})B(\lrflex{a}{b})
		-\sum_{\substack{
			\vecx_m=abde \\
			a,b,d\neq\emptyset}}
		C(\ulflex{a}{bd} e)B(\lrflex{a}{d})A(\lrflex{a}{b}) \\
	&+\sum_{\substack{
			\vecx_m=abcde \\
			c,d,e\neq\emptyset}}
		C({a} \urflex{bcd}{e})A(\llflex{({b}\urflex{c}{d})}{e})B(\llflex{c}{de})
		-\sum_{\substack{
			\vecx_m=abcde \\
			b,c,e\neq\emptyset}}
		C({a} \urflex{bcd}{e})A(\llflex{(\ulflex{b}{c}{d})}{e})B(\lrflex{ab}{c}) \\
	&-\sum_{\substack{
			\vecx_m=abcde \\
			c,d,e\neq\emptyset}}
		C({a} \urflex{bcd}{e})B(\llflex{({b}\urflex{c}{d})}{e})A(\llflex{c}{de})
		+\sum_{\substack{
			\vecx_m=abcde \\
			b,c,e\neq\emptyset}}
		C({a} \urflex{bcd}{e})B(\llflex{(\ulflex{b}{c}{d})}{e})A(\lrflex{ab}{c}) \\
	&+\sum_{\substack{
			\vecx_m=abcde \\
			a,b,c\neq\emptyset}}
		C(\ulflex{a}{bcd} e)A(\lrflex{a}{(\ulflex{b}{c}{d})})B(\lrflex{ab}{c})
		-\sum_{\substack{
			\vecx_m=abcde \\
			a,c,d\neq\emptyset}}
		C(\ulflex{a}{bcd} e)A(\lrflex{a}{({b}\urflex{c}{d})})B(\llflex{c}{de}) \\
	&-\sum_{\substack{
			\vecx_m=abcde \\
			a,b,c\neq\emptyset}}
		C(\ulflex{a}{bcd} e)B(\lrflex{a}{(\ulflex{b}{c}{d})})A(\lrflex{ab}{c})
		+\sum_{\substack{
			\vecx_m=abcde \\
			a,c,d\neq\emptyset}}
		C(\ulflex{a}{bcd} e)B(\lrflex{a}{({b}\urflex{c}{d})})A(\llflex{c}{de}).
\end{align*}}
Here, the first term is calculated such that
\begin{equation*}
	\sum_{\substack{
			\vecx_m=abde \\
			b,d,e\neq\emptyset}}
		C({a} \urflex{bd}{e})A(\llflex{b}{e})B(\llflex{d}{e})
	=\sum_{\substack{
			\vecx_m=abce \\
			b,c,e\neq\emptyset}}
		C({a} \urflex{bc}{e})A(\llflex{b}{e})B(\llflex{c}{e})
	=\sum_{\substack{
			\vecx_m=abcde \\
			b,c,e\neq\emptyset \\
			d=\emptyset}}
		C({a} \urflex{bcd}{e})A(\llflex{(b\urflex{c}{d})}{e})B(\llflex{c}{de}),
\end{equation*}
The second, third and fourth terms are also calculated respectively as
\begin{align*}
	&-\sum_{\substack{
			\vecx_m=abde \\
			b,d,e\neq\emptyset}}
		C({a} \urflex{bd}{e})B(\llflex{b}{e})A(\llflex{d}{e})
	=-\sum_{\substack{
			\vecx_m=abcde \\
			b,c,e\neq\emptyset \\
			d=\emptyset}}
		C({a} \urflex{bcd}{e})B(\llflex{(b\urflex{c}{d})}{e})A(\llflex{c}{de}), \\
	&\quad\sum_{\substack{
			\vecx_m=abde \\
			a,b,d\neq\emptyset}}
		C(\ulflex{a}{bd} {e})A(\lrflex{a}{d})B(\lrflex{a}{b})
	=\sum_{\substack{
			\vecx_m=abcde \\
			a,c,d\neq\emptyset \\
			b=\emptyset}}
		C(\ulflex{a}{bcd} {e})A(\lrflex{a}{(\ulflex{b}{c} {d})})B(\lrflex{ab}{c}), \\
	&-\sum_{\substack{
			\vecx_m=abde \\
			a,b,d\neq\emptyset}}
		C(\ulflex{a}{bd} {e})B(\lrflex{a}{d})A(\lrflex{a}{b})
	=-\sum_{\substack{
			\vecx_m=abcde \\
			a,c,d\neq\emptyset \\
			b=\emptyset}}
		C(\ulflex{a}{bcd} e)B(\lrflex{a}{(\ulflex{b}{c}{d})})A(\lrflex{ab}{c}).
\end{align*}
These computations yield
\begin{align*}
	&(\aritu(B)\circ\aritu(A)-\aritu(A)\circ\aritu(B))(C)(\vecx_m) \\
	&=\sum_{\substack{
			\vecx_m=abcde \\
			c,e\neq\emptyset}}
		C({a} \urflex{bcd}{e})A(\llflex{({b}\urflex{c}{d})}{e})B(\llflex{c}{de})
		-\sum_{\substack{
			\vecx_m=abcde \\
			b,c,e\neq\emptyset}}
		C({a} \urflex{bcd}{e})A(\llflex{(\ulflex{b}{c}{d})}{e})B(\lrflex{ab}{c}) \\
	&\quad
	-\sum_{\substack{
			\vecx_m=abcde \\
			c,e\neq\emptyset}}
		C({a} \urflex{bcd}{e})B(\llflex{({b}\urflex{c}{d})}{e})A(\llflex{c}{de})
		+\sum_{\substack{
			\vecx_m=abcde \\
			b,c,e\neq\emptyset}}
		C({a} \urflex{bcd}{e})B(\llflex{(\ulflex{b}{c}{d})}{e})A(\lrflex{ab}{c}) \\
	&\quad
	+\sum_{\substack{
			\vecx_m=abcde \\
			a,c\neq\emptyset}}
		C(\ulflex{a}{bcd} e)A(\lrflex{a}{(\ulflex{b}{c}{d})})B(\lrflex{ab}{c})
		-\sum_{\substack{
			\vecx_m=abcde \\
			a,c,d\neq\emptyset}}
		C(\ulflex{a}{bcd} e)A(\lrflex{a}{({b}\urflex{c}{d})})B(\llflex{c}{de}) \\
	&\quad
	-\sum_{\substack{
			\vecx_m=abcde \\
			a,c\neq\emptyset}}
		C(\ulflex{a}{bcd} e)B(\lrflex{a}{(\ulflex{b}{c}{d})})A(\lrflex{ab}{c})
		+\sum_{\substack{
			\vecx_m=abcde \\
			a,c,d\neq\emptyset}}
		C(\ulflex{a}{bcd} e)B(\lrflex{a}{({b}\urflex{c}{d})})A(\llflex{c}{de})
\end{align*}
\begin{align*}
	&=\sum_{\substack{
		\vecx_m=afe \\
		e\neq\emptyset}}
	C({a} \urflex{f}{e})
	\left\{
	\sum_{\substack{
			f=bcd \\
			c\neq\emptyset}}
		A(\llflex{({b}\urflex{c}{d})}{e})B(\llflex{c}{de})
		-\sum_{\substack{
			f=bcd \\
			b,c\neq\emptyset}}
		A(\llflex{(\ulflex{b}{c}{d})}{e})B(\lrflex{ab}{c})
	\right\} \\
	&\quad
	-\sum_{\substack{
		\vecx_m=afe \\
		e\neq\emptyset}}
	C({a} \urflex{f}{e})
	\left\{
	\sum_{\substack{
			f=bcd \\
			c\neq\emptyset}}
		B(\llflex{({b}\urflex{c}{d})}{e})A(\llflex{c}{de})
		-\sum_{\substack{
			f=bcd \\
			b,c\neq\emptyset}}
		B(\llflex{(\ulflex{b}{c}{d})}{e})A(\lrflex{ab}{c})
	\right\} \\
	&\quad
	+\sum_{\substack{
		\vecx_m=afe \\
		a\neq\emptyset}}
	C(\ulflex{a}{f} e)
	\left\{
	\sum_{\substack{
			f=bcd \\
			c\neq\emptyset}}
		A(\lrflex{a}{(\ulflex{b}{c}{d})})B(\lrflex{ab}{c})
		-\sum_{\substack{
			f=bcd \\
			c,d\neq\emptyset}}
		A(\lrflex{a}{({b}\urflex{c}{d})})B(\llflex{c}{de})
	\right\} \\
	&\quad
	-\sum_{\substack{
		\vecx_m=afe \\
		a\neq\emptyset}}
	C(\ulflex{a}{f} e)
	\left\{
	\sum_{\substack{
			f=bcd \\
			c\neq\emptyset}}
		B(\lrflex{a}{(\ulflex{b}{c}{d})})A(\lrflex{ab}{c})
		-\sum_{\substack{
			f=bcd \\
			c,d\neq\emptyset}}
		B(\lrflex{a}{({b}\urflex{c}{d})})A(\llflex{c}{de})
	\right\}
\end{align*}
\begin{align*}
	&=\sum_{\substack{
		\vecx_m=afe \\
		e\neq\emptyset}}
	C({a} \urflex{f}{e})
	\left\{
	\sum_{\substack{
			f=bcd \\
			c,d\neq\emptyset}}
		A(\llflex{({b}\urflex{c}{d})}{e})B(\llflex{c}{d})
		-\sum_{\substack{
			f=bcd \\
			b,c\neq\emptyset}}
		A(\llflex{(\ulflex{b}{c}{d})}{e})B(\lrflex{b}{c})
		+\sum_{\substack{
			f=bc \\
			c\neq\emptyset}}
		A(\llflex{b}{e})B(\llflex{c}{e})
	\right\} \\
	&\quad
	-\sum_{\substack{
		\vecx_m=afe \\
		e\neq\emptyset}}
	C({a} \urflex{f}{e})
	\left\{
	\sum_{\substack{
			f=bcd \\
			c,d\neq\emptyset}}
		B(\llflex{({b}\urflex{c}{d})}{e})A(\llflex{c}{d})
		-\sum_{\substack{
			f=bcd \\
			b,c\neq\emptyset}}
		B(\llflex{(\ulflex{b}{c}{d})}{e})A(\lrflex{b}{c})
		+\sum_{\substack{
			f=bc \\
			c\neq\emptyset}}
		B(\llflex{b}{e})A(\llflex{c}{e})
	\right\} \\
	&\quad
	+\sum_{\substack{
		\vecx_m=afe \\
		a\neq\emptyset}}
	C(\ulflex{a}{f} e)
	\left\{
	\sum_{\substack{
			f=bcd \\
			b,c\neq\emptyset}}
		A(\lrflex{a}{(\ulflex{b}{c}{d})})B(\lrflex{b}{c})
		-\sum_{\substack{
			f=bcd \\
			c,d\neq\emptyset}}
		A(\lrflex{a}{({b}\urflex{c}{d})})B(\llflex{c}{d})
		+\sum_{\substack{
			f=cd \\
			c\neq\emptyset}}
		A(\lrflex{a}{d})B(\lrflex{a}{c})
	\right\} \\
	&\quad
	-\sum_{\substack{
		\vecx_m=afe \\
		a\neq\emptyset}}
	C(\ulflex{a}{f} e)
	\left\{
	\sum_{\substack{
			f=bcd \\
			b,c\neq\emptyset}}
		B(\lrflex{a}{(\ulflex{b}{c}{d})})A(\lrflex{b}{c})
		-\sum_{\substack{
			f=bcd \\
			c,d\neq\emptyset}}
		B(\lrflex{a}{({b}\urflex{c}{d})})A(\llflex{c}{d})
		+\sum_{\substack{
			f=cd \\
			c\neq\emptyset}}
		B(\lrflex{a}{d})A(\lrflex{a}{c})
	\right\}.
	\intertext{By using Lemma \ref{lem:flexions property}.(2), (3) and Lemma \ref{lem:flexion's action}, we have}
	&=\sum_{\substack{
		\vecx_m=afe \\
		e\neq\emptyset}}
	C({a} \urflex{f}{e})
	\left\{
	(\aritu(B)(A))(\llflex{f}{e})+(A\times B)(\llflex{f}{e})
	-(\aritu(A)(B))(\llflex{f}{e})-(B\times A)(\llflex{f}{e})
	\right\} \\
	&\quad
	+\sum_{\substack{
		\vecx_m=afe \\
		a\neq\emptyset}}
	C(\ulflex{a}{f} e)
	\left\{
	-(\aritu(B)(A))(\lrflex{a}{f})+(B\times A)(\lrflex{a}{f})
	+(\aritu(A)(B))(\lrflex{a}{f})-(A\times B)(\lrflex{a}{f})
	\right\} \\
	&=\sum_{\substack{
		\vecx_m=afe \\
		e\neq\emptyset}}
	C({a} \urflex{f}{e}) \ariu(A,B)(\llflex{f}{e})
	-\sum_{\substack{
		\vecx_m=afe \\
		a\neq\emptyset}}
	C(\ulflex{a}{f} e) \ariu(A,B)(\lrflex{a}{f}) \\
	&=(\aritu(\ariu(A,B))(C))(\vecx_m).
\end{align*}
Therefore, we obtain the claim.
\end{proof}

\begin{defn}[\cite{E-flex} (2.46)]
We consider a binary operation
$\preariu:\ARI(\Gamma)^{\otimes2}\rightarrow\ARI(\Gamma)$ 
which is defined by
\begin{equation*}
	\preariu(A,B):=\aritu(B)(A)+A\times B
\end{equation*}
for $A,B\in\ARI(\Gamma)$.
\end{defn}

Then we have $\ariu(A,B)=\preariu(A,B)-\preariu(B,A)$.
\begin{prop}\label{prop:preari}
The pair $(\ARI(\Gamma),\preariu)$ forms a pre-Lie algebra, i.e, the following formula holds:
\begin{align*}
	&\preariu(A,\preariu(B,C))-\preariu(\preariu(A,B),C) \\
	&=\preariu(A,\preariu(C,B))-\preariu(\preariu(A,C),B)
\end{align*}
for $A,B,C\in\ARI(\Gamma)$.
\end{prop}
\begin{proof}
We have
\begin{align*}
	&\{\preariu(A,\preariu(B,C))-\preariu(\preariu(A,B),C)\} \\
	&-\{\preariu(A,\preariu(C,B))-\preariu(\preariu(A,C),B)\} \\
	&=\aritu(\preariu(B,C))(A)+A\times\preariu(B,C)
	-\aritu(C)(\preariu(A,B))-\preariu(A,B)\times C \\
	&\quad-\aritu(\preariu(C,B))(A)-A\times\preariu(C,B)
	+\aritu(B)(\preariu(A,C))+\preariu(A,C)\times B \\
	&=\aritu(\aritu(C)(B))(A)+\aritu(B\times C)(A)+A\times(\aritu(C)(B))+A\times(B\times C) \\
	&\quad-\aritu(C)(\aritu(B)(A))-\aritu(C)(A\times B)-\aritu(B)(A)\times C-(A\times B)\times C \\
	&\quad-\aritu(\aritu(B)(C))(A)-\aritu(C\times B)(A)-A\times(\aritu(B)(C))-A\times(C\times B) \\
	&\quad\quad+\aritu(B)(\aritu(C)(A))+\aritu(B)(A\times C)+\aritu(C)(A)\times B+(A\times C)\times B. \\
	\intertext{By using associativity of $(\ARI(\Gamma),\times)$ and using Lemma \ref{lem:arit derivation}, we get}
	&=\aritu(\aritu(C)(B))(A)+\aritu(B\times C)(A) - \aritu(C)(\aritu(B)(A)) \\
	&\quad-\aritu(\aritu(B)(C))(A)-\aritu(C\times B)(A) + \aritu(B)(\aritu(C)(A)) \\
	&=\aritu(\ariu(B,C))(A) -\{ \aritu(C)\circ\aritu(B)(A)- \aritu(B)\circ\aritu(C)(A)\}.
\end{align*}
Therefore, by using Proposition \ref{prop:arit-ari}, we obtain the claim.
\end{proof}

\noindent
{\it Proof of Proposition \ref{ARI Lie algebra}.} 
It is sufficient to prove the Jacobi identity
\begin{equation*}
	\ariu(\ariu(A,B),C) + \ariu(\ariu(B,C),A) + \ariu(\ariu(C,A),B) = 0
\end{equation*}
for $A,B,C\in\ARI(\Gamma)$. By the relationship between $\ariu$ and $\preariu$, we calculate
\begin{align*}
	&\ariu(\ariu(A,B),C) + \ariu(\ariu(B,C),A) + \ariu(\ariu(C,A),B) \\
	&=\preariu(\ariu(A,B),C) - \preariu(C,\ariu(A,B)) \\
	&\quad+ \preariu(\ariu(B,C),A) - \preariu(A,\ariu(B,C)) \\
	&\quad\quad+ \preariu(\ariu(C,A),B) - \preariu(B,\ariu(C,A)) \\
	&=\preariu(\preariu(A,B)-\preariu(B,A),C)
	- \preariu(C,\preariu(A,B)-\preariu(B,A)) \\
	&\quad+ \preariu(\preariu(B,C)-\preariu(C,B),A)
	- \preariu(A,\preariu(B,C)-\preariu(C,B)) \\
	&\quad\quad+ \preariu(\preariu(C,A)-\preariu(A,C),B)
	- \preariu(B,\preariu(C,A)-\preariu(A,C)).
\end{align*}
By Proposition \ref{prop:preari}, it is equal to 0. So we obtain the Jacobi identity.
\hfill $\Box$

\subsection{Proof of Proposition \ref{ARIal Lie algebra}}\label{sec:A.2}
We give a proof of Proposition \ref{ARIal Lie algebra} which claims that  $\ARI(\Gamma)_\al$ forms a Lie algebra.

We show the following key lemma in this section.
\begin{lem}\label{lem:shuffle coefficient}
	For $\omega,\eta,\alpha_1,\dots,\alpha_r\in X_{\Z}^\bullet$, we have
	\begin{equation}
		\Sh{\omega}{\eta}{\alpha_1\cdots\alpha_r}
		=\sum_{\substack{
		\omega=\omega_1\cdots\omega_r \\	
		\eta=\eta_1\cdots\eta_r	}}
		\Sh{\omega_1}{\eta_1}{\alpha_1}\cdots\Sh{\omega_r}{\eta_r}{\alpha_r}.
	\end{equation}
	where $\omega_1,\dots,\omega_r,\eta_1,\dots,\eta_r$ run over $X_{\Z}^\bullet$.
\end{lem}
\begin{proof}
	We consider the deconcatenation coproduct $\Delta:\mathcal A_X\rightarrow\mathcal A_X^{\otimes2}$ defined by
	\begin{equation*}
		\Delta(\omega):=\sum_{\substack{
			\omega=\omega_1\omega_2 \\
			\omega_1,\omega_2\in X_{\Z}^\bullet}}
		\omega_1\otimes\omega_2
	\end{equation*}
	for $\omega\in X_{\Z}^\bullet$.
	We recursively define $\mathbb Q$-linear maps $\Delta_r:\mathcal{A}_X\rightarrow\mathcal{A}_X^{\otimes r}$ by $\Delta_2:=\Delta$ and for $r\geqslant 3$
	\begin{equation*}
		\Delta_r:=(\underbrace{{\rm Id}\otimes \cdots\otimes {\rm Id}}_{r-2}\otimes \Delta)\circ \Delta_{r-1}.
	\end{equation*}
	It is clear that $\Delta_r$ is an algebra homomorphism
	\footnote{Note that the product $\shuffle$ of $\mathcal A_X$ induces the product of $\mathcal A_X^{\otimes r}$ (we also denote this product to the same symbol $\shuffle$) as
	$$
	(\omega_1\otimes\cdots\otimes\omega_r)\shuffle(\eta_1\otimes\cdots\otimes\eta_r)
	:=(\omega_1\ \shuffle\ \eta_1)\otimes\cdots\otimes(\omega_r\ \shuffle\ \eta_r)
	$$
	for any $\omega_i,\eta_i\in X_{\Z^\bullet}$.}
	, i.e, for $\omega,\eta\in X_{\Z}^\bullet$, we have
	\begin{equation*}
		\Delta_r(\omega\ \shuffle\ \eta)=\Delta_r(\omega)\shuffle\Delta_r(\eta).
	\end{equation*}
	By expanding the above both sides and taking the coefficient of $\alpha_1\otimes \cdots\otimes \alpha_r$ for $\alpha_1,\dots,\alpha_r\in X_{\Z}^\bullet$, we obtain the claim.
\end{proof}
\noindent
{\it Proof of Proposition \ref{ARIal Lie algebra}.} Because we have
\begin{equation*}
	\ari(A,B)=\arit(B)(A)-\arit(A)(B)+[A,B],
\end{equation*}
it is sufficient to prove the following two formulae for $A,B\in\ARI(\Gamma)_{\al}$:
\begin{align}
	\label{mu-al} [A,B]&\in\ARI(\Gamma)_{\al}, \\
	\label{arit-al} \arit(B)(A)&\in\ARI(\Gamma)_{\al}.
\end{align}

Firstly, we prove \eqref{mu-al}. Let $p,q\geqslant1$ and put $\omega=\varia{x_1,\ \dots,\ x_p}{\sigma_1,\ \dots,\ \sigma_p}$ and $\eta=\varia{x_{p+1},\ \dots,\ x_{p+q}}{\sigma_{p+1},\ \dots,\ \sigma_{p+q}}$. For our simplicity, we denote
\begin{equation}\label{eqn:def of shuffle map}
\mathpzc{Sh}(M)(\omega;\eta)
:=\sum_{\alpha\in X_{\Z}^\bullet}\Sh{\omega}{\eta}{\alpha}M^{p+q}(\alpha).
\end{equation}
 Then we have
\begin{align*}
	&\mathpzc{Sh}(A\times B)(\omega;\eta) 
	=\sum_{\alpha\in X_{\Z}^\bullet}\Sh{\omega}{\eta}{\alpha}\sum_{\alpha=\alpha_1\alpha_2}
		A(\alpha_1)B(\alpha_2) 
	=\sum_{\alpha_1,\alpha_2\in X_{\Z}^\bullet}\Sh{\omega}{\eta}{\alpha_1\alpha_2}
		A(\alpha_1)B(\alpha_2).
	\intertext{By using Lemma \ref{lem:shuffle coefficient} for $r=2$, we get}
	=&\sum_{\alpha_1,\alpha_2\in X_{\Z}^\bullet}\sum_{\substack{
		\omega=\omega_1\omega_2 \\
		\eta=\eta_1\eta_2}}
	\Sh{\omega_1}{\eta_1}{\alpha_1}\Sh{\omega_2}{\eta_2}{\alpha_2}
		A(\alpha_1)B(\alpha_2) \\
	=&\sum_{\substack{
		\omega=\omega_1\omega_2 \\
		\eta=\eta_1\eta_2}}
	\left\{\sum_{\alpha_1\in X_{\Z}^\bullet}\Sh{\omega_1}{\eta_1}{\alpha_1}
		A(\alpha_1)\right\}
	\left\{\sum_{\alpha_2\in X_{\Z}^\bullet}\Sh{\omega_2}{\eta_2}{\alpha_2}
		B(\alpha_2)\right\}.
	\intertext{By using alternality of $A$, we calculate}
	=&\sum_{\substack{
		\omega=\omega_1\omega_2 \\
		\omega_1\neq\emptyset}}
	A(\omega_1)
	\left\{\sum_{\alpha_2\in X_{\Z}^\bullet}\Sh{\omega_2}{\eta}{\alpha_2}
		B(\alpha_2)\right\}
	+\sum_{\substack{
		\eta=\eta_1\eta_2 \\
		\eta_1\neq\emptyset}}
	A(\eta_1)
	\left\{\sum_{\alpha_2\in X_{\Z}^\bullet}\Sh{\omega}{\eta_2}{\alpha_2}
		B(\alpha_2)\right\}.
	\intertext{By using alternality of $B$, we obtain}
	=&A(\omega)B(\eta)+A(\eta)B(\omega).
\end{align*}
Because we have $[A,B]=A\times B-B\times A$, we get
$\mathpzc{Sh}([A,B])(\omega;\eta)=0,$
that is, we obtain \eqref{mu-al}.

Secondly, we prove \eqref{arit-al}. 
We remark that, for $\alpha\in X_{\Z}^\bullet$ with $l(\alpha)\geqslant2$, we have
\begin{align}\label{eqn: remark on arit(B)(A)}
	& \arit(A)(B)(\alpha) \\
	& =\sum_{\substack
		{\alpha=\alpha_1\alpha_2\alpha_3 \\
		\alpha_2,\alpha_3\neq\emptyset}}
	A(\alpha_1\urflex{\alpha_2}{\alpha_3})B(\llflex{\alpha_2}{\alpha_3}) 
	-\sum_{\substack
		{\alpha=\alpha_1\alpha_2\alpha_3 \\
		\alpha_1,\alpha_2\neq\emptyset}}
	A(\ulflex{\alpha_1}{\alpha_2}\alpha_3)B(\lrflex{\alpha_1}{\alpha_2}) \nonumber\\
	& =\sum_{\substack
		{\alpha=\alpha_1\alpha_2\alpha_3 \\
		\alpha_3\neq\emptyset}}
	A(\alpha_1\urflex{\alpha_2}{\alpha_3})B(\llflex{\alpha_2}{\alpha_3}) 
	-\sum_{\substack
		{\alpha=\alpha_1\alpha_2\alpha_3 \\
		\alpha_1\neq\emptyset}}
	A(\ulflex{\alpha_1}{\alpha_2}\alpha_3)B(\lrflex{\alpha_1}{\alpha_2}) \nonumber\\
	& =\sum_{\substack
		{\alpha=\alpha_1\alpha_2x\alpha_3 \\
		\alpha_i\in X_{\Z}^\bullet,x\in X_{\Z}}}
	A(\alpha_1\urflex{\alpha_2}{x}\alpha_3)B(\llflex{\alpha_2}{x}) 
	-\sum_{\substack
		{\alpha=\alpha_1x\alpha_2\alpha_3 \\
		\alpha_i\in X_{\Z}^\bullet,x\in X_{\Z}}}
	A(\alpha_1\ulflex{x}{\alpha_2}\alpha_3)B(\lrflex{x}{\alpha_2}). \nonumber
\end{align}
Then by the definition \eqref{eqn:def of shuffle map} of the map $\mathpzc{Sh}$,
we  calculate
\begin{align*}
	& \mathpzc{Sh}(\arit(B)(A))(\omega;\eta) 
	=\sum_{\alpha\in X_{\Z}^\bullet}
	\Sh{\omega}{\eta}{\alpha}
	\arit(A)(B)(\alpha).
\intertext{Then by \eqref{eqn: remark on arit(B)(A)}, we have}
	& =\sum_{\alpha\in X_{\Z}^\bullet}
	\Sh{\omega}{\eta}{\alpha}
	\left\{ \sum_{\substack
		{\alpha=\alpha_1\alpha_2x\alpha_3 \\
		\alpha_i\in X_{\Z}^\bullet,x\in X_{\Z}}}
	 A(\alpha_1\urflex{\alpha_2}{x}\alpha_3) B(\llflex{\alpha_2}{x}) 
	-\sum_{\substack
		{\alpha=\alpha_1x\alpha_2\alpha_3 \\
		\alpha_i\in X_{\Z}^\bullet,x\in X_{\Z}}}
	 A(\alpha_1\ulflex{x}{\alpha_2}\alpha_3) B(\lrflex{x}{\alpha_2}) \right\} \\
	& =\sum_{\substack
		{\alpha_1,\alpha_2,\alpha_3\in X_{\Z}^\bullet \\
		x\in X_{\Z}}}
	\left\{ \Sh{\omega\ }{\ \eta}{\alpha_1\alpha_2x\alpha_3}
	 A(\alpha_1\urflex{\alpha_2}{x}\alpha_3) B(\llflex{\alpha_2}{x}) 
	-\Sh{\omega\ }{\ \eta}{\alpha_1x\alpha_2\alpha_3}
	 A(\alpha_1\ulflex{x}{\alpha_2}\alpha_3) B(\lrflex{x}{\alpha_2}) \right\}.
\end{align*}
By using Lemma \ref{lem:shuffle coefficient}, we have
\begin{align*}
	& =\sum_{\substack
		{\alpha_1,\alpha_2,\alpha_3\in X_{\Z}^\bullet \\
		x\in X_{\Z}}}
	\left\{ \sum_{\substack{
		\omega=\omega_1\omega_2\omega_2'\omega_3 \\
		\eta=\eta_1\eta_2\eta_2'\eta_3}}
	\Sh{\omega_1}{\eta_1}{\alpha_1}\Sh{\omega_2}{\eta_2}{\alpha_2}\Sh{\omega_2'}{\eta_2'}{x}\Sh{\omega_3}{\eta_3}{\alpha_3}
	 A(\alpha_1\urflex{\alpha_2}{x}\alpha_3) B(\llflex{\alpha_2}{x}) \right. \\
	& \hspace{0.2cm}
	\left. -\sum_{\substack{
		\omega=\omega_1\omega_2'\omega_2\omega_3 \\
		\eta=\eta_1\eta_2'\eta_2\eta_3}}
	\Sh{\omega_1}{\eta_1}{\alpha_1}\Sh{\omega_2'}{\eta_2'}{x}\Sh{\omega_2}{\eta_2}{\alpha_2}\Sh{\omega_3}{\eta_3}{\alpha_3}
	 A(\alpha_1\ulflex{x}{\alpha_2}\alpha_3) B(\lrflex{x}{\alpha_2}) \right\}.
\end{align*}
Here, if $\Sh{\omega_2}{\eta_2}{\alpha_2}\neq0$ holds for $\alpha_2\in X_{\Z}^\bullet$, then all letters appearing in $\alpha_2$ match with all ones appearing in $\omega_2$ and $\eta_2$. So we have $\urflex{\alpha_2}{x}=\urflex{(\omega_2,\eta_2)}{x}$ and $\ulflex{x}{\alpha_2}=\ulflex{x}{(\omega_2,\eta_2)}$ for $\alpha_2\in X_{\Z}^\bullet$ and we continue
\begin{align*}
	& =\sum_{\substack
		{\alpha_1,\alpha_3\in X_{\Z}^\bullet \\
		x\in X_{\Z}}}
	\left\{ \sum_{\substack{
		\omega=\omega_1\omega_2\omega_2'\omega_3 \\
		\eta=\eta_1\eta_2\eta_2'\eta_3}}
	\Sh{\omega_1}{\eta_1}{\alpha_1}\Sh{\omega_2'}{\eta_2'}{x}\Sh{\omega_3}{\eta_3}{\alpha_3}
	 A(\alpha_1\urflex{(\omega_2,\eta_2)}{x}\alpha_3)\mathpzc{Sh}( B)(\llflex{\omega_2}{x};\llflex{\eta_2}{x}) \right. \\
	& \hspace{0.2cm}
	\left. -\sum_{\substack{
		\omega=\omega_1\omega_2'\omega_2\omega_3 \\
		\eta=\eta_1\eta_2'\eta_2\eta_3}}
	\Sh{\omega_1}{\eta_1}{\alpha_1}\Sh{\omega_2'}{\eta_2'}{x}\Sh{\omega_3}{\eta_3}{\alpha_3}
	 A(\alpha_1\ulflex{x}{(\omega_2,\eta_2)}\alpha_3)\mathpzc{Sh}( B)(\lrflex{x}{\omega_2};\lrflex{x}{\eta_2}) \right\}.
\end{align*}
Because $ B\in\ARI(\Gamma)_{\al}$, we have $\mathpzc{Sh}( B)(\emptyset;\emptyset)=0$ and $\mathpzc{Sh}( B)(\llflex{\omega_2}{x};\llflex{\eta_2}{x})=\mathpzc{Sh}( B)(\lrflex{x}{\omega_2};\lrflex{x}{\eta_2})=0$ for $\omega_2,\eta_2\neq\emptyset$. So we have
\begin{align*}
	& =\sum_{\substack
		{\alpha_1,\alpha_3\in X_{\Z}^\bullet \\
		x\in X_{\Z}}}
	\left\{ \sum_{\substack{
		\omega=\omega_1\omega_2\omega_2'\omega_3 \\
		\eta=\eta_1\eta_2'\eta_3 \\
		\omega_2\neq\emptyset}}
	\Sh{\omega_1}{\eta_1}{\alpha_1}\Sh{\omega_2'}{\eta_2'}{x}\Sh{\omega_3}{\eta_3}{\alpha_3}
	 A(\alpha_1\urflex{\omega_2}{x}\alpha_3) B(\llflex{\omega_2}{x}) \right. \\
	& \hspace{2cm}
	-\sum_{\substack{
		\omega=\omega_1\omega_2'\omega_2\omega_3 \\
		\eta=\eta_1\eta_2'\eta_3 \\
		\omega_2\neq\emptyset}}
	\Sh{\omega_1}{\eta_1}{\alpha_1}\Sh{\omega_2'}{\eta_2'}{x}\Sh{\omega_3}{\eta_3}{\alpha_3}
	 A(\alpha_1\ulflex{x}{\omega_2}\alpha_3) B(\lrflex{x}{\omega_2}) \\
	& \hspace{2.5cm}+ \sum_{\substack{
		\omega=\omega_1\omega_2'\omega_3 \\
		\eta=\eta_1\eta_2\eta_2'\eta_3 \\
		\eta_2\neq\emptyset}}
	\Sh{\omega_1}{\eta_1}{\alpha_1}\Sh{\omega_2'}{\eta_2'}{x}\Sh{\omega_3}{\eta_3}{\alpha_3}
	 A(\alpha_1\urflex{\eta_2}{x}\alpha_3) B(\llflex{\eta_2}{x}) \\
	& \hspace{3cm}
	\left. -\sum_{\substack{
		\omega=\omega_1\omega_2'\omega_3 \\
		\eta=\eta_1\eta_2'\eta_2\eta_3 \\
		\eta_2\neq\emptyset}}
	\Sh{\omega_1}{\eta_1}{\alpha_1}\Sh{\omega_2'}{\eta_2'}{x}\Sh{\omega_3}{\eta_3}{\alpha_3}
	 A(\alpha_1\ulflex{x}{\eta_2}\alpha_3) B(\lrflex{x}{\eta_2}) \right\}.
\end{align*}
Here, $\Sh{\omega_2'}{\eta_2'}{x}\neq0$ holds for $x\in X_{\Z}$ if and only if $(\omega_2',\eta_2')=(x,\emptyset)$ or $(\emptyset,x)$. So we get
\begin{align*}
	& =\sum_{\substack
		{\alpha_1,\alpha_3\in X_{\Z}^\bullet \\
		x\in X_{\Z}}}
	\left[ \sum_{\substack{
		\omega=\omega_1\omega_2x\omega_3 \\
		\eta=\eta_1\eta_3 \\
		\omega_2\neq\emptyset}}
	\Sh{\omega_1}{\eta_1}{\alpha_1}\Sh{\omega_3}{\eta_3}{\alpha_3}
	 A(\alpha_1\urflex{\omega_2}{x}\alpha_3) B(\llflex{\omega_2}{x}) \right. \\
	& \hspace{0.2cm}
	-\sum_{\substack{
		\omega=\omega_1x\omega_2\omega_3 \\
		\eta=\eta_1\eta_3 \\
		\omega_2\neq\emptyset}}
	\Sh{\omega_1}{\eta_1}{\alpha_1}\Sh{\omega_3}{\eta_3}{\alpha_3}
	 A(\alpha_1\ulflex{x}{\omega_2}\alpha_3) B(\lrflex{x}{\omega_2}) \\
	& \left.\hspace{0.4cm}+ \sum_{\substack{
		\omega=\omega_1x\omega_3 \\
		\eta=\eta_1\eta_2\eta_3 \\
		\eta_2\neq\emptyset}}
	\Sh{\omega_1}{\eta_1}{\alpha_1}\Sh{\omega_3}{\eta_3}{\alpha_3}
	\left\{  A(\alpha_1\urflex{\eta_2}{x}\alpha_3) B(\llflex{\eta_2}{x})
	- A(\alpha_1\ulflex{x}{\eta_2}\alpha_3) B(\lrflex{x}{\eta_2}) \right\} \right] \\
	& +\sum_{\substack
		{\alpha_1,\alpha_3\in X_{\Z}^\bullet \\
		x\in X_{\Z}}}
	\left[ \sum_{\substack{
		\omega=\omega_1\omega_2\omega_3 \\
		\eta=\eta_1x\eta_3 \\
		\omega_2\neq\emptyset}}
	\Sh{\omega_1}{\eta_1}{\alpha_1}\Sh{\omega_3}{\eta_3}{\alpha_3}
	\left\{  A(\alpha_1\urflex{\omega_2}{x}\alpha_3) B(\llflex{\omega_2}{x})  
	- A(\alpha_1\ulflex{x}{\omega_2}\alpha_3) B(\lrflex{x}{\omega_2}) \right\} \right. \\
	& \hspace{0.2cm}+ \sum_{\substack{
		\omega=\omega_1\omega_3 \\
		\eta=\eta_1\eta_2x\eta_3 \\
		\eta_2\neq\emptyset}}
	\Sh{\omega_1}{\eta_1}{\alpha_1}\Sh{\omega_3}{\eta_3}{\alpha_3}
	 A(\alpha_1\urflex{\eta_2}{x}\alpha_3) B(\llflex{\eta_2}{x}) \\
	& \hspace{0.4cm}
	\left. -\sum_{\substack{
		\omega=\omega_1\omega_3 \\
		\eta=\eta_1x\eta_2\eta_3 \\
		\eta_2\neq\emptyset}}
	\Sh{\omega_1}{\eta_1}{\alpha_1}\Sh{\omega_3}{\eta_3}{\alpha_3}
	 A(\alpha_1\ulflex{x}{\eta_2}\alpha_3) B(\lrflex{x}{\eta_2}) \right].
\end{align*}
Because $x$ runs over $X_{\Z}$, we get $\urflex{\omega_2}{x}=\ulflex{x}{\omega_2}$ and $\urflex{\eta_2}{x}=\ulflex{x}{\eta_2}$ and $\llflex{\eta_2}{x}=\lrflex{x}{\eta_2}$ and $\llflex{\omega_2}{x}=\lrflex{x}{\omega_2}$. Hence we calculate
\begin{align*}
	& =\sum_{\substack
		{\alpha_1,\alpha_3\in X_{\Z}^\bullet \\
		x\in X_{\Z}}}
	\left\{ \sum_{\substack{
		\omega=\omega_1\omega_2x\omega_3 \\
		\eta=\eta_1\eta_3 \\
		\omega_2\neq\emptyset}}
	\Sh{\omega_1}{\eta_1}{\alpha_1}\Sh{\omega_3}{\eta_3}{\alpha_3}
	 A(\alpha_1\urflex{\omega_2}{x}\alpha_3) B(\llflex{\omega_2}{x}) \right. \\
	& \left.\hspace{3cm}
	-\sum_{\substack{
		\omega=\omega_1x\omega_2\omega_3 \\
		\eta=\eta_1\eta_3 \\
		\omega_2\neq\emptyset}}
	\Sh{\omega_1}{\eta_1}{\alpha_1}\Sh{\omega_3}{\eta_3}{\alpha_3}
	 A(\alpha_1\urflex{\omega_2}{x}\alpha_3) B(\llflex{\omega_2}{x}) \right\} \\
	& +\sum_{\substack
		{\alpha_1,\alpha_3\in X_{\Z}^\bullet \\
		x\in X_{\Z}}}
	\left\{ \sum_{\substack{
		\omega=\omega_1\omega_3 \\
		\eta=\eta_1\eta_2x\eta_3 \\
		\eta_2\neq\emptyset}}
	\Sh{\omega_1}{\eta_1}{\alpha_1}\Sh{\omega_3}{\eta_3}{\alpha_3}
	 A(\alpha_1\urflex{\eta_2}{x}\alpha_3) B(\llflex{\eta_2}{x}) \right. \\
	& \hspace{3cm}
	\left. -\sum_{\substack{
		\omega=\omega_1\omega_3 \\
		\eta=\eta_1x\eta_2\eta_3 \\
		\eta_2\neq\emptyset}}
	\Sh{\omega_1}{\eta_1}{\alpha_1}\Sh{\omega_3}{\eta_3}{\alpha_3}
	 A(\alpha_1\urflex{\eta_2}{x}\alpha_3) B(\llflex{\eta_2}{x}) \right\}.
\end{align*}
For $\omega_1,\omega_2,\omega_3,x$ with $\omega=\omega_1\omega_2x\omega_3$, by using Lemma \ref{lem:shuffle coefficient} with $r=3$ and $\omega=\omega_1\urflex{\omega_2}{x}\omega_3$ and $\alpha_2=\urflex{\omega_2}{x}$, we have
\begin{align*}
	\Sh{\omega_1\urflex{\omega_2}{x}\omega_3}{\eta}{\alpha_1\urflex{\omega_2}{x}\alpha_3}
	&=\sum_{\substack{
		\omega_1\urflex{\omega_2}{x}\omega_3=\omega_1'\omega_2'\omega_3' \\
		\eta=\eta_1\eta_2\eta_3}}
	\Sh{\omega_1'}{\eta_1}{\alpha_1}\Sh{\omega_2'}{\eta_2}{\urflex{\omega_2}{x}}\Sh{\omega_3'}{\eta_3}{\alpha_3}.
	\intertext{Because $\eta=\varia{x_{p+1},\ \dots,\ x_{p+q}}{\sigma_{p+1},\ \dots,\ \sigma_{p+q}}$ and $\omega_2\neq\emptyset$, the letter $\urflex{\omega_2}{x}\in X_\mathbb Z$ does not appear in $\eta$. So for any word $\eta_2$ such that $\eta=\eta_1\eta_2\eta_3$, we get $\eta_2\neq\urflex{\omega_2}{x}$. Hence, $\Sh{\omega_2'}{\eta_2}{\urflex{\omega_2}{x}}\neq0$ holds if and only if $\omega_2'=\urflex{\omega_2}{x}$ and $\eta_2=\emptyset$. So we have}
	&=\sum_{\eta=\eta_1\eta_3}
	\Sh{\omega_1}{\eta_1}{\alpha_1}\Sh{\omega_3}{\eta_3}{\alpha_3}.
\end{align*}
Similarly, we get
\begin{align*}
	\sum_{\omega=\omega_1\omega_3}
	\Sh{\omega_1}{\eta_1}{\alpha_1}\Sh{\omega_3}{\eta_3}{\alpha_3}
	=\Sh{\omega}{\eta_1\urflex{\eta_2}{x}\eta_3}{\alpha_1\urflex{\omega_2}{x}\alpha_3}.
\end{align*}
So by using these, we calculate
\begin{align*}
	&\mathpzc{Sh}(\arit( B)( A))(\omega;\eta) \\
	&= \sum_{x\in X_{\Z}}
	\left\{ \sum_{\substack{
		\omega=\omega_1\omega_2x\omega_3 \\
		\omega_2\neq\emptyset}}
	\mathpzc{Sh}( A)(\omega_1\urflex{\omega_2}{x}\omega_3;\eta) B(\llflex{\omega_2}{x}) 
	-\sum_{\substack{
		\omega=\omega_1x\omega_2\omega_3 \\
		\omega_2\neq\emptyset}}
	\mathpzc{Sh}( A)(\omega_1\urflex{\omega_2}{x}\omega_3;\eta) B(\llflex{\omega_2}{x}) \right\} \\
	&\quad +\sum_{x\in X_{\Z}}
	\left\{ \sum_{\substack{
		\eta=\eta_1\eta_2x\eta_3 \\
		\eta_2\neq\emptyset}}
	\mathpzc{Sh}( A)(\omega;\eta_1\urflex{\eta_2}{x}\eta_3) B(\llflex{\eta_2}{x}) 
	-\sum_{\substack{
		\eta=\eta_1x\eta_2\eta_3 \\
		\eta_2\neq\emptyset}}
	\mathpzc{Sh}( A)(\omega;\eta_1\urflex{\eta_2}{x}\eta_3) B(\llflex{\eta_2}{x}) \right\}.
\end{align*}

Lastly, similarly to footnote \ref{footnote:arit embedding} in Definition \ref{def:aritu}, all letters appearing in two words $\omega_1\urflex{\omega_2}{x}\omega_3$ and $\eta$ (resp. $\omega$ and $\eta_1\urflex{\eta_2}{x}\eta_3$) are algebraically independent over $\mathbb Q$, so by Remark \ref{rem:component of mould by embedding}, the component $\mathpzc{Sh}( A)(\omega_1\urflex{\omega_2}{x}\omega_3;\eta)$ (resp. $\mathpzc{Sh}( A)(\omega;\eta_1\urflex{\eta_2}{x}\eta_3)$) is well-defined.
Hence, by using alternality of $A$, we obtain \eqref{arit-al}. 
\hfill $\Box$

\subsection{Proof of Proposition \ref{ARIalal Lie algebra}}\label{sec:A.3}
We give a proof Proposition \ref{ARIalal Lie algebra} which claims that 
$\ARI(\Gamma)_{\underline\al/\underline\al}$ 
 forms a filtered Lie subalgebra of $\ARI(\Gamma)_\al$ under the $\ari_u$-bracket.

For $A,B\in\ARI(\Gamma)_{\underline\al/\underline\al}$, it is enough to prove $\ari_u(A,B)\in\ARI(\Gamma)_{\underline\al/\underline\al}$, that is,
\begin{align}
	\label{ari in ARIal}& \ari_u(A,B)\in\ARI(\Gamma)_{\al}, \\
	\label{swapari in barARIal}& \swap(\ari_u(A,B))\in\overline{\ARI}(\Gamma)_\al, \\
	\label{length =1 is even}& \ari_u(A,B)^1\varia{x_1}{\sigma_1}=\ari_u(A,B)^1\varia{-x_1}{\sigma_1^{-1}}.
\end{align}
Because $\ARI(\Gamma)_{\al}$ forms a Lie algebra under the $\ari_u$-bracket, \eqref{ari in ARIal} is obvious. By the definition of $\ari_u$, \eqref{length =1 is even} is also clear. By Lemma \ref{swap-ari commutation} and Proposition \ref{prop:DKV1}, we have
\begin{equation*}
	\swap(\ari_u(A,B))=\ari_v(\swap(A),\swap(B)).
\end{equation*}
Since we have $\swap(A),\swap(B)\in\overline{\ARI}(\Gamma)_{\al}$ for $A$, $B\in\ARI(\Gamma)_{\underline\al/\underline\al}$,
we get $\swap(\ari_u(A,B))\in\overline{\ARI}(\Gamma)_{\al}$ 
by Proposition \ref{barARIal Lie algebra}. 
Thus we obtain \eqref{swapari in barARIal}. \hfill $\Box$

\section{Multiple polylogarithms at roots of unity}\label{sec:MPL at roots}
In this appendix we recall how 
$\ARI(\Gamma)_{\underline\al/\underline\al}$ and ${\mathbb D}(\Gamma)_{\bullet\bullet}$
are related to multiple polylogarithms at roots of unity.

Multiple polylogarithm is the several variable complex function
which is defined by the following power series
$$
\Li_{n_1,\dots,n_r}(z_1,\dots,z_r):=\sum_{0<k_1<\cdots <k_r}\frac{z_1^{k_1}\cdots z_r^{k_r}}{k_1^{n_1}\cdots k_r^{n_r}}
$$
for $n_1,\dots,n_r, r\in\N$.
For $N\in\N$, we denote $\mu_N$  to be the group of $N$-th roots of unity in  $\C$.
Its limit value at $r$-tuple $(\epsilon_1, \dots, \epsilon_r)\in\mu_N^{\oplus r}$ makes sense if and only if $(n_r,\epsilon_r)\neq (1,1)$.
They are called {\it multiple $L$-values} (particularly  multiple zeta values when $N=1$).

In \cite[\S 8]{E-ARIGARI} and \cite[\S 1.2]{E-flex}
(it is also recalled in \cite[\S A.3]{R-PhD} in the case when $\Gamma=\{e\}$),
the following moulds in $\mathcal M(\mathcal F;\Gamma)$
with $\mathcal F=\cup_m {\mathbb C}[[u_1,\dots, u_m]]$ and $\Gamma=\mu_N$
$$\Zag=\{\Zag\varia{u_1,\ \dots,\ u_m}{\epsilon_1,\ \dots,\ \epsilon_m}\}_m
\quad \text{ and }\quad
\Zig=\{\Zig\varia{\epsilon_1,\ \dots,\ \epsilon_m}{v_1,\ \dots,\ v_m}\}_m
$$
are introduced and defined by
\begin{align*}
&\Zag\varia{u_1,\ \dots,\ u_m}{\epsilon_1,\ \dots,\ \epsilon_m}=
\sum_{n_1,\dots, n_m>0}
\Li^{\scalebox{.5}{$\Sha$}}_{n_1,\dots,n_m}(\frac{\epsilon_1}{\epsilon_2},\dots,\frac{\epsilon_{m-1}}{\epsilon_m},\epsilon_m)
u_1^{n_1-1}(u_1+u_2)^{n_2-1}\cdots
(u_1+\cdots+u_m)^{n_m-1} \\
&\Zig\varia{\epsilon_1,\ \dots,\ \epsilon_m}{v_1,\ \dots,\ v_m}=
\sum_{n_1,\dots, n_m>0}
\Li^\ast_{n_m,\dots,n_1}(\epsilon_m, \dots, \epsilon_1)
v_1^{n_1-1}\cdots v_m^{n_m-1}
\end{align*}
where
 $\Li^{\scalebox{.5}{$\Sha$}}_{n_1,\dots,n_m}(\epsilon_1,\dots,\epsilon_m)$
 and
  $\Li^{\ast}_{n_1,\dots,n_m}(\epsilon_1,\dots,\epsilon_m)$
mean the shuffle regularization and the harmonic (stuffle) regularization of
$\Li_{n_1,\dots,n_m}(\epsilon_1,\dots,\epsilon_m)$
respectively (cf. \cite{AK}).
Particularly $\Li^{\scalebox{.5}{$\Sha$}}_{n_1,\dots,n_m}(\epsilon_1,\dots,\epsilon_m)=\Li^{\ast}_{n_1,\dots,n_m}(\epsilon_1,\dots,\epsilon_m)=
\Li_{n_1,\dots,n_m}(\epsilon_1,\dots,\epsilon_m)$
when $(n_m,\epsilon_m)\neq (1,1)$.
In \cite[(37)]{E-ARIGARI} and \cite[(1.27)]{E-flex}, 
it is explained  that they are related as follows
\begin{equation}\label{eq: Ecalle regularization relations}
\Zig=\Mini\times \swap(\Zag)
\end{equation}
(see \eqref{eq:swap} for $\swap$).
Here 
$\Mini=\{\Mini\varia{u_1,\ \dots,\ u_m}{\epsilon_1,\ \dots,\ \epsilon_m}\}_m$
is the mould defined by
$$
\Mini\varia{\epsilon_1,\ \dots,\ \epsilon_m}{v_1,\ \dots,\ v_m}=
\begin{cases}
\Mono_m & \text{when} \quad (\epsilon_1,\ \dots,\ \epsilon_m)=(1,\dots, 1),\\
0 & \text{otherwise}  \\
\end{cases}
$$
with $1+\sum_{r=2}^\infty\Mono_r\cdot t^r:=\exp\{\sum_{k=2}^\infty (-1)^{k-1}\frac{\zeta(k)}{k}t^{k}\}$ (cf. \cite[(1.30)]{E-flex}).

Goncharov's arguments in \cite{G-MPL} suggest us  to express them as
\begin{align*}
&\Zag\varia{u_1,\ \dots,\ u_m}{\epsilon_1,\ \dots,\ \epsilon_m}=
\reg^{\scalebox{.5}{$\Sha$}}\int_{0<s_1<\cdots< s_m<1}
\frac{\epsilon_1s_1^{-u_1}}{1-\epsilon_1s_1}ds_1\wedge\cdots\wedge
\frac{\epsilon_ms_m^{-u_m}}{1-\epsilon_ms_m}ds_m, \\
&\Zig\varia{\epsilon_1,\ \dots,\ \epsilon_m}{v_1,\ \dots,\ v_m}=
\reg^\ast\sum_{0<k_m<\cdots<k_1}
\frac{\epsilon_m^{k_m}\cdots\epsilon_1^{k_1}}
{(k_m-v_m)\cdots (k_1-v_1)}
\end{align*}
by using his regularizations.
These might also helpful to understand \cite[(9.5)]{E-flex} saying that
$\Zag$ belongs to $\GARI(\Gamma)_{\as/\is}$.
The authors are not aware of its definition but  expect that it means 
the symmetrality for $\Zag$ and the symmetrilty for $\Zig$ 
(\cite[\S\S 1.1--1.2]{E-flex}),
which looks corresponding to the shuffle and the harmonic product
among multiple $L$-values.
In \cite[\S 4.7]{E-flex} it is directed to combine
a related group $\GARI(\Gamma)_{\underline{\as}/\underline{\is}}$
and its Lie algebra 
$\ARI(\Gamma)_{\underline{\al}/\underline{\il}}$
with a bigraded variant $\GARI(\Gamma)_{\underline{\as}/\underline{\as}}$
and $\ARI(\Gamma)_{\underline{\al}/\underline{\al}}$
(cf. Definition \ref{def:ARIalal})
under maps $\mathrm{adgari(pal)}$ and $\mathrm{adari(pal)}$ 
(cf. \cite[(2.54), (2.55), and \S 4.2]{E-flex}):
\begin{equation}\label{eq:Ecalle CD}
\xymatrix{ 
\GARI(\Gamma)_{\underline{\as}/\underline{\as}}\ar[r]^{\mathrm{adgari(pal)}}\ar[d]^{\mathrm{logari}}& \GARI(\Gamma)_{\underline{\as}/\underline{\is}}\ar[d]^{\mathrm{logari}}\\ 
\ARI(\Gamma)_{\underline{\al}/\underline{\al}}\ar[r]^{\mathrm{adari(pal)}}&\ARI(\Gamma)_{\underline{\al}/\underline{\il}}.
}
\end{equation}

While the cyclotomic analogue of Drinfeld's KZ-associator (\cite{Dr}) 
which is constructed  from the KZ-like differential equation in $\hat{\mathbb A}_\C$
over 
$\mathbb{P}^1(\C)\setminus\{0,\mu_N,\infty\}$
with a complex variable $s$
$$
\frac{d}{ds}H(s)=\left(\frac{x}{s}+\sum_{\xi\in\mu_N}\frac{y_\xi}{\xi s-1}\right)
H(s)
$$
is introduced and discussed in \cite{R, En}.
It is a non-commutative formal power series
$\Phi_{\KZ}^N\in\hat{\mathbb A}_\C$
with $\Gamma=\mu_N$
whose coefficients are
given by multiple $L$-values.
In particular, the coefficient of 
$$x^{n_r-1}y_{\epsilon_r}
x^{n_{r-1}-1}y_{\epsilon_{r-1}\epsilon_{r}}\cdots
x^{n_1-1}y_{\epsilon_{1}\cdots\epsilon_{r}}$$
is $(-1)^r\Li^{\scalebox{.5}{$\Sha$}}_{n_1,\dots,n_r}(\epsilon_1,\dots,\epsilon_r)$.
Definition \ref{def:ma} enables us to calculate
a relation between the associator ${\Phi}_{\KZ}^N$ and the mould $\Zag$
as follows
\begin{align}\label{eq: ma=Zag}
\ma_{({\Phi}_{\KZ}^N)^{-1}}=\left\{
\Zag\varia{-u_1, \dots, -u_m}{\epsilon_1^{-1},\ \dots \ ,\ \epsilon_m^{-1}}
\right\}_m
\in\mathcal M(\mathcal F;\mu_N). 
\end{align}
Put $\widetilde\Phi_{\KZ}^N=\Phi_{\KZ}^N(x,(-y_\sigma))$ 
and
$\widetilde\Phi_{\KZ,\corr}=\exp\left\{\sum\limits_{n\geqslant 2}
\frac{(-1)^{n-1}}{n}
\Li_n(1)y_1^n\right\}
\in \hat{\mathbb A}_\C$
and define
\begin{equation}\label{eq: Racinet regularization relations}
\widetilde\Phi_{\KZ,\ast}^N:=\widetilde\Phi_{\KZ, \corr}\cdot
\q(\pi_Y(\widetilde\Phi_{\KZ}^N)) 
\end{equation}
(for $\pi_Y$ and $\q$, see \S \ref{subsec:KV condition}).
It is shown in \cite{R} that $\widetilde\Phi_{\KZ}^N$ is group-like
with respect to the shuffle (deconcatenation) coproduct of $\hat{\mathbb A}_\C$
and $\widetilde\Phi_{\KZ,\ast}^N$  is so
with respect to the harmonic coproduct of $\mathrm{Im}\,\pi_Y$.
They correspond to the shuffle and the harmonic product
among multiple $L$-values respectively.
The regularized double shuffle relations (the shuffle and the harmonic products
and  the regularization relations \eqref{eq: Racinet regularization relations})
are the defining equations of his torsor 
$\mathrm{DMR}_\mu$ 
\footnote{The word DMR
stands for the French `double m\'elange et r\'egularisation'.}
for $\mu\in\C^\times$
which contains $\widetilde\Phi_{\KZ}^N$ as a specific point when $\mu=2\pi\sqrt{-1}$.
It is equipped with a free and transitive action of 
the group $\mathrm{DMR}_0$. 
Its associated Lie algebra, 
which he denotes by $\mathfrak{dmr}_0$,
is a filtered graded Lie algebra
defined by the regularized double shuffle relations  modulo products.
Our dihedral Lie algebra $\mathbb D(\mu_N)_{\bullet\bullet}$ should be called
its bigraded variant,
defined by  \lq their highest depth parts' of the relations.
By definition,  it contains 
the associated graded $\gr\mathfrak{dmr}_0$.
By translating Ecalle's pictures including the diagram \eqref{eq:Ecalle CD}
to this setting,
we might learn
more enriched structures on these Lie algebras. 



\begin{thebibliography}{}
\bibitem[AET]{AET}
Alekseev, A., Enriquez, B. and Torossian, C.,
{\it Drinfeld associators, braid groups and explicit solutions of the Kashiwara-Vergne equations},
Publ. Math. Inst. Hautes \'{E}tudes Sci. No. {\bf 112} (2010), 143--189. 

\bibitem[AKKN]{AKKN}
Alekseev, A., Kawazumi. N, Kuno. Y. and  Naef, F.
{\it The Goldman-Turaev Lie bialgebra in genus zero and the Kashiwara-Vergne problem}, Advances in Mathematics, {\bf 326}, 1-53 (2018).

\bibitem[AT]{AT}
Alekseev, A. and Torossian, C.,
{\it The Kashiwara-Vergne conjecture and Drinfeld's associators},
Ann. of Math. (2) {\bf 175} (2012), no. 2, 415--463. 

\bibitem[AK]{AK}
Arakawa, T., Kaneko, M., 
{\it On multiple $L$-values}, 
J. Math. Soc. Japan 56 (2004), no.4, 967--991. 


\bibitem[C]{Cre}
Cresson. J.,
{\it Calcul moulien},
Ann. Fac. Sci. Toulouse Math. (6)  {\bf 18}  (2009),  no. 2, 307--395. 

\bibitem[D]{Dr} V. Drinfeld,
{\it On quasitriangular quasi-Hopf algebras and on a group that is closely connected with 
$\mathrm{Gal}(\overline{\mathbb Q}/\mathbb Q)$},
Leningrad Math. J. 2 (1991), no. 4, 829--860.  

\bibitem[Ec81]{E81}
Ecalle, J.,
{\it Les fonctions r\'{e}surgentes. Tome I et II},
Publications Math\'{e}matiques d'Orsay {\bf 81},
6. Universit\'{e} de Paris-Sud, D\'{e}partement de Math\'{e}matique, Orsay, 1981. 
 
\bibitem[Ec03]{E-ARIGARI}
Ecalle, J., 
{\it ARI/GARI, la dimorphie et l'arithm\'{e}tique des multiz\^{e}tas: un premier bilan},
J. Th\'{e}or. Nombres Bordeaux {\bf 15} (2003), no. 2, 411--478. 

\bibitem[Ec11]{E-flex}
Ecalle, J.,
{\it The flexion structure and dimorphy: flexion units, singulators, generators, and the enumeration of multizeta irreducibles}, 
With computational assistance from S. Carr. CRM Series, {\bf 12}, Asymptotics in dynamics, geometry and PDEs; generalized Borel summation. Vol. II, 27--211, Ed. Norm., Pisa, 2011.

\bibitem[En]{En}
Enriquez, B.,
{\it Quasi-reflection algebras and cyclotomic associators},
Selecta Math.(N.S.) 13 (2007), no. 3, 391--463.

\bibitem[F]{F}
Furusho, H.,
{\it Around associators}, Automorphic forms and Galois representations. Vol. 2, 105--117, London Math. Soc. Lecture Note Ser., {\bf 415}, Cambridge Univ. Press, Cambridge, 2014.

\bibitem[G01a]{G}
Goncharov, A. B., 
{\it The dihedral Lie algebras and Galois symmetries of $\pi_1^{(l)}(\mathbb{P}^1-(\{0,\infty\}\cup \mu_N))$},
Duke Math. J. {\bf 110} (2001), no. 3, 397--487. 

\bibitem[G01b]{G-MPL}
Goncharov, A. B., 
{\it Multiple polylogarithms and mixed Tate motives},
preprint, {\tt arXiv:math/0103059}.

\bibitem[M]{M}
Maassarani, M.,
{\it Bigraded Lie algebras related to MZVs},
preprint, {\tt arXiv:1907.07200}.

\bibitem[R00]{R-PhD} 
Racinet, G.,
{\it S\'{e}ries g\'{e}n\'{e}ratrices non-commutatives de polyz\^{e}tas et associateurs de Drinfeld},
Ph.D. desseratation, Paris, France, 2000.

\bibitem[R02]{R} 
Racinet, G.,
{\it Doubles m\'elanges des polylogarithmes multiples aux racines de l'unit\'e}, 
Publ. Math. Inst. Hautes \'Etudes Sci. {\bf 95} (2002), 185--231. 

\bibitem[RS]{RS}
Raphael, E., Schneps, L.,
{\it On linearised and elliptic versions of the Kashiwara-Vergne Lie algebra},
{\tt arXiv:1706.08299v1}, preprint.

\bibitem[SaSch]{SS}
Salerno, A., Schnepps, L.,
{\it Mould theory and the double shuffle Lie algebra structure}, 
Periods in Quantum Field Theory and Arithmetic,
Springer Proceedings in Mathematics \& Statistics, vol {\bf 314} (2020), Springer, 399--430. 

\bibitem[Sau]{Sau}
Sauzin, D.,
{\it Mould expansions for the saddle-node and resurgence monomials},
Renormalization and Galois theories, 
83--163, IRMA Lect. Math. Theor. Phys., {\bf 15}, Eur. Math. Soc., Z\"{u}rich, 2009. 

\bibitem[Sch12]{S}
Schneps, L.,
{\it Double shuffle and Kashiwara-Vergne Lie algebras},
J. Algebra {\bf 367} (2012), 54--74.

\bibitem[Sch15]{S-ARIGARI}
Schneps, L.,
{\it ARI, GARI, ZIG and ZAG: An introduction to Ecalle's theory of multiple zeta values},
{\tt arXiv:1507.01534}, preprint.


\end{thebibliography}
\end{document}